\documentclass{amsart}
\usepackage{amsmath,amssymb}
\makeatletter
\newenvironment{proofwithmanualqed}[1][\proofname]{\par
  \normalfont \topsep6\p@\@plus6\p@\relax
  \trivlist
  \item[\hskip\labelsep\itshape #1\@addpunct{.}]\ignorespaces
}{\endtrivlist\@endpefalse}
\makeatother
\usepackage{graphicx}
\usepackage{hyperref}
\usepackage{enumerate}
\usepackage{needspace}

\numberwithin{equation}{section}

\renewcommand{\emptyset}{\varnothing}

\newcommand\cv{``$\circ$''}

\renewcommand{\epsilon}{\varepsilon}
\newcommand{\Vol}{\operatorname{Vol}}
\newcommand{\carea}{c_{\mathit area}}

\newcommand{\cB}{{\mathcal B}}
\newcommand{\cC}{{\mathcal C}}
\newcommand{\cH}{{\mathcal H}}
\newcommand{\cQ}{{\mathcal Q}}
\newcommand{\calQ}{{\mathcal Q}}

\newcommand{\GG}{{\mathbb G}}
\newcommand{\C}{{\mathbb C}}

\newcommand{\N}{{\mathbb N}}

\newcommand{\CP}{{\mathbb C}\!\operatorname{P}^1}

\newlength{\halfbls}

\newtheorem{theorem}{Theorem}[section]
\newtheorem*{NNTheorem}{Theorem}

\newtheorem{proposition}[theorem]{Proposition}
\newtheorem{conjecture}[theorem]{Conjecture}
\newtheorem{problem}[theorem]{Problem}
\newtheorem{lemma}[theorem]{Lemma}
\newtheorem{corollary}[theorem]{Corollary}

\theoremstyle{definition}
\newtheorem{definition}[theorem]{Definition}
\newtheorem{remark}[theorem]{Remark}

\newtheorem{example}[theorem]{Example}

\begin{document}

\title[Asymptotics of Masur--Veech volume]
{Asymptotics of Masur--Veech volume of strata
of quadratic differentials with many simple poles}

% First author
\author{Kasra Rafi}
\address{
Department of Mathematics,
University of Toronto,
40 St. George Street,
Toronto, ON M5S 2E4,
Canada
}
\email{rafi@math.utoronto.ca}

% Second author
\author[A.~Zorich]{Anton Zorich}
\address{
Institut de Math\'ematiques de Jussieu --
Paris Rive Gauche,
Case 7012,
8 Place Aur\'elie Nemours,
75205 PARIS Cedex 13, France}
\email{anton.zorich@gmail.com}

\date{October 20, 2026}

\begin{abstract}
We provide large $n$ asymptotics of the Masur--Veech volumes for strata of quadratic differentials of the form $\cQ_g(n,d_1,\dots,d_k,-1^{n+D-4g+4})$, where all positive integers $d_i$ are fixed and $D=d_1+\dots+d_k$. Our formulae are particularly explicit in genus $1$. For the two series of non-varying strata $\cQ_1(n,-1^n)$ and $\cQ_1(n,1,-1^{n+1})$ we give simple closed formulae for the volumes for any $n$.
\end{abstract}

\maketitle

\tableofcontents

%%%%%%%%%%%%%%%%%%%%%%%%%%%%%%%%%%%%%%%%%%%%%%
\section{Introduction}
%--------------------------------------------------------------
\subsection{State of the art}

An ample knowledge of Masur--Veech volumes and of Siegel--Veech constants of strata of \textit{Abelian} differentials was achieved in recent years due to breakthrough advances by \mbox{A.~Aggarwal} and by D.~Chen, M.~Möller, A.~Sauvaget, D.~Zagier. These results allow one to compute efficiently exact values of volumes in small genera and also provide uniform large genus asymptotics for all connected components of all strata of Abelian differentials conjectured in~\cite{Eskin:Zorich} and proved in~\cite{Aggarwal:SVconstants}, \cite{Aggarwal:Abelian} and in a series of papers~\cite{Chen:Moeller:Zagier}, \cite{Chen:Moeller:Sauvaget:Zagier}, \cite{Sauvaget}.

Our current knowledge of Masur--Veech volumes and of Siegel--Veech constants of strata of \textit{meromorphic} quadratic differentials with at most simple poles is very limited compared to what is known for strata of Abelian differentials. It is, basically, restricted to the following three cases.

We have a closed formula for the volume of any stratum in genus $0$ conjectured by Kontsevich and proved  in~\cite{Athreya:Eskin:Zorich}.

We also know Masur--Veech volumes and Siegel--Veech constants for all low-dimensional strata. Namely, exact values of volumes of all strata of dimension at most $11$ are computed in~\cite[arXiv:1501.01611v4, Appendix~A]{Goujard:vol} using an algorithm of Eskin--Okounkov~\cite{Eskin:Okounkov}. Tables of Siegel--Veech constants for many low-dimensional strata and an algorithm which allows one to compute them are elaborated in~\cite{Goujard:carea}.

Finally, a series of papers~\cite{Delecroix:Goujard:Zograf:Zorich}, \cite{Chen:Moeller:Sauvaget}, \cite{Aggarwal:quadratic}, \cite{Kazarian}, \cite{DiYang:Zagier:Zhang} allowed one to compute Masur--Veech volumes for the principal strata of meromorphic quadratic differentials $\cQ_g(1^{4g-4+k},-1^k)$. For reasonably small $g$ (say, $g\le 2000$), the efficient recursions allow one to compute the volumes exactly. The regimes when $k$ is fixed and $g\to\infty$ and when $g$ is fixed and $k\to\infty$ are described by asymptotic formulae.

Despite promising further ongoing research on the Masur--Veech volumes of other strata of quadratic differentials, as in~\cite{Duryev:Goujard:Yakovlev} and in~\cite{Moeller:Prado}, the conjecture~\cite{Aggarwal:Delecroix:Goujard:Zograf:Zorich} on uniform large genus asymptotics of Masur--Veech volumes of all strata of holomorphic quadratic differentials is wide open. For the regime when the genus is fixed and the number of poles tends to infinity studied in the current paper, there was not even a conjectural answer with the exception of the principal stratum for which the answer was found in~\cite{Chen:Moeller:Sauvaget} and strengthened in~\cite{DiYang:Zagier:Zhang}.

In this paper we describe asymptotic values of Masur--Veech volumes and of the Siegel--Veech constants for a new large family of strata. Namely, we consider families of strata in a fixed genus having any collection of zeroes of fixed degrees and a zero of degree $n$ which tends to infinity. The remaining singularities are simple poles. For two series of \textit{non-varying} strata $\cQ_1(n,-1^n)$ and $\cQ_1(n,1,-1^{n+1})$ we obtain exact closed formulae for any $n$. For other strata of the form $\cQ_g(n,d_1,\dots,d_k,-1^{n-4g+4+D})$, where $D=d_1+\dots+d_k$, we obtain asymptotic formulae as $n\to+\infty$.

In some aspects, our approach is similar to that of~\cite{Athreya:Eskin:Zorich}: we start with a recursion based on the Eskin--Kontsevich--Zorich formula~\cite{Eskin:Kontsevich:Zorich} for the sum of Lyapunov exponents. In the current paper we combine it with Goujard’s formula~\cite{Goujard:carea} for the area Siegel--Veech constants, with Fougeron’s bound~\cite{Foujeron} for the Lyapunov exponents, and with the Athreya--Eskin--Zorich formula~\cite{Athreya:Eskin:Zorich} for volumes of strata in genus $0$ to derive a contracting recurrence for the normalized volumes. The major difference is that now we do not know either exact values of the Lyapunov exponents, or exact values of $\Vol\cQ_g(n,d_1,\dots,d_k,-1^{n+D-4g+4})$, for small $n$ to initiate the recurrence. Despite the lack of this important information, we manage to prove a simple asymptotic formula.

In some other aspects our recursion for the Masur--Veech volumes
of strata of quadratic differentials with many poles resembles Mirzakhani's recursion for Weil--Petersson volumes~\cite{Mirzakhani:volumes} and  the Dijkgraaf--Verlinde--Verlinde recursion for the intersection numbers~\cite{Verlinde:Verlinde}, \cite{Dijkgraaf},
and its asymptotic behavior studied by A.~Aggarwal~\cite{Aggarwal:quadratic}.

Our interest in volume asymptotics is motivated, in particular,
by the study of random flat surfaces of large complexity.

%--------------------------------------------------------------
\subsection{Statements of results}

Following~\cite{Athreya:Eskin:Zorich}, define the
function $v$ on integers $n$ greater than or equal to $-1$ by
\begin{equation}
\label{eq:v}
v(n):=\pi^{n+\frac{1}{2}}\cdot
\frac{\Gamma\!\left(\frac{n}{2}+1\right)}
     {\Gamma\!\left(\frac{n}{2}+\frac{3}{2}\right)}
=\cfrac{n!!}{(n+1)!!} \cdot \pi^n \cdot
\begin{cases}
\pi& \text{ when $n$ is odd}\\
2  & \text{ when $n$ is even}
\end{cases}
\end{equation}
for     $n=-1,0,1,2,3\dots$, where the     double    factorial
$n!!=n\cdot(n-2)\cdot\dots$  is the product of all
positive integers less than or equal to $n$ with the same parity as $n$. By convention we
set
$$
(-1)!! = 0!!=1 \ ,
$$
which implies that
\begin{equation}
\label{eq:v:for:small:n}
v(-1)=1\,,\quad
v(0)=2\,,\quad
v(1)=\frac{1}{2}\pi^2\,,\quad
v(2)=\frac{4}{3}\,\pi^2\,. % \quad
\end{equation}

\begin{theorem}
\label{th:n:1}
For any integer $n\ge 1$ the following relations are valid
\begin{align}
\label{eq:volume:n}
\Vol\cQ_1(n,-1^n)&=\frac{(n+2)(n+1)(n-1)}{3n^2}\cdot v(n)\,,\\
\label{eq:volume:n:1}
\Vol\cQ_1(n,1,-1^{n+1})&=\frac{(n+1)^2}{3n}\cdot v(n)\cdot v(1)\,.
\end{align}
\end{theorem}
Note that according to this formula we have $\Vol\cQ_1(1,-1)=0$.
There is no contradiction: the stratum $\cQ_1(1,-1)$ is empty.
Formula~\eqref{eq:volume:n} is proved in Section~\ref{ss:n:g:1};
Formula~\eqref{eq:volume:n:1} is proved in Section~\ref{ss:n:1:g:1}.

For all families of strata in genus $1$ for which all degrees
of zeroes except one are fixed we have the following asymptotics.
\begin{theorem}
\label{th:n:m1:dots:mk:lim}
For any nonnegative integers
$d_1,\dots, d_k$ one has
\begin{multline}
\label{eq:n:m1:dots:mk:lim}
\Vol\cQ_1(n,d_1,d_2,\dots,d_k,-1^{n+D})
\\=
\frac{n}{3}\cdot v(n)\cdot v(d_1)\cdots v(d_k)
\cdot\left(1+O_{1,d}\left(\frac{1}{n}\right)\right)
\\
=\frac{\sqrt{2n}}{3}\,\pi^{n+D+\frac{k+1}{2}}
\prod_{i=1}^{k}
\frac{\Gamma\!\left(\frac{d_i}{2}+1\right)}
     {\Gamma\!\left(\frac{d_i}{2}+\frac{3}{2}\right)}
\left(1+O_{1,d}\!\left(\frac1n\right)\right)
% =\frac{\sqrt{2}}{3}\cdot\pi^{n+\frac{1}{2}}\cdot n^{\frac{1}{2}}
% \cdot v(d_1)\cdots v(d_k)
% \cdot\left(1+O_{1,d}\left(\frac{1}{n}\right)\right)}
\quad\text{as}\quad n\to+\infty\,,
\end{multline}
where $D=d_1+\dots+d_k$.
\end{theorem}
\begin{proof}
Explicit expressions~\eqref{eq:volume:n} and~\eqref{eq:volume:n:1} from Theorem~\ref{th:n:1} imply~\eqref{eq:n:m1:dots:mk:lim}  for $d=\emptyset$ and $d=(1)$, respectively. Theorem~\ref{th:n:m1:dots:mk:lim} now follows from the general Theorem~\ref{th:n:D:large:zero:any:g} under the choice $C_1^{\mathit{even}}=C_1^{\mathit{odd}}=\frac{1}{3}$.
\end{proof}

\begin{theorem}
\label{th:n:D:large:zero:any:g}
For any genus $g\ge 0$ there exist strictly positive constants $C^{\mathit{even}}_g$
and $C^{\mathit{odd}}_g$ such that
for any collection
$(d_1,\dots, d_k)$ of nonnegative integers, one has
\begin{multline}
\label{eq:n:m1:dots:mk:lim:genus:g}
\Vol\cQ_g(n,d_1,d_2,\dots,d_k,-1^{n+D-4g+4})
\\=
C^{p(n)}_g\cdot n^g \cdot v(n)\cdot v(d_1)\cdots v(d_k)
\cdot\left(1+O_{g,d}\left(\frac{1}{n}\right)\right)
\\
=\sqrt{2}\,C^{p(n)}_g\,n^{g-\frac{1}{2}}\,\pi^{n+D+\frac{k+1}{2}}
\prod_{i=1}^{k}
\frac{\Gamma\!\left(\frac{d_i}{2}+1\right)}
     {\Gamma\!\left(\frac{d_i}{2}+\frac{3}{2}\right)}
\left(1+O_{g,d}\!\left(\frac1n\right)\right)
\quad\text{as}\quad n\to+\infty\,,
\end{multline}
where $p(n)$ is \textit{even} or \textit{odd} depending on whether
$n$ is even or odd.
\end{theorem}
For positive genus and positive fixed degrees, Theorem~\ref{th:n:D:large:zero:any:g}
follows from Theorem~\ref{th:n:m1:dots:mk:lim:for:renormalized}
and Definition~\eqref{eq:W:g:n}.
The latter theorem is proved in Sections~\ref{ss:Minimal:stratum:in:arbitrary:genus} and~\ref{ss:Generic:stratum:in:arbitrary:genus}.
The genus-zero case follows from~\eqref{eq:volume:in:genus:0}.
Stirling's formula applied to the Gamma-function expression
in~\eqref{eq:v} gives the last line of~\eqref{eq:n:m1:dots:mk:lim:genus:g}.

\begin{conjecture}
%\label{conj:even:equals:odd}
For any $g\ge 2$ one has $C^{\mathit{even}}_g = C^{\mathit{odd}}_g$.
\end{conjecture}

We state the conjecture only for $g\ge 2$ because
for $g=0$ it is proved in~\cite{Athreya:Eskin:Zorich}
and for $g=1$ it is proved in Theorem~\ref{th:n:1}. Namely
\begin{align*}
C^{\mathit{even}}_0 = C^{\mathit{odd}}_0 &= 2\pi^2\,,
\\
C^{\mathit{even}}_1 = C^{\mathit{odd}}_1 &= \frac{1}{3}\,.
\end{align*}

Note that by~\cite[Formula (8)]{Chen:Moeller:Sauvaget} one has
for $n\ge 2$:
\[
\Vol\cQ_1(1^n,-1^n)=\pi^{2n}\frac{n!}{3(2n-1)!}
\big((2n-3)!!+(2n-2)!!\big)\,,
\]
so
\begin{equation}
\label{eq:1:1:1:lim}
\lim_{n\to+\infty}
\frac{\Vol\cQ_1(1^n,-1^n)}
{\frac{\sqrt{\pi\,n}}{3}\cdot \big(v(1)\big)^n}
=1\,.
\end{equation}
More generally, \cite[Formula (7)]{Chen:Moeller:Sauvaget}
implies that
\begin{equation}
\label{eq:1:1:1:lim:genus:g}
\Vol\cQ_g(1^{4g-4+n},-1^n)\sim
\kappa_g n^{\frac{g}{2}} \big(v(1)\big)^n\,,
\end{equation}
for any fixed $g$ and for $n\to+\infty$,
see~\cite[(55)--(57)]{DiYang:Zagier:Zhang} for the value of $\kappa_g$
as a function of the genus.
This shows that the asymptotic Formulae~\eqref{eq:n:m1:dots:mk:lim}
and~\eqref{eq:n:m1:dots:mk:lim:genus:g}
are not applicable uniformly for all partitions.

\begin{problem}
Find a general asymptotic formula
interpolating between the answers as in~\eqref{eq:n:m1:dots:mk:lim} and~\eqref{eq:n:m1:dots:mk:lim:genus:g} for particularly short partitions $(N+4g-4-D,d_1,\dots,d_k)$ and the answers~\eqref{eq:1:1:1:lim}
and~\eqref{eq:1:1:1:lim:genus:g} for the longest possible partition $\underbrace{(1,1,\dots,1)}_{N+4g-4}$ of $N+4g-4$,
where $g$ is the genus and $N\to+\infty$ is the number of simple poles.
\end{problem}

%--------------------------------------------------------------------
\subsection{Acknowledgements}
The authors cordially thank Elise Goujard for the time which she spent on explaining the details of her work~\cite{Goujard:vol} to them and for the precious help with debugging exact formulae for $\carea(\calQ(n,k,-1^{n+k}))$ for small values of $k$ including Formulae~\eqref{eq:c:area:minimal} and~\eqref{eq:n:1:init}. The authors are grateful to Sonja~\v{C}uki\`c for the valuable help with graphics and pictures. Various parts of the project were performed during our stays at CIRM in Luminy, IMJ-PRG in Paris, MPIM in Bonn, Fields Institute in Toronto and at the Yau Center for Mathematics in Beijing. We are grateful to all these places for hospitality. Kasra Rafi is supported by NSERC Discovery grant RGPIN-05507.
Anton Zorich is supported by the ERC AdG 101141508 UniGeoDyM, ANR-23-CE40-0020-02 MOST grants and by IUF. We used ChatGPT to help with checking and streamlining of our calculations and for proofreading.

%#####################################################################
%#####################################################################
%#####################################################################

\section{Background results and recursion structure}

%--------------------------------------------------------------------
\subsection{Background results}

Our induction is based on a combination of the following results. The first Theorem corresponds to~\cite[Formula~(2.3)]{Eskin:Kontsevich:Zorich}.

\begin{NNTheorem}[A.~Eskin, M.~Kontsevich, A.~Zorich]
The top $g$ Lyapunov exponents $\lambda_1^+\ge\dots\ge\lambda_g^+$ of the bundle $H^+$ over any connected component $\cQ^c_g(m_1,\dots,m_p)$ of any stratum of meromorphic quadratic differentials with at most simple poles in genus $g$ along the Teichmüller geodesic flow satisfy the following relation:
\begin{equation}
\label{eq:sum:of:lambda:plus}
\lambda_1^+ + \dots +\lambda_g^+=
\frac{1}{24}\sum_{i=1}^p \frac{m_i(m_i+4)}{m_i+2}
+\frac{\pi^2}{3}\cdot c_{area}(\cQ^c_g(m_1,\dots,m_p))\,,
\end{equation}
where $c_{area}(\cQ^c_g(m_1,\dots,m_p))$ is the area Siegel--Veech
constant.
\end{NNTheorem}
Note that we usually separate simple poles from simple zeroes, whereas, in the notation $\cQ_g(m_1,\dots,m_p)$ they are listed together. In particular, unlike zeroes, simple poles make negative contributions  $\frac{(-1)(-1+4)}{(-1+2)}=-3$ to the above sum.

The left-hand side of equality~\eqref{eq:sum:of:lambda:plus}   is known explicitly for the two \textit{non-varying} strata $\cQ_1(n,-1^n)$ and $\cQ_1(n,1,-1^{n+1})$ in genus one. These are the two strata for which we get a closed formula for the volume building on the following result of D.~Chen and M.~Möller~\cite[Theorem~8.1]{Chen:Moeller}:
\begin{NNTheorem}[D.~Chen, M.~Möller]
The Lyapunov exponent $\lambda_1^+$ of the bundle $H^+$ over the strata $\cQ_1(n,-1^n)$ and $\cQ_1(n,1,-1^{n+1})$ along the Teichmüller geodesic flow satisfies the following relation:
\begin{align}
\label{eq:lambda1:plus:n}
\lambda_1^+\big(\cQ_1(n,-1^n)\big)
=\frac{2}{n+2}\,,\quad{for}\ n\ge 2\,,
\\
\label{eq:lambda1:plus:n:1}
\lambda_1^+\big(\cQ_1(n,1,-1^{n+1})\big)
=\frac{2}{n+2}\,,\quad{for}\ n\ge 1\,.
\end{align}
\end{NNTheorem}

For other strata of the form $\cQ_g(n,d_1,d_2,\dots,d_k,-1^{n+D-4g+4})$ the left-hand side of the equality~\eqref{eq:sum:of:lambda:plus} admits upper and lower bounds which are sufficient to elaborate the volume asymptotics. All Lyapunov exponents involved in~\eqref{eq:sum:of:lambda:plus} are by definition nonnegative, so $\lambda_1^+\ge\dots\ge\lambda_g^+\ge 0$ which provides a trivial lower bound $0$ for the sum in the left-hand side of the equality~\eqref{eq:sum:of:lambda:plus}. The upper bound for this sum is given by the following Theorem of Ch.~Fougeron~\cite{Foujeron}.

\begin{NNTheorem}[Ch.~Fougeron~\cite{Foujeron}]
Let $g$ and $n\ge d_1 \ge d_2 \ge \cdots\ge d_k$ be positive integers. Let $D=d_1+\dots+d_k$. If both of the following conditions hold, namely if $n+D\ge 4g-4$ and $n\ge 2g$, then for any connected component $Q^c(n,d_1,\dots, d_k,-1^{n+D-4g+4})$ of the stratum of meromorphic quadratic differentials in genus $g$, the following inequality holds:
\begin{equation}
\label{eq:Fougeron:bound}
\lambda_1^+ + \dots +\lambda_g^+
\le \frac{(3g-1)g}{n+2}\,.
\end{equation}
\end{NNTheorem}

Consider now the right-hand side of the key equality~\eqref{eq:sum:of:lambda:plus}. The area Siegel--Veech constant is given by~\cite[Formula~(19)]{Goujard:carea}. It is computed as a sum over ``\textit{configurations of $\hat{h}$\hspace*{-1pt}omologous saddle connections}'' (or, equivalently, over admissible decorated ribbon graphs, or, equivalently, over the \textit{principal boundary classes} $\cC$).
\begin{NNTheorem}[E.~Goujard]
Let $\cC$ be an admissible configuration for a connected stratum $\cQ(\alpha)$ of quadratic differentials.
Let $q_1$ denote the number of thin cylinders, $q_2$ the number
of thick cylinders
for the configuration $\cC$.
Assume that the associated principal boundary stratum $\cQ(\alpha'(\cC))$ is nonempty, and that $q=q_1+q_2\geq 1$. Then the contribution of the configuration $\cC$ to the area Siegel--Veech constant has the following form:
\begin{multline}
\label{eq:carea}
c_{area}(\cC)=M\cdot\frac{(4q_1+q_2)}{2^{m+q+3}}
\\
\times\frac{
\prod_i\big((a_i-1)!\Vol\cQ(\alpha'_i)\big)\cdot
\prod_j\big(2^{b_j}\cdot (b_j-1)! \cdot \Vol\cH(\beta'_j)\big)
}
{(\dim_{\C} \cQ(\alpha)-1)!\Vol\cQ(\alpha)}\,,
\end{multline}
where
$M$ is an explicit combinatorial constant,
$a_i=\dim_{\C} \cQ(\alpha'_i)$, $b_j=\dim_{\C} \cH(\beta'_j)$,
and $m$ is the total number of irreducible components
$\cQ(\alpha'_i)$ and $\cH(\beta'_j)$ counted together.
\end{NNTheorem}

We denote singularity data of quadratic strata by symbols $\alpha'_i$ and their dimensions by $a_i$. We denote singularity data of Abelian strata by $\beta'_j$ and their dimensions by $b_j$. We warn the reader that E.~Goujard uses reverse notations in~\cite[Formula~(19)]{Goujard:carea}. In the current paper the symbol $b$ is mostly reserved for degrees of  endpoint singularities. We refer to the original paper~\cite{Goujard:carea} for area conventions and for a convention on normalization of the Masur--Veech volume. Recall that all the singularities are always assumed to be labeled.

\begin{remark}
\label{rm:connected:strata}
In genera $0$ and $1$ all strata are connected. It follows from the classification Theorem by E.~Lanneau~\cite{Lanneau} rectified by D.~Chen and M.~Möller~\cite{Chen:Moeller} that all strata having at least $3$ simple poles are connected. Dealing with strata of quadratic differentials in genera larger than one we will mostly assume that the stratum has at least $5$ simple poles. We shall always work with connected strata only.
\end{remark}

The remaining important background result is the following closed formula for the Masur--Veech volume of any stratum in genus $0$ conjectured by M.~Kontsevich and proved in~\cite[Theorem~1.1]{Athreya:Eskin:Zorich}.
\begin{NNTheorem}[M.~Kontsevich -- conjecture;
J.~Athreya, A.~Eskin, A.~Zorich -- proof]
The volume of any stratum $\cQ_0(m_1, \dots, m_p)$ of meromorphic quadratic differentials with at most simple poles on $\CP$ (i.e. $m_i\in \{-1\,;\,0\}\cup\N$ for $i=1,\dots,p$, and $\sum_{i=1}^p m_i = -4$) is equal to
\begin{equation}
\label{eq:volume:in:genus:0}
\Vol\cQ_0(m_1, \dots, m_p) =  2\pi^2\cdot \prod_{i=1}^p v(m_i)\,,
\end{equation}
where all the zeroes and poles are labeled.
\end{NNTheorem}
In particular, since by convention $v(-1)=1$,
Formula~\eqref{eq:volume:in:genus:0} gives
\begin{equation}
\label{eq:Vol:minimal:g0}
  \Vol\cQ_0(n,-1^{n+4}) = 2\pi^2\cdot v(n)\,.
\end{equation}

\begin{remark}
%\label{rm:contribution:of:a:marked:point:to:volume}
Standard normalization conventions on the Masur--Veech volume element
imply that for any stratum of Abelian differentials in any genus,
marking an extra point does not change the volume, while for any
stratum of quadratic differentials in any genus,
marking an extra point doubles the volume:
\begin{equation}
\label{eq:vol:with:marked:point}
\begin{aligned}
\Vol\cH_g(n_1,\dots n_k,0)&=\Vol\cH(n_1,\dots, n_k)\,,\\
\Vol\cQ_g(m_1,\dots m_k,0)&=2\Vol\cQ(m_1,\dots, m_k)\,.
\end{aligned}
\end{equation}
\end{remark}

%----------------------------------------------------------------------------------
\subsection{Recurrence relation in terms of normalized volumes}

The symbol $d$ without indices will mostly denote
the collection $d=(d_1,\dots,d_k)$.
It would be convenient to use normalized Masur--Veech volumes defined as
\begin{equation}
\label{eq:W:g:n}
W_g(n;d)=
W_g(n;d_1,\dots,d_k):=
\frac{\Vol\cQ_g(n,d,-1^{n+D-4g+4})}
{n^g\cdot v(n)\cdot v(d_1)\cdots v(d_k)}\,,
\end{equation}
and normalized area Siegel--Veech constants defined as
\begin{equation}
\label{eq:C:g:n}
C_g^\cC(n;d)=
C_g^\cC(n;d_1,\dots,d_k):=4\pi^2\cdot
\frac{\Vol\cQ_g(n,d,-1^{n+D-4g+4})}
{n^g\cdot v(n)\cdot v(d_1)\cdots v(d_k)}
\cdot\carea(\cC)\,.
\end{equation}

Let
\begin{equation}
\label{eq:D:and:Delta}
D:=|d|:=d_1+\dots+d_k,
\qquad
\Delta:=2\sum_{j=1}^k\frac{1}{d_j+2}.
\end{equation}
and let
\begin{equation}
\label{eq:Lambda:plus}
\Lambda^+_g(n,d)=
\sum_{i=1}^g \lambda_i^+\bigl(\cQ_g(n,d_1,\dots,d_k,-1^{n+D-4g+4})\bigr)\,.
\end{equation}

The Proposition below establishes a recurrence relation for $W_g(n;d)$ in terms of volumes of \textit{simpler} strata, where ``simpler'' means ``smaller $n$, or smaller partition $d$, or smaller genus $g$, or combinations of those''. This relation is based on a combination of the Eskin--Kontsevich--Zorich Formula~\eqref{eq:sum:of:lambda:plus} for the sum of Lyapunov exponents $\Lambda_g^+(n,d)$, Fougeron's bound~\eqref{eq:Fougeron:bound} for the sum of the Lyapunov exponents and  of Goujard's Formula~\eqref{eq:carea} for the area Siegel--Veech constant $\carea$ appearing in Formula~\eqref{eq:sum:of:lambda:plus}. For the two non-varying strata in genus $1$ we use exact Formulae~\eqref{eq:lambda1:plus:n}--\eqref{eq:lambda1:plus:n:1} of Chen--Möller instead of Fougeron's bound for the Lyapunov exponent $\lambda_1^+$.

\begin{theorem}[Main recurrence relation]
\label{th:main:recurrence:relation}
For any fixed $g$ and $d=(d_1,\dots,d_k)$, let
$n\geq\max(1,4g-D-4)$ and assume that the stratum
$\cQ_g(n,d_1,\dots,d_k,-1^{n+D-4g+4})$
is nonempty and connected.
Then the following relations are valid, where the sum is
taken over all admissible configurations containing at least
one cylinder.
\begin{multline}
\label{eq:EKZ:in:W:g}
W_g(n;d)
=\frac{1}{
n+D-k+5-6g+\Delta
+\frac{2}{n+2}
+12 \Lambda^+_g(n,d)
}
\sum_{\cC} C_g^\cC(n;d)
\\=
\frac{1}{
n+D-k+5-6g+\Delta
+O_{g,d}\left(\frac1n\right)
}
\sum_{\cC} C_g^\cC(n;d)\,,
\end{multline}
where the second relation is asymptotic as $n\to+\infty$.
\end{theorem}
\begin{proofwithmanualqed}
For
\(
\cQ(m_1,\dots,m_p)
=
\cQ_g(n,d_1,\dots,d_k,-1^{n+D-4g+4}),
\)
we have
\(
p=n+D+k-4g+5.
\)
Therefore
\begin{multline}
\label{eq:minus:1:24:sum}
-\frac{1}{24}\sum_{i=1}^{n+D+k-4g+5}
\frac{m_i(m_i+4)}{m_i+2}
\\
=
\frac{1}{24}
\left(
3(n+D-4g+4)
-\frac{n(n+4)}{n+2}
-\sum_{j=1}^k\frac{d_j(d_j+4)}{d_j+2}
\right)
\\
=
\frac{1}{24}
\left(
3(n+D-4g+4)
-(n+2)
-\sum_{j=1}^k(d_j+2)
+\sum_{j=1}^k\frac{4}{d_j+2}
+\frac{4}{n+2}
\right)
\\
=
\frac{1}{12}
\left(
n+D-k+5-6g+\Delta+\frac{2}{n+2}
\right).
\end{multline}

By Fougeron's bound, for fixed $g$ and fixed $d=(d_1,\dots,d_k)$, we have
$\Lambda^+_g(n,d)=O_{g,d}\left(\frac{1}{n}\right)$ as $n\to+\infty$.
Thus,
\begin{multline*}
\sum_{i=1}^g
\lambda_i^+\bigl(\cQ_g(n,d_1,\dots,d_k,-1^{n+D-4g+4})\bigr)
-\frac{1}{24}\sum_{i=1}^{n+D+k-4g+5}
\frac{m_i(m_i+4)}{m_i+2}
\\
=
\frac{1}{12}
\left(
n+D-k+5-6g+\Delta
+\frac{2}{n+2}
+12 \Lambda^+_g(n,d)
\right)
\\
=
\frac{1}{12}
\left(
n+D-k+5-6g+\Delta
+O_{g,d}\left(\frac1n\right)
\right)\,.
\end{multline*}
Formula~\eqref{eq:sum:of:lambda:plus} therefore gives
\begin{multline*}
\sum_{\cC} C_g^\cC(n;d)
=\left(
n+D-k+5-6g+\Delta
+\frac{2}{n+2}
+12 \Lambda^+_g(n,d)
\right)W_g(n;d)
\\=\left(
n+D-k+5-6g+\Delta
+O_{g,d}\left(\frac1n\right)
\right)
W_g(n;d).\qed
\end{multline*}
\end{proofwithmanualqed}

Our Key Theorem~\ref{th:n:D:large:zero:any:g} takes the following
form in terms of the renormalized volumes.

\begin{theorem}
\label{th:n:m1:dots:mk:lim:for:renormalized}
For any $g\ge 1$ there exists a pair of strictly positive constants $C_g^{\mathit{even}}$ and $C_g^{\mathit{odd}}$ such that for any (possibly empty) collection $d$ of positive integers $d=(d_1,\dots,d_k)$, one has
\begin{equation}
\label{eq:W:any:genus:g:limits}
W_g(n;d)=
\begin{cases}
C_g^{\mathit{even}}+O_{g,d}\left(\frac{1}{n}\right)
        &\text{for even }n\,,
\\
C_g^{\mathit{odd}}+O_{g,d}\left(\frac{1}{n}\right)
        &\text{for odd }n\,.
\end{cases}
\end{equation}
\end{theorem}
\begin{proof}
Fix any positive integer upper bounds $g_{max}, D_{max}$ for respectively $g$ and $D=d_1+\dots+d_k$. The set of pairs $(g,d)$ satisfying these bounds is finite. By the standard convention, $D=0$ when $d=\emptyset$.

In Proposition~\ref{prop:minimal:any:g} we prove by induction in $g$
that~\eqref{eq:W:any:genus:g:limits} is valid for $d=\emptyset$ and for all $g$ satisfying $g\le g_{max}$. The relation~\eqref{eq:W:in:genus:0} according to which $W_0(n)=2\pi^2$ for any $n$ serves as a base of induction.
For $g=1$ the statement also follows from~\eqref{eq:volume:n}
which we already proved, providing us with explicit values
$C^{\mathit{even}}_1 = C^{\mathit{odd}}_1 = \frac{1}{3}$.
We show in Proposition~\ref{prop:minimal:any:g} that for the minimal stratum
the sum in~\eqref{eq:EKZ:in:W:g} gets reduced to $W_g(n-2)$ with
an explicit multiplicative coefficient plus a controllable small error term,
so the recursion becomes particularly simple.

Increasing progressively the genus $g\le g_{max}$, we prove
in Proposition~\ref{prop:general:any:g} the general statement~\eqref{eq:W:any:genus:g:limits} for all non-minimal strata
in genus $g$ satisfying $D\le D_{max}$. We prove it by induction in $D$,
where the case $D=0$ serves us as a base of the induction.
At each step the sum in~\eqref{eq:EKZ:in:W:g} is expressed
in terms of $W_g(n-2;d)$, which is our recursion, and in terms of
$W_g(n;d')$, with $|d'|<D$, and $W_{g'}(n;d')$, with $g'<g$, $|d'|\le D$,
for which~\eqref{eq:W:any:genus:g:limits} has already been established.
This completes the proof of Theorem~\ref{th:n:m1:dots:mk:lim:for:renormalized}
conditional on Propositions~\ref{prop:minimal:any:g} and~\ref{prop:general:any:g}.
\end{proof}

Note that we are using error terms $O(\frac{1}{n})$ all the time. Thus,
it is important at any point to use in the induction only a finite number of pairs $(g,d)$ for which we assume the statement to be proved.

%----------------------------------------------------------------------------------
\subsection{Structure of the paper}

From now on we will always work with the key recurrence Formula~\eqref{eq:EKZ:in:W:g}. In Section~\ref{s:Configurations} we recall basic facts on coding configurations by decorated ribbon graphs as in~\cite{Masur:Zorich} and~\cite{Goujard:carea} and characterize configurations specific for strata having a zero of large order. In Section~\ref{s:from:graphs:to:principal:boundary} we prepare certain general relations for the renormalized volumes and recall Goujard's algorithm of evaluation of the combinatorial factor involved in her Formula~\eqref{eq:carea} for the area Siegel--Veech constant. In Sections~\ref{s:Contributions:of:graphs:a}--\ref{s:Contributions:of:graphs:c} we apply Goujard's general Formula~\eqref{eq:carea} to configurations selected in Section~\ref{s:Configurations} and express the answers in terms of renormalized variables $W_g(n;d)$ and $C_g^\cC(n;d)$. In Section~\ref{s:Dominant:and:negligible:contributions} we detect configurations carrying dominant contributions to the sum in~\eqref{eq:EKZ:in:W:g} and prove that the total contribution of remaining configurations is negligible for large values of $n$. In Section~\ref{s:Preparatory:lemmas} we state and prove several preparatory lemmas handling certain specific abstract recursions with controllable small error terms. Finally, in Section~\ref{s:Volume:asymptotics} we plug formulae from Sections~\ref{s:Contributions:of:graphs:a}--\ref{s:Dominant:and:negligible:contributions} for contributions of individual configurations and evaluate the sum over configurations found in Section~\ref{s:Configurations} thus deriving a much more explicit form of the key recurrence Formula~\eqref{eq:EKZ:in:W:g}. For the two non-varying series of strata this provides us with identities justifying Formulae~\eqref{eq:volume:n} and~\eqref{eq:volume:n:1} from Theorem~\ref{th:n:1}. For other strata, this explicit form of~\eqref{eq:EKZ:in:W:g} allows us to interpret the recursion as one of the abstract ones considered in Section~\ref{s:Preparatory:lemmas}. To make the paper self-contained we reproduce in Appendix~\ref{a:configurations} necessary information on combinatorial description of configurations of saddle connections elaborated originally in~\cite{Masur:Zorich}.

We admit that, formally speaking, computational Sections~\ref{s:Contributions:of:graphs:a}--\ref{s:Contributions:of:graphs:c} consist in plugging various combinatorial data one-by-one into Goujard's formula. However, each ``plugging in'' requires evaluation of numerous combinatorial constants, of discrete symmetries and such. A mistake in a factor of $2$ in a contribution of a single dominant configuration would make the abstract recursion of Lemma~\ref{lm:lemma:in:calculus} inapplicable. Thus, our calculations were thoroughly debugged (by hand, not by AI) using available exact values of volumes in small dimensions from~\cite{Goujard:vol} and experimental values of Lyapunov exponents. We expect that alternative calculations of Masur--Veech volumes (through quasimodularity, intersection theory, integrable hierarchies, etc) would emerge in the future as in the paper by D.-M.~Nguyen~\cite{Nguyen} providing an alternative computation of volumes in genus $0$ for strata with only odd-order zeroes. The close answers for contributions of various configurations derived in Sections~\ref{s:Contributions:of:graphs:a}--\ref{s:Contributions:of:graphs:c} have all chances to serve as reliable tests of intermediate results in most of these approaches.

% ###############################################################
% ###############################################################
% ###############################################################
\section{Configurations which contribute to specific strata}
\label{s:Configurations}

In this section we prepare lists of configurations which contribute to area Siegel--Veech constants of strata under consideration. We are interested only in strata having many poles. Among such strata we distinguish the minimal strata, especially the minimal strata $\cQ_1(n,-1^n)$ in genus $1$. We are also particularly interested in the series of non-varying strata $\cQ_1(n,1,-1^{n+1})$ in genus $1$.

%------------------------------------------------------------------
\subsection{Ambient stratum associated to a decorated ribbon graph}
\label{ss:Ambient:stratum:associated:to:graph}

We use encoding and classification of \textit{configurations of \^homologous saddle connections} given in~\cite{Masur:Zorich}.
To make the paper self-contained we outline the terminology and the results of~\cite{Masur:Zorich} in Appendix~\ref{a:configurations}.

\begin{definition}
\label{def:endpoint}
Having a configuration $\cC$ of saddle connections on a flat surface $S=(X,q)$ corresponding to a meromorphic quadratic differential $q$ with at most simple poles, we say that a conical singularity of $S$ (i.e. a zero or pole of $q$) is an \textit{endpoint of $\cC$} when it is an endpoint of at least one of the saddle connections in the configuration $\cC$. The \textit{interior} singularities are those zeroes and simple poles which do not intersect with any saddle connections of the configuration $\cC$. Thus, all zeroes and poles of $q$ are split into two disjoint subsets of \textit{endpoints} of $\cC$ and of \textit{interior} singularities.
\end{definition}

Configurations of \^homologous saddle connections are encoded by decorated ribbon graphs described in Appendix~\ref{a:configurations}. This encoding is axiomatized in Definition~\ref{def:configuration}. We recall now how to reconstruct an ambient stratum from a legal decorated ribbon graph $\GG$.

Following any boundary component $(\partial\GG)_i$ of the global ribbon graph $\GG$ we get a chain of edges $\gamma_{i,1}\to\dots\gamma_{i,t(i)}\to$ of the graph; the case of $t=1$ is not excluded. The decoration of the graph associates to every pair of consecutive edges a nonnegative integer number $k_{i,j}$ called an \textit{order of boundary singularity}. In the case $t=1$ the edge follows itself, but we still have an associated order of boundary singularity. Define an integer number $b_i$ as
\begin{equation}
\label{eq:degree:of:singularity}
b_i = -2+\sum_{j=1}^{t(i)} (k_{i,j}+1)
\end{equation}
where the sum is taken over all edges of the boundary component $(\partial\GG)_i$ of the global ribbon graph.

Each vertex of the decorated ribbon graph $\GG$ has one of three types: $\oplus$, $\ominus$, or $\circ$, (where $\circ$-vertices represent maximal cylinders filled with regular closed geodesics). The decoration of the graph $\GG$ associates to every $\oplus$- or $\ominus$-vertex $v_j$ a (possibly empty) set with multiplicities
\[
\{m_{j,1},\dots,m_{j,q(j)}\}
\]
consisting of integers satisfying $m_{j,l}\ge -1$.

A concatenated set with multiplicities
\begin{equation}
\label{eq:set:with:multiplicities}
\{b_1,\dots,b_r\}\sqcup\{m_{1,1},\dots,m_{1,q(1)}\}\sqcup\dots\sqcup\{m_{s,1},\dots,m_{s,q(s)}\}\,,
\end{equation}
where $r$ is the number of boundary components of $\partial\GG$ and
$s$ is the total number of $\oplus$ and $\ominus$-vertices of the graph,
is in bijection with the collection of singularities
$\{m_1,\dots,m_p\}$ of the ambient stratum
$\cQ(m_1,\dots,m_p)$.
Note that all singularities
$\{m_1,\dots,m_p\}$ of
$\cQ(m_1,\dots,m_p)$ are labeled. Our bijection endows
the set with multiplicities~\eqref{eq:set:with:multiplicities}
with these labels.

%------------------------------------------------------------------
\subsection{Boundary strata associated to a decorated ribbon graph}

Recall now how to read a principal boundary stratum $\cQ(\alpha'(\cC))$
involved in Goujard's Formula~\eqref{eq:carea} from a decorated ribbon graph $\GG(\cC)$ (see~\cite[Section 1.4]{Masur:Zorich} for details and examples).

To every $\ominus$-vertex $v_i$ (if any) we associate
a stratum of quadratic differentials $\cQ(\alpha_i)$; to every
$\oplus$-vertex $v_j$ (if any) we associate
a stratum of Abelian differentials $\cH(\beta_j)$; $\circ$-vertices are ignored.
In both cases the stratum is completely determined by a
collection of \textit{interior singularities} associated to
the vertex $v$, by a \textit{local ribbon graph} $\GG(v)$
around the vertex $v$ and by degrees of \textit{boundary singularities}
of this local ribbon graph. The global structure of the graph is irrelevant
for determining the boundary stratum associated to any given vertex $v$.

Let $v$ be a $\ominus$ or a $\oplus$-vertex of the graph. To each   \textit{connected}   component $\cB$ of  the \textit{local} ribbon graph $\GG(v)$ (which should not be confused with \textit{boundary} components
$(\partial\GG)_m$
of the \textit{global} ribbon graph $\GG(\cC)$ considered in Section~\ref{ss:Ambient:stratum:associated:to:graph}) we assign a number
\begin{equation}
\label{eq:Dji}
\delta=k_1+\dots+k_p-2,
\end{equation}
where $k_1, \dots, k_p$ are degrees of boundary singularities corresponding to the connected component $\cB$ of the local ribbon graph. It follows from Definition~\ref{def:configuration} that the    number $\delta$ is always  a nonnegative even integer whenever $v$ is a $\oplus$-vertex and an integer satisfying $\delta\ge -1$ whenever $v$ is a $\ominus$-vertex.

We assign to a $\ominus$-vertex $v_i$
a  stratum
\begin{equation}
\label{eq:Q:boundary:stratum}
\cQ(\alpha'_i)=\cQ(d_1, \dots, d_{s(i)}, \delta_1, \dots,  \delta_{r(i)})
\end{equation}
of meromorphic quadratic differentials with at most simple poles,
where  $d_1,  \dots,   d_{s(i)}$   are  the  orders  of  interior
singularities   of    $v_i$ and $r(i)$ is the number of
connected components of the local ribbon graph $\GG(v_i)$.

In Figure~\ref{fig:embedded:local:ribbon:graphs} classifying the local
ribbon graphs, the graphs are labeled by two integers separated by a point.
The first number is the valency of the vertex, while the second
number is the number of connected components $r(v_i)$ of the
local ribbon graph $\GG(v_i)$.

Condition   (5)   in Definition~\ref{def:configuration}   of    a   configuration   of \^homologous  saddle  connections  guarantees  that  the  sum  of entries of $\alpha'_i$ defined above equals $0$ modulo $4$, and is greater than or equal to $-4$, while condition             (6)             guarantees             that $\alpha'\not\in\big\{\emptyset,\{-1,1\},\{3,1\},\{4\}\big\}$, which implies that the stratum $\cQ(\alpha'_i)$ is nonempty. For this nonemptiness test, marked points represented by entries $0$ are omitted.

We assign to  a $\oplus$-vertex $v_j$  a stratum
\begin{equation}
\label{eq:H:boundary:stratum}
\cH(\beta'_j)=\cH\left(\frac{d_1}{2}, \dots, \frac{d_{s(j)}}{2}, \frac{\delta_1}{2}, \dots,  \frac{\delta_{r(j)}}{2}\right)
\end{equation}
of  holomorphic   Abelian   differentials,   where  $d_1,  \dots,
d_{s(j)}$ are the orders of interior singularities of $v_j$
and $r(j)$ is the number of connected components of the
local ribbon graph $\GG(v_j)$.

Conditions (4) and (5) in definition~\ref{def:configuration}
of a configuration of \^homologous saddle  connections imply that
the entries of  $\beta'_j$  are integers  and  that their sum  is
even, so the stratum $\cH(\beta'_j)$ is nonempty.

%------------------------------------------------------------------
\subsection{Configurations for strata with many poles}

\begin{lemma}
\label{lm:simple:pole:in:a:graph}
A simple pole $P$ of a quadratic differential serves as an endpoint of a configuration $\cC$ if and only if the associated boundary component of the global decorated ribbon graph $\GG(\cC)$ contains a single boundary singularity, and if, moreover, this single boundary singularity has degree $0$.
\end{lemma}
\begin{proof}
Recall that all saddle connections in any configuration of \^homologous saddle connections are parallel. Thus, if a simple pole $P$ is an endpoint of $\cC$ in the sense of definition~\ref{def:endpoint}, then there is exactly one saddle connection $\gamma_i$ of the configuration $\cC$ having $P$ as the endpoint. Going at a small distance around $P$ we start from $\gamma_i$ and return to $\gamma_i$ without meeting any other saddle connections. So the associated boundary component is a loop with a single boundary singularity. The cone angle at the simple pole $P$ is $\pi=(k+1)\pi$, so the degree $k$ of this single boundary singularity is $0$. The proof of the reciprocal statement is completely analogous.
\end{proof}

\begin{corollary}
\label{cor:at:most:four:poles}
A configuration can have at most four endpoint poles. In particular,
an ambient stratum $\cQ$ corresponding to a graph containing no $\ominus$-vertices can have at most four simple poles.
\end{corollary}
\begin{proof}
It follows from the classification of decorated ribbon graphs~\ref{fig:graphs} and~\ref{fig:embedded:local:ribbon:graphs} that a global graph $\GG$ can have at most four boundary components as in Lemma~\ref{lm:simple:pole:in:a:graph}, which proves the first part of the statement. In the absence of $\ominus$-vertices, all simple poles are endpoint singularities of the configuration, which completes the proof.
\end{proof}

\begin{lemma}
\label{lm:abc:only}
A stable graph carrying a nonzero contribution to the area Siegel--Veech constant contains at least one $\circ$-vertex. If the associated stratum of quadratic differentials has at least five simple poles, then the corresponding graph contains at least one $\ominus$-vertex, and, hence, has one of the types ``a,b,c'' in notations of Figure~\ref{fig:graphs}.
\end{lemma}
\begin{proof}
The first part of the statement is trivial:
by definition of the \textit{area} Siegel--Veech constant,
any contributing configuration contains at least one cylinder.
The second part follows immediately from Corollary~\ref{cor:at:most:four:poles}.
\end{proof}

\begin{lemma}
\label{lm:ominus:vertex:associated:to:zero:of:degree:n}
If a zero of degree $n>4g-4$ is an interior singularity, then it is associated to a $\ominus$-vertex.

If a zero of degree $n>4g$ is an endpoint singularity, and the graph contains at least one $\circ$-vertex, then the boundary component $(\partial\GG)_m$ of the global ribbon graph $\GG$ associated to this zero has nontrivial intersection with the local ribbon graph of exactly one vertex of $\ominus$-type.

Both bounds are sharp.
\end{lemma}
\begin{proof}
An interior singularity of an Abelian boundary stratum is necessarily even. If it has order $m$, the associated zero has order $\frac{m}{2}$, see~\eqref{eq:H:boundary:stratum}, and the genus $g_i$ of the associated boundary stratum satisfies $\frac{m}{2}\le 2g_i-2$. Since the genus of the boundary stratum cannot be larger than the genus of the stratum itself, we get a bound $m\le 4g-4$ for the degree of any interior singularity of any $\oplus$-vertex.

For $g\ge 2$, to verify that the bound in the first assertion is sharp consider a graph of type ``d'' (see Figure~\ref{fig:graphs}) with two $\circ$-vertices of valency $3$, a single $\oplus$-vertex of type $+2.1$ between them and with no other vertices (in particular, with no vertices on the loops). Let the $\oplus$-vertex have a single interior singularity of order $4g-4$ and two boundary singularities of orders $1$. We get a legal decorated ribbon graph corresponding to the ambient stratum $\cQ_g(4g-4,4,-1^4)$ in genus $g$.
This example also shows that the bound on the number of poles is sharp
in Lemma~\ref{lm:abc:only}.
For $g=1$, omit the interior singularity of order $0$; the same graph then gives a four-pole example in $\cQ_1(4,-1^4)$.

Suppose now that a zero of degree $n$ of the ambient stratum is an endpoint singularity of $\cC$. Graphs having more than one $\ominus$-vertex, namely two $\ominus$-vertices are of type ``a''. A graph $\GG$ of type ``a'' containing a $\circ$-vertex cannot have a boundary component $\partial\GG$ intersecting the local ribbon graphs of both $\ominus$-vertices. This implies that for any decorated ribbon graph $\GG$ (not only of type ``a''), the boundary component $(\partial\GG)_m$ from the second assertion of Lemma~\ref{lm:ominus:vertex:associated:to:zero:of:degree:n} has nontrivial intersection either with one $\ominus$-vertex, or with no $\ominus$-vertices at all. In order to prove the Lemma, we have to exclude the second case (and also prove that the bound $n>4g$ is sharp).

If the entire graph $\GG$ does not have $\ominus$-vertices, by Corollary~\ref{cor:at:most:four:poles} the ambient stratum has at most $4$ simple poles. Hence, the sum of orders of zeroes is at most $4g$, and hence the degree of each zero, in particular, $n$ is at most $4g$.

If all $\ominus$-vertices of the graph $\GG$ are of local type $-1.1$, then by assumption, they do not have intersection with the distinguished component $(\partial\GG)$. Replace each $\ominus$-vertex with a vertex carrying genus $0$, having three poles as interior singularities and no other interior singularities and having a single boundary singularity of degree $1$. We get a legal decorated ribbon graph $\tilde\GG$. The boundary component $(\partial\tilde\GG)_0$ of the global ribbon graph $\tilde\GG$ containing the boundary singularity of degree $1$ of the modified $\ominus$-vertex has at least one more boundary singularity, so the associated endpoint zero of the ambient stratum has degree at least $1$ and this zero is distinct from our distinguished zero $n$. Thus, if there was only one $\ominus$-vertex of local type $-1.1$, the ambient stratum has at most $5$ simple poles and at least one more zero of positive degree in addition to the distinguished zero of degree $n$. Since the sum of the degrees of all singularities is at most $4g-4$, we conclude that $n\le 4g$. If there were two $\ominus$-vertices of local type $-1.1$, the ambient stratum has exactly $6$ simple poles and at least two zeroes of positive degrees in addition to $n$. We again have $n\le 4g$.

If our graph has type ``b'' we have a single $\ominus$-vertex.
If it is of local type $-2.1$, replace it by a $-2.1$-vertex carrying
genus $0$, having $\{-1^3\}$ as interior singularities and $0$ and $1$
as boundary singularities. It follows from assumptions that these
boundary singularities belong to the same boundary component of the
global ribbon graph and hence this boundary component corresponds to a zero of positive degree, and this zero is distinct from our distinguished zero.
Thus, the ambient surface has exactly three poles and at least one zero in addition to the distinguished zero of degree $n$. Since the genus of the ambient stratum is at most $g$, we get $n+1-3\le 4g-4$ and hence $n\le 4g-2$.

Finally, if the
graph has type ``b'' and a single $\ominus$-vertex of local type $-2.2$,
we replace it by a $\ominus$-vertex of local type $-2.2$ carrying genus $0$,
and having two boundary singularities of degrees $1$, two interior singularities of degrees $-1^2$ and no other singularities. The same reasoning as above shows that in this case we have $n\le 4g-4$.
Lemma~\ref{lm:ominus:vertex:associated:to:zero:of:degree:n} is proved.

Consider a graph of type ``a'' having a single $\oplus$-vertex of local type $+2.1$ separated from the two $\ominus$-vertices by $\circ$-vertices. Let both $\ominus$-vertices carry genus $0$ and let the $\oplus$-vertex carry genus $g$. Let the $\oplus$-vertex have no interior singularities; denote by $k', k''$ the degrees of its boundary singularities. By~\eqref{eq:Dji} and~\eqref{eq:H:boundary:stratum} the  associated Abelian boundary stratum has a single zero of degree $(k'+k''-2)/2$, so $k'+k''=4g-2$. Hence, $n=-2+(0+1)+(k'+1)+(0+1)+(k''+1)=k'+k''+2=4g$ for such a graph. This proves that the bound cannot be loosened.
\end{proof}

\begin{definition}
\label{def:zero:carried:by:a:minus:vertex}
For a configuration containing a cylinder, we will say that a zero of order $n>4g$ is \textit{carried by a $\ominus$-vertex} $v$ if $v$ is the $\ominus$-vertex satisfying the
assertion of Lemma~\ref{lm:ominus:vertex:associated:to:zero:of:degree:n}.
\end{definition}
For configurations containing a cylinder, existence and uniqueness of such a
$\ominus$-vertex $v$ are guaranteed by Lemma~\ref{lm:ominus:vertex:associated:to:zero:of:degree:n}.

%------------------------------------------------------------------
\Needspace{6\baselineskip}
\subsection{Configurations appearing for minimal strata}
\begin{proposition}
\label{pr:list:of:configurations:in:minimal:stratum}
If a stable ribbon graph contributes to the area Siegel--Veech constant of the minimal stratum $\cQ_g(n,-1^{n-4g+4})$ with $g\ge 1$ and $n>4g$, then it has one of the following three types in the notation of Figures~\ref{fig:graphs} and~\ref{fig:embedded:local:ribbon:graphs}:

I. Graph of type ``b'' with a single $\circ$-vertex of local type $\circ2.2$, with a single $\ominus$-vertex of local type $-2.1$
and with at most $(g-1)$ $\oplus$-vertices, all of local type $+2.1$ (possibly with no $\oplus$-vertices).

II. Graph of type ``c'' with a single $\oplus$-vertex of valency $3$
of local type $+3.1$, a single $\circ$-vertex
of local type $\circ2.2$
located
at the loop of the graph, a single
$\ominus$-vertex of valency $1$, and with at most $(g-2)$
(possibly with none)
$\oplus$-vertices of local type $+2.1$.

III. Graph of type ``c'' with a single $\circ$-vertex of valency $3$
and with a single $\ominus$-vertex of valency $1$.
The loop of the graph does not have any vertices.
All other vertices (if any) are of local type $+2.1$; they are located on the stem between the $\ominus$-vertex and the single $\circ$-vertex.

The genus $g_0$ carried by the $\ominus$-vertex is nonnegative.
The genera $g_1,\dots, g_k$ carried by $\oplus$-vertices
are strictly positive. These genera satisfy the relation
\[
g_0+g_1+\dots+g_k=
\begin{cases}
g-1&\text{in cases I and II}\,,\\
g&\text{in case III}\,.
\end{cases}
\]
\end{proposition}
\begin{figure}[htb]
  \centering
\includegraphics[scale=0.43]{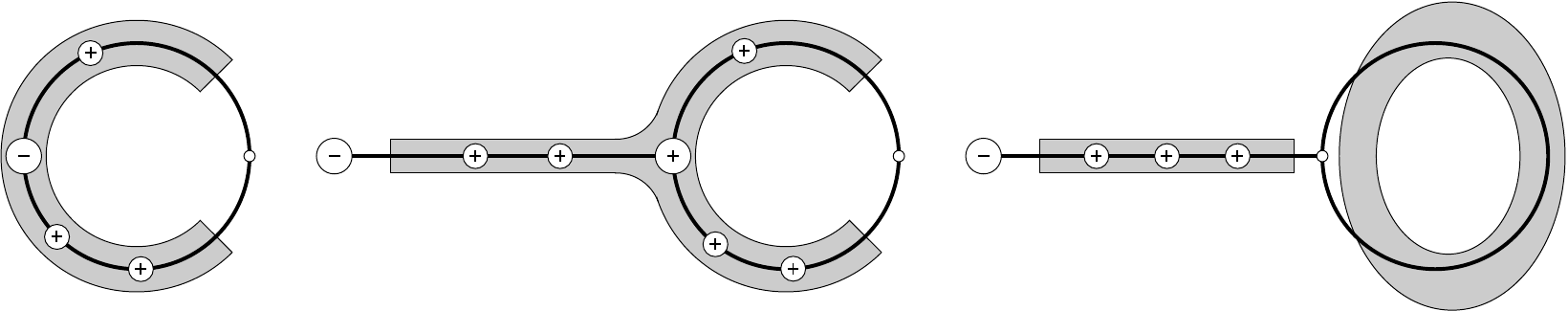}
\vspace*{10pt}
\begin{picture}(0,0)(350,-76)
\put(35,-88){$I.$}
\put(173,-88){$II.$}
\put(314,-88){$III.$}
\end{picture}
  \caption{Stable ribbon graphs contributing to the area Siegel--Veech constant of the minimal stratum. The number of $\oplus$-vertices
  of valency $2$ varies from $0$ to $g-1$ for the graph I,
  from $0$ to $g-2$ for the graph II and from $0$ to $g$ for the graph III.}
  \label{fig:minimal:genus:g}
\end{figure}
\begin{proof}
The surfaces in the stratum under consideration have
$n-4g+4$ poles. Since $n>4g$, it implies that they have at least $5$ poles and Lemma~\ref{lm:abc:only} is applicable: we can consider graphs of type ``a,b,c'' only, and only those of them which contain at least one $\circ$-vertex.

Global ribbon graphs of type ``a'' containing a $\circ$-vertex
have at least two boundary components
giving rise to at least two distinct zeroes
(see the picture $\circ2.2$ of the only possible local ribbon graph
adjacent to a $\circ$-vertex of valency $2$ in Figure~\ref{fig:embedded:local:ribbon:graphs}). Thus,
they do not appear in the minimal stratum.

For a graph of type ``b'' any $\circ$-vertex has valency $2$, and hence, by the classification of local ribbon graphs from Figure~\ref{fig:embedded:local:ribbon:graphs} is of $\circ2.2$-type. Since the ribbon graph should have a single connected component, it means that the $\ominus$-vertex  is of type $-2.1$ and all $\oplus$-vertices (if any) have type $+2.1$. The cylinder represented by the $\circ$-vertex contributes $1$ to the genus, which implies that $g_0+g_1+\dots+g_k=g-1$.

Any graph of type ``c'' has a single vertex of valency $3$.
Suppose that this vertex of valency $3$ is a $\oplus$-vertex.
The graph should also contain at least one $\circ$-vertex.
Such a $\circ$-vertex cannot be located
on the stem between the
$\ominus$-vertex (the only vertex of valency $1$) and the
vertex of valency $3$ since otherwise, the ribbon graph would contain
at least two boundary components giving rise to two distinct zeroes.
Thus, the $\circ$-vertex is located on the loop. This excludes
local ribbon graphs of types $+3.2a, +3.2b, +3.3$ around the
vertex of valency $3$ because in combination with a $\circ$-vertex
on the loop of such graphs, the
associated global ribbon
graphs get at least two boundary components giving rise to
two distinct zeroes.
Ribbon graphs of type ``c'' carrying a vertex
of type $+3.1$ are realizable in the minimal
stratum. To have a single boundary component they are forced to have a single $\circ$-vertex located on the loop
and possibly a certain number of $\oplus$-vertices all of types
$+2.1$. There is no restriction on location of such $\oplus$-vertices. The cylinder represented by the $\circ$-vertex
contributes $1$ to the genus, which implies that $g_0+g_1+\dots+g_k=g-1$, where, as always, the genus $g_0$ carried
by the $\ominus$-vertex is nonnegative, while genera
$g_1,\dots,g_k$ carried by $\oplus$-vertices (if any)
are strictly positive.

In the only remaining case the single vertex of valency $3$
of a graph of type ``c'' is a $\circ$-vertex.
By the classification of local ribbon graphs,
there is a unique local ribbon graph structure around a $\circ$-vertex of valency $3$,
denoted by
$\circ3.2$ in Figure~\ref{fig:embedded:local:ribbon:graphs}.
This local ribbon graph structure implies that
if such a graph  corresponds to a minimal
stratum, then the loop of the graph does not have
any vertices inside it, in which case the connected component of the
global ribbon graph corresponding to the annulus around the loop
has two boundary components both representing simple poles.
(Such a graph corresponds to a version of the
pocket configuration, see Figure~\ref{fig:pocket}.)
The edge between the single $\ominus$-vertex of valency $1$
and the single vertex of valency $3$ may contain $\oplus$-vertices (possibly none) all of $+2.1$-type. As before the genus $g_0$ carried by the single $\ominus$ vertex is nonnegative, while the
genera corresponding to $\oplus$-vertices are strictly positive.
The only cylinder in this configuration is ``pocket like'': one of its
boundary components is decomposed into two segments of equal lengths,
which are isometrically identified. Thus, this time the cylinder does not contribute to the genus and we have
$g_0+g_1+\dots+g_k=g$. The proof of Proposition~\ref{pr:list:of:configurations:in:minimal:stratum}
is completed.
\end{proof}

%------------------------------------------------------------------
\subsection{Configurations appearing for minimal strata in genus one}

\begin{figure}[htb]
  \centering
\hspace*{25pt}
\includegraphics[scale=0.5]{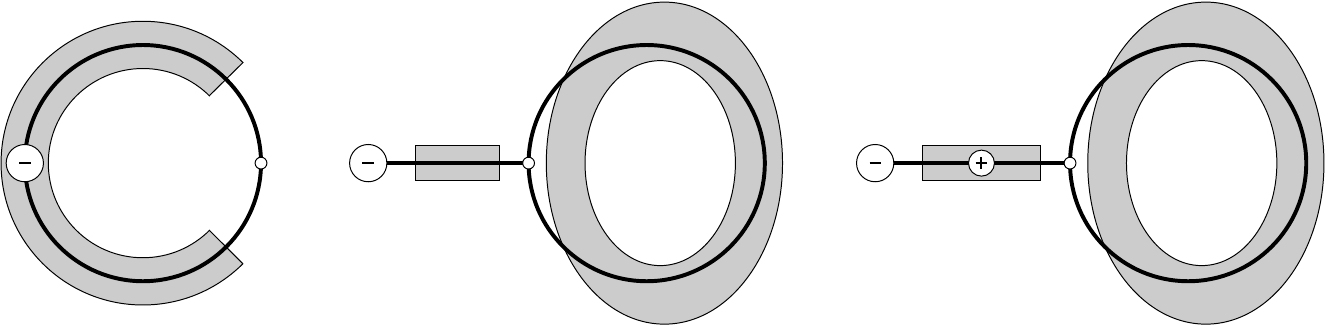}
\begin{picture}(0,0)(499,-87) % (350,-76)
\put(162,-40){$\scriptstyle g=0$}
\put(151,-68){$\scriptstyle \{-1^n\}$}
\put(227,-64){$\scriptstyle 0$}
\put(227,-36){$\scriptstyle 0$}
\put(167,-50.5){$\scriptstyle n'$}
\put(190,-50.5){$\scriptstyle n''$}
\put(205,-96){$I$}
\end{picture}
\begin{picture}(0,0)(435,-75) % (430,-70)
\put(190,-27){$\scriptstyle g=1$}
\put(185,-55){$\scriptstyle \{-1^{n-2}\}$}
\put(203,-42){$\scriptstyle n$}
\put(229.5,-42){$\scriptstyle 0$}
\put(242,-38){$\scriptstyle 0$}
\put(251,-38){$\scriptstyle 0$}
\put(260,-85){$II$}
\end{picture}
\begin{picture}(0,0)(308,-75)
\put(185,-27){$\scriptstyle g=0$}
\put(180,-55){$\scriptstyle \{-1^{n-2}\}$}
\put(185,-45){$\scriptstyle n-4$}
\put(213,-30.5){$\scriptstyle 1$}
\put(213,-46){$\scriptstyle 1$}
\put(229.5,-45){$\scriptstyle 0$}
\put(242,-38){$\scriptstyle 0$}
\put(251,-38){$\scriptstyle 0$}
\put(260,-85){$III$}
\end{picture}
\vspace*{10pt}
  \caption{Stable ribbon graphs contributing to the area Siegel--Veech constant $\carea(\cQ_1(n,-1^n))$}
  \label{fig:minimal:genus:1}
\end{figure}

\begin{corollary}
\label{cor:configurations:minimal:genus:1}
A stable ribbon graph contributes to $\carea(\cQ_1(n,-1^n))$,
with $n\ge 5$, if and only if it is one of the
three graphs represented in Figure~\ref{fig:minimal:genus:1}.
\end{corollary}
\begin{proof}
Proposition~\ref{pr:list:of:configurations:in:minimal:stratum}
limits the list of configurations contributing to
the area Siegel--Veech constant of the minimal stratum
to three types.

A graph of type I cannot have any $\oplus$-vertices in genus $1$.
It has a single $\ominus$-vertex of local type $-2.1$ and a single
$\circ$-vertex of local type $\circ2.2$. The corresponding
global ribbon graph is presented in the left picture in Figure~\ref{fig:minimal:genus:1}.

Graphs of type II in the classification of Proposition~\ref{pr:list:of:configurations:in:minimal:stratum}
are present
only starting from genus $2$ since the $\oplus$-vertex already
carries genus $g_1\ge 1$, and the cylinder
contributes $1$ to the genus.

It follows from the classification given in
Proposition~\ref{pr:list:of:configurations:in:minimal:stratum}
that in genus $1$ a graph of type III either has a $\ominus$-vertex
carrying genus $g_0=1$, in which case it has no $\oplus$-vertices at all, or has a $\ominus$-vertex
carrying genus $g_0=0$ and a single $\oplus$-vertex of type $+2.1$
carrying genus $g_1=1$. The corresponding
global ribbon graphs are presented in the middle and in the right pictures in Figure~\ref{fig:minimal:genus:1}.
\end{proof}

%-------------------------------------------------------------------
\subsection{%
  \texorpdfstring{%
    \for{toc}{Configurations appearing for strata $\mathcal Q_1(n,1,-1^{n+1})$}%
    \except{toc}{%
      \boldmath Configurations appearing for strata $\mathcal Q_1(n,1,-1^{n+1})$%
    }%
  }{Configurations appearing for strata Q1(n,1, with n+1 simple poles)}%
}

In the rest of this section we elaborate a complete list of decorated ribbon graphs corresponding to the stratum $\cQ_1(n,1,-1^{n+1})$. This list will be needed in proving Formula~\eqref{eq:volume:n:1} from Theorem~\ref{th:n:1}. The list is longer than for the minimal stratum and requires a tedious analysis of combinatorial data. We note that the closed formula~\eqref{eq:volume:n} for the volumes of the minimal strata in genus $1$ is already sufficient for the proof of Theorem~\ref{th:n:m1:dots:mk:lim}. Thus, a reader who is not particularly interested in all details of construction of a complete list of configurations specific for the strata $\cQ_1(n,1,-1^{n+1})$ might proceed directly to Section~\ref{s:from:graphs:to:principal:boundary}.

We start with a preparatory technical Lemma.

\begin{lemma}
\label{lm:no:oplus:vertices:of:valence:3}
No graph of type ``c'' contributing to
$\carea(\cQ_1(n,1,-1^{n+1}))$, where $n\geq5$,
contains a $\oplus$-vertex of valence $3$.
\end{lemma}

\begin{proof}
Suppose the contrary. Each $\oplus$-vertex carries strictly positive genus which contributes to the genus of the ambient surface. Thus, in genus $1$ the $\oplus$-vertex is unique and carries genus $1$. Denote it by $v_+$; denote the unique $\ominus$-vertex by $v_-$.

Since the configuration contributes to the area Siegel--Veech
constant, the graph contains at least one $\circ$-vertex.
A cylinder on the loop contributes to the genus, so
there are no $\circ$-vertices on the loop.
By the classification of graphs, there is exactly one
$\circ$-vertex, and it lies on the stem between $v_+$ and $v_-$.

Condition~(5) of Definition~\ref{def:configuration} gives
\[
\sum_l d_l+\sum_i\delta_i=0.
\]
for the vertex $v_+$. Orders $d_l$ of all interior singularities of any $\oplus$-vertex are strictly positive, and all $\delta_i$ are nonnegative. Hence $v_+$ has no interior singularities and every $\delta_i$ is zero. By~\eqref{eq:Dji}, the sum of the orders of boundary singularities on each connected component of $\GG(v_+)$ is therefore equal to $2$.

The local ribbon graph of the cylinder is of type $\circ2.2$.
Its two local components separate the $v_+$-side from the
$v_-$-side of the global ribbon graph.
In particular, no boundary component of $\partial\GG$
meeting $\GG(v_+)$ also meets $\GG(v_-)$.
Since $n\geq5>4g$, Lemma
\ref{lm:ominus:vertex:associated:to:zero:of:degree:n}
implies that none of these boundary components represents
the zero of degree $n$.
Thus every boundary component meeting $\GG(v_+)$
represents either the zero of degree $1$ or a simple pole.

We now consider one-by-one the four possible local ribbon graphs of valency $3$. In each case the fact that $v_+$ carries genus $1$ combined with parity restrictions from Figure~\ref{fig:embedded:local:ribbon:graphs} uniquely determines all boundary singularities. For each type the condition that every boundary component meeting $\GG(v_+)$ represents either the zero of degree $1$ or a simple pole is not satisfied. Namely, a local ribbon graph of type $+3.1$ develops a zero of order $3$; local ribbon graphs of types $+3.2a$ and $+3.3$ develop a zero of order $2$; a local ribbon graph of type $+3.2b$ brings us to genus $2$.

This excludes all four possible local ribbon graphs of valency $3$ for
the $\oplus$-vertex, which
proves Lemma~\ref{lm:no:oplus:vertices:of:valence:3}.
\end{proof}

\begin{figure}[htb]
  \centering
\vspace*{10pt}
\includegraphics[scale=0.43]{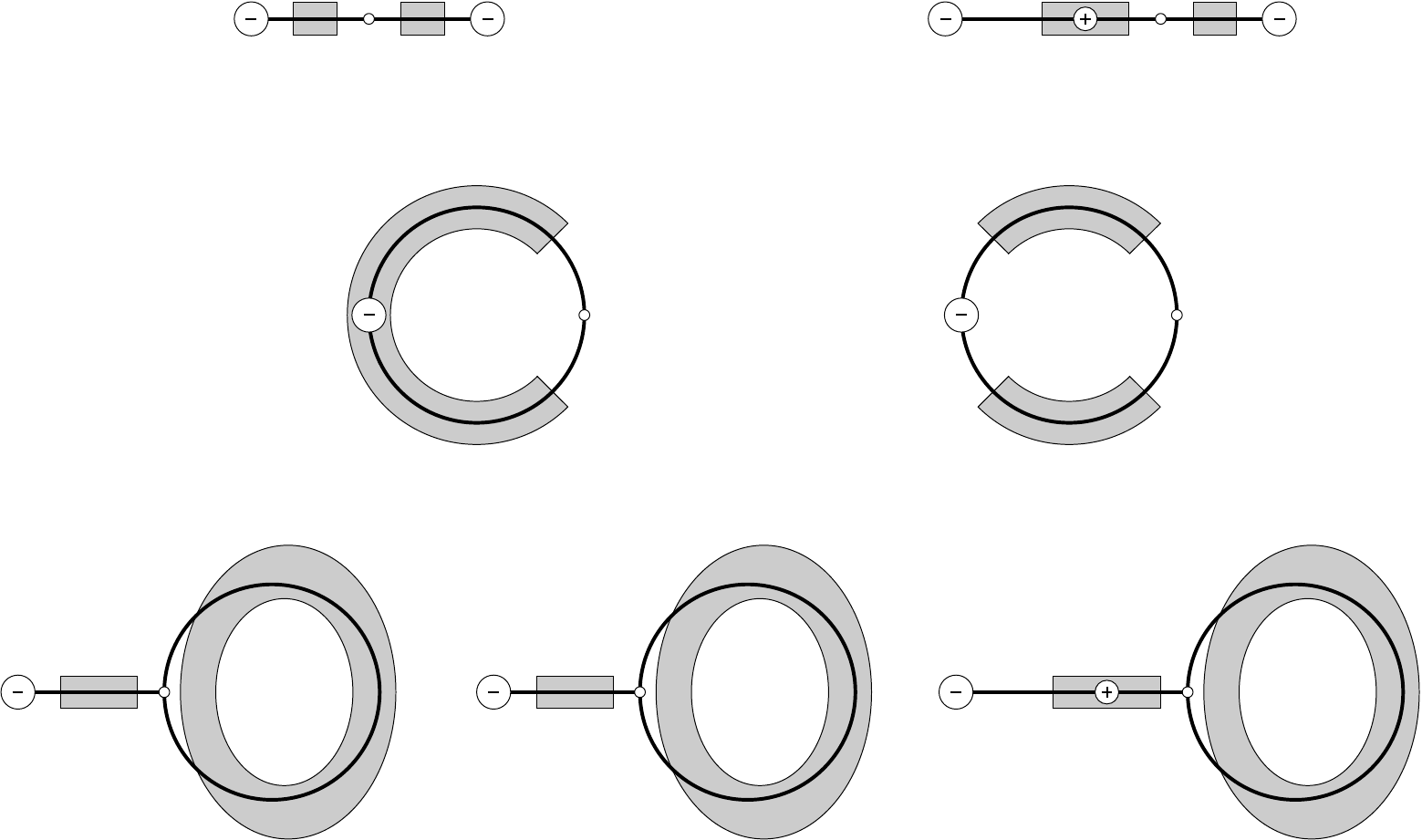}
\begin{picture}(0,0)(467,-217) % (430,-217)
\put(190,-20){$\scriptstyle g=0$}
\put(169,-32.5){$\scriptstyle \{-1^3\}$}
\put(203,-36.5){$\scriptstyle 1$}
\put(218,-36.5){$\scriptstyle 0$}
\put(227,-36.5){$\scriptstyle 0$}
\put(242,-36){$\scriptstyle n$}
\put(245,-20){$\scriptstyle g=1$}
\put(257,-32.5){$\scriptstyle \{-1^{n-2}\}$}
\put(223,-50){$i$}
\end{picture}
\begin{picture}(0,0)(300,-217)
\put(178,-20){$\scriptstyle g=0$}
\put(148,-32.5){$\scriptstyle \{-1^{n-2}\}$}
\put(192,-36.5){$\scriptstyle n-4$}
\put(227.5,-36.5){$\scriptstyle 0$}
\put(236.5,-36.5){$\scriptstyle 0$}
\put(252,-36.5){$\scriptstyle 1$}
\put(253.5,-20){$\scriptstyle g=0$}
\put(267,-32.5){$\scriptstyle \{-1^3\}$}
\put(231,-50){$ii$}
\end{picture}
\begin{picture}(0,0)(431,-167)
\put(227,-33){$\scriptstyle 0$}
\put(227,-67){$\scriptstyle 0$}
\put(162,-40){$\scriptstyle g=0$}
\put(168,-50){$\scriptstyle n'$}
\put(135,-60){$\scriptstyle \{1,-1^{n+1}\}$}
\put(189,-50){$\scriptstyle n''$}
\put(200,-87){$iii$}
\end{picture}
\begin{picture}(0,0)(300,-167)
\put(227,-33){$\scriptstyle 0$}
\put(227,-67){$\scriptstyle 0$}
\put(162,-40){$\scriptstyle g=0$}
\put(181,-33){$\scriptstyle 1$}
\put(144,-49.5){$\scriptstyle \{-1^{n+1}\}$}
\put(181,-67){$\scriptstyle n$}
\put(200,-87){$ iv$}
\end{picture}
\begin{picture}(0,0)(536.5,-65)
\put(193,-22){$\scriptstyle g=1$}
\put(190,-47){$\scriptstyle \{1,-1^{n-1}\}$}
\put(205.5,-37){$\scriptstyle n$}
\put(228.5,-37.5){$\scriptstyle 0$}
\put(238,-34){$\scriptstyle 0$}
\put(247,-34){$\scriptstyle 0$}
\put(260,-74){$v$}
\end{picture}
\begin{picture}(0,0)(431,-65) % (421,-65)
\put(192,-22){$\scriptstyle g=1$}
\put(185,-47){$\scriptstyle \{n,-1^{n-1}\}$}
\put(206,-38){$\scriptstyle 1$}
\put(228.5,-38){$\scriptstyle 0$}
\put(238,-34){$\scriptstyle 0$}
\put(246,-34){$\scriptstyle 0$}
\put(258,-74){$vi$}
\end{picture}
\begin{picture}(0,0)(315,-65) %
\put(185,-22){$\scriptstyle g=0$}
\put(175,-47){$\scriptstyle \{1,-1^{n-1}\}$}
\put(192.5,-37.5){$\scriptstyle n-4$}
\put(218.5,-26){$\scriptstyle 1$}
\put(218.5,-42){$\scriptstyle 1$}
\put(233,-37.5){$\scriptstyle 0$}
\put(243,-34){$\scriptstyle 0$}
\put(251,-34){$\scriptstyle 0$}
\put(260,-74){$vii$}
\end{picture}
\vspace*{10pt}
\caption{Stable ribbon graphs contributing to the area Siegel--Veech constant $\carea(\cQ_1(n,1,-1^{n+1}))$ with $n\ge 5$.}
\label{fig:graphs:n:1:genus:1}
\end{figure}

\begin{proposition}
\label{prop:configurations:with:two:zeroes}
A stable ribbon graph contributes to $\carea(\cQ_1(n,1,-1^{n+1}))$,
with $n\ge 5$, if and only if it is one of the
seven graphs represented in Figure~\ref{fig:graphs:n:1:genus:1}.
\end{proposition}
\begin{proof}
The surface again has at least $5$ poles, so Lemma~\ref{lm:abc:only} applies. The global ribbon graph contains at most two boundary components
not counting boundary components with a single boundary singularity
of degree $0$, see Lemma~\ref{lm:simple:pole:in:a:graph}.
The latter boundary components of the global ribbon graph do not appear in graphs of types ``a'' and  ``b''.

The local structure around a $\circ$-vertex of valency $2$
(see Figure~\ref{fig:embedded:local:ribbon:graphs})
implies that a global ribbon graph of type ``a'' containing
a $\circ$-vertex
has at least two boundary components, and each of them
represents a zero (and not a marked point or a pole).
To represent the stratum $\cQ(n,1,-1^{n+1})$ it should contain exactly two components representing zeroes of orders $n$ and $1$. The boundary component representing the zero of order $1$ should have exactly two boundary singularities of degrees $0$ and $1$, which implies that it corresponds to an edge joining the $\circ$-vertex and one of the $\ominus$-vertices.
This implies that the corresponding $\ominus$-vertex cannot carry genus $1$. Indeed, in genus $1$ this $\ominus$-vertex would have to have a
single interior singularity of degree $1$, which is forbidden
in the first line of the exceptional list in part 6 of Definition~\ref{def:configuration}. (In other terms, this would
give rise to a surface in the boundary stratum $\cQ(1,-1)$, which is empty.) Thus, this $\ominus$ vertex carries genus $0$, has
a single boundary singularity of order $1$, and three simple poles as interior singularities.

The graph can have either no $\oplus$-vertices, in which case the remaining $\ominus$-vertex carries genus $1$, or can have exactly one $\oplus$-vertex of type $+2.1$ carrying genus $1$ in which case the remaining $\ominus$-vertex carries genus $0$. The restrictions listed above leave us with two ribbon graphs of type ``a''
represented in Figure~\ref{fig:graphs:n:1:genus:1}.

Let us pass to graphs of type ``b'' containing a $\circ$-vertex.
A $\circ$-vertex of valency $2$ has local type $\circ2.2$, see Figure~\ref{fig:embedded:local:ribbon:graphs}. It contributes
to the genus, so in the case of $g=1$, the graph cannot have
$\oplus$-vertices (and, hence, other $\circ$-vertices),
and the single $\ominus$-vertex carries genus $0$.
If the $\ominus$-vertex has local type $-2.2$ in the notation
of Figure~\ref{fig:embedded:local:ribbon:graphs}, then the
global ribbon graph has two boundary components, and
the decoration is completely determined.
If the $\ominus$-vertex has local type $-2.1$, the
global ribbon graph has a single boundary component
with four boundary singularities. The degrees of
the boundary singularities adjacent to the vertex
of type $\circ2.2$
are equal to zero. Hence,
by Formula~\eqref{eq:degree:of:singularity} the associated zero
has degree at least $2$. Thus, in the case of this graph the zero of degree $1$ is necessarily an interior singularity. This exhausts the list of possible graphs of type ``b'' leaving us with the two graphs
presented in Figure~\ref{fig:graphs:n:1:genus:1}.

It remains to list the graphs of type ``c'' which contribute
to $\carea(\cQ(n,1,-1^{n+1}))$ with $n\ge 5$. By Lemma~\ref{lm:no:oplus:vertices:of:valence:3}
the unique vertex of valency $3$ of such a graph should be
a $\circ$ vertex, and hence, have the only possible local type
$\circ3.2$, see Figure~\ref{fig:embedded:local:ribbon:graphs}.

We start by proving that the loop of such a graph does not have
further vertices. First note that the presence of a $\oplus$-vertex at the loop excludes any other vertices at the loop. Indeed, a combination of a $\circ$-vertex and of a $\oplus$-vertex of the loop implies
that the genus $g\ge 2$. Suppose that there is a single $\oplus$-vertex on the loop. It should carry genus $1$. If it is of type $+2.1$
it has two boundary singularities of order $1$; such a vertex develops two zeroes of order $1$ for the ambient stratum.
 If the loop contains a single $\oplus$-vertex of local type $+2.2$,
then the associated surface has genus at least $2$.
The remaining a priori possibility for the loop to contain vertices is to have a $\circ$-vertex of valency $2$ as in Figure~\ref{fig:extra:graph:c:with:no:plus:vertices}. The orders of all four boundary singularities on the corresponding
boundary component of the global ribbon graph are zero, so Formula~\eqref{eq:degree:of:singularity} gives an endpoint singularity of order $-2+4\cdot(0+1)=2$, which is not allowed for $n\ge 5$.

We have proved that the loop does not contain any further vertices.
Hence, the global ribbon graph has a component represented by an annulus around the loop. Each of the two boundary components of the annulus has a single boundary singularity of degree $0$. These two components correspond to simple poles.

If a $\oplus$-vertex occurs on the stem, it carries genus $1$
and the corresponding local ribbon graph is of one of the types $+2.1$ or $+2.2$. Classification of graphs forbids presence of a $\circ$-vertex
on a stem next to a trivalent $\circ$-vertex, so the $\oplus$-vertex and the trivalent $\circ$-vertex are neighbors.
For type $+2.2$, the orders of two boundary singularities
of a $\oplus$-vertex carrying genus $1$ are $2,2$.
Thus, the boundary of the component of the global ribbon graph $\GG$ bounded between the $\oplus$-vertex and the $3$-valent $\circ$-vertex has exactly two boundary singularities
of orders $2$ and $0$, and hence represents a zero of degree $2$.
This is incompatible with the prescribed stratum for $n\ge 5$,
whether or not another cylinder occurs on the other side.
For a vertex of type $+2.1$ carrying genus $1$, the orders of
boundary singularities are $1,1$.
An additional cylinder between this vertex and the
$\ominus$-vertex would create a boundary component with four boundary singularities of
orders $1,1,0,0$, hence a zero of degree $4$.
Thus this additional cylinder is excluded as well.
The remaining graph with a single vertex of type $+2.1$ on the stem is the pocket-through-torus
configuration. Together with the two ordinary pockets, this
leaves precisely the three graphs $v, vi, vii$ of type ‘‘c’’.
\end{proof}

% ###################################################################
% ###################################################################
% ###################################################################
\section{Relations for renormalized volumes. Combinatorial constant}
\label{s:from:graphs:to:principal:boundary}

In sections~\ref{s:Contributions:of:graphs:a}, \ref{s:Contributions:of:graphs:b}, \ref{s:Contributions:of:graphs:c} we study in more detail contributions of graphs of types ``a'', ``b'', ``c''
to the area Siegel--Veech constant $\carea(\cQ_g(n,d,-1^{n+D-4g+4})$
detected in section~\ref{s:Configurations}.
Each time we apply first Goujard's formula~\eqref{eq:carea},
and then rewrite the answer in terms of
renormalized volumes $W_g(n;d)$ and renormalized Siegel--Veech constants defined in respectively~\eqref{eq:W:g:n} and~\eqref{eq:C:g:n}.
We do it twice, because normalization of $W_g(n;d)$ is needed for the
purposes of the current paper, while the standard normalization
might be useful for applications.

% ------------------------------------------------------------------------
\subsection{Some relations for renormalized volumes}

In this section we record normalization identities which will be
repeatedly used in sections~\ref{s:Contributions:of:graphs:a}--\ref{s:Contributions:of:graphs:c} when
passing to renormalized volumes. In particular we will repeatedly use the following identities which immediately follow from Equation~\eqref{eq:v} defining $v(n)$:
\begin{align}
\label{eq:v:n:minus:2}
\frac{v(n-2)}{v(n)}&=\frac{1}{\pi^2}\cdot\frac{n+1}{n}\,,
\\
\label{eq:v:n:minus:4}
\frac{v(n-4)}{v(n)}&=\frac{1}{\pi^4}\cdot\frac{(n+1)(n-1)}{n(n-2)}\,,
\\
\label{eq:v:n:minus:6}
\frac{v(n-6)}{v(n)}&=\frac{1}{\pi^6}\cdot\frac{(n+1)(n-1)(n-3)}{n(n-2)(n-4)}\,.
\end{align}

We will also use the following fact.
\begin{lemma}
The renormalized volume of any stratum in genus zero has the following constant value:
\begin{equation}
\label{eq:W:in:genus:0}
W_0(n;d)=2\pi^2\,.
\end{equation}
\end{lemma}
\begin{proof}
This relation is a trivial consequence of the combination of
Definition~\eqref{eq:W:g:n} of $W_0$ and of the closed Formula~\eqref{eq:volume:in:genus:0}.
\end{proof}

We will now extend the definition of $W_g(n;d_1,\dots,d_k)$
allowing the partition $d$ to contain entries $0$ and $-1$.
Note that the strata $\cQ_g(n,d_1,\dots,d_k,-1^{n+D-4g+4})$ were originally defined for positive integers $d_i$. This definition can be naturally extended allowing the variables $d_j$ to take values $0$ and $-1$. When $d_j=0$ for some $j$, then by~\eqref{eq:v:for:small:n} we have $v(0)=2$, and hence Equation~\eqref{eq:vol:with:marked:point} implies that
\begin{equation}
\label{eq:W:of:0}
W_g(n;d_1,\dots,d_{j-1},0,d_{j+1},\dots,d_k):=
W_g(n;d_1,\dots,d_{j-1},d_{j+1},\dots,d_k)\,.
\end{equation}

When $d_j=-1$ for some $j$, then
\begin{multline*}
\cQ_g(n;d_1,\dots,d_{j-1},-1,d_{j+1},\dots,d_k,-1^{n+D-4g+4})
\\=
\cQ_g(n;d_1,\dots,d_{j-1},d_{j+1},\dots,d_k,-1^{n+\hat{D}-4g+4})\,,
\end{multline*}
where
\begin{align*}
       D  &= d_1+\dots+d_{j-1}+(-1)+d_{j+1}+\dots+d_k\,, \\
  \hat{D} &= d_1+\dots+d_{j-1}+d_{j+1}+\dots+d_k\,,\\
  \hat{D}& = D+1\,.
\end{align*}
Thus, it is natural to define
\begin{equation}
\label{eq:W:of:minus:1}
W_g(n;d_1,\dots,d_{j-1},-1,d_{j+1},\dots,d_k):=
W_g(n;d_1,\dots,d_{j-1},d_{j+1},\dots,d_k)\,.
\end{equation}

%------------------------------------------------------------------
\subsection{Combinatorial constant in Goujard's Formula}
\label{ss:Formula:for:the:combinatorial:constant}

\noindent\textbf{Thin and thick cylinders.}
Goujard's Formula~\eqref{eq:carea} uses notation $q_1$ for the
number of ``thin'' cylinders and $q_2$ for the
number of ``thick'' cylinders. For the purposes of the current paper,
it will be convenient to define ``thin'' and ``thick'' cylinders
in terms of the classification of the associated stable graphs
given in Figure~\ref{fig:graphs}.

\begin{lemma}[Thin and thick cylinders]
\label{lm:thin:and:thick:cylinders}
A cylinder involved in a configuration of \^homologous saddle connections
is called \textit{thin} if and only if the associated $\circ$-vertex satisfies one of the following conditions:\newline
-- it is located at a graph of one of the types ``a, b''; \newline
-- it has valency $2$ and is located at one of the loops of graphs of types ``c, d, e''.

All other cylinders are called \textit{thick}. In other words, a cylinder is called \textit{thick} if and only if the associated $\circ$-vertex satisfies one of the following conditions:\newline
-- it has valency $3$ or $4$; \newline
-- it is located between the unique $\ominus$-vertex of a graph of type ``c'' and the unique vertex of valency $3$ of this graph; \newline
-- it is located between the two vertices of valency $3$ of a graph of type ``d''.
\end{lemma}
\begin{proof}
By definition, a cylinder is ``thin'' if the corresponding $\circ$-vertex has valency $2$ and is located on an edge labeled by $1$ in Figure~\ref{fig:graphs}; otherwise it is ``thick''. The classification
of graphs in Figure~\ref{fig:graphs} implies that this definition
is equivalent to the statement of the Lemma.
\end{proof}

We now recall the algorithm from~\cite{Goujard:carea} of evaluation of the combinatorial constant $M$ present in Goujard's key Formula~\eqref{eq:carea} for the area Siegel--Veech constant.
The constant $M$ is defined in~\cite{Goujard:carea} as
$$
M=\frac{M_s M_c}{M_t}\,,
$$
where $M_c, M_t, M_s$ are defined by Equations (9), (12) and (20)
in~\cite{Goujard:carea}:
\begin{equation*}
%\tag{9}
M_c=\begin{cases}
4^q&\text{if $\cC$ is of type }``a''\,,\\
4^{q+1}&\text{otherwise}\,.
\end{cases}
\end{equation*}

Next, $M_t$ is defined as
\begin{equation*}
%\tag{12}
M_t=\prod_{i=q_1+1}^q o_{t_i}\,.
\end{equation*}
This constant is responsible for the thick cylinders having
a figure eight as at least one of the boundaries. We have
$$
o(t_i)=\begin{cases}
2 & \text{if at least one of the sides}\\
  & \text{of the cylinder is a figure eight\,,}\\
1& \text{otherwise}
\end{cases}
$$

Finally,
\begin{equation*}
%\tag{Gj(20)}
  M_s=\frac{K}{|\Gamma(\cC)|}\,,
\end{equation*}
where the numerator and the denominator are defined as follows.
For every vertex $v_i$
we compute $K(v_i)$ as the following product over all components
of the local ribbon graph adjacent to the vertex $v_i$.
$$
K(v_i)=
\prod_{\substack{components\\of\ the\ local\\ribbon\ graph}}\quad
\sum_{\substack{boundary\\singularities\\in\ the\\component}} k_i
\quad
\text{if $v_i$ is a $\ominus$-vertex}\,,
$$
and
$$
K(v_i)=
2\prod_{\substack{components\\of\ the\ local\\ribbon\ graph}}\quad
\sum_{\substack{boundary\\singularities\\in\ the\\component}} \frac{k_i}{2}
\quad
\text{if $v_i$ is a $\oplus$-vertex}\,.
$$
We define
$$
K=\prod_{i=1}^m K(v_i)\,.
$$
Finally, $|\Gamma(\cC)|$ is defined as
$$
|\Gamma(\cC)|=\prod_{i=1}^m |\Gamma(v_i)|\,,
$$
where
\begin{equation}
\label{eq:Gamma}
|\Gamma(v_i)|=\begin{cases}
                2, & \text{if $S(v_i)$ is a torus of trivial holonomy} \\
                2, & \text{if $S(v_i)$ is in one of }\cQ(2,-1,-1),\ \cQ(2,2),\ \cQ(2,1,1)\\
                   & \text{and the surgery applies to one or two fixed points}\\
                   & \text{of the hyperelliptic involution} \\
                1, & \mbox{otherwise}\,.
              \end{cases}
\end{equation}

% ###################################################################
% ###################################################################
% ###################################################################
\section{Contribution of configurations of type ``a''}
\label{s:Contributions:of:graphs:a}

% ---------------------------------------------------------------------
\subsection{Dumbbell configuration}
We start with the configuration encoded by the decorated ribbon graph as in the top picture in Figure~\ref{fig:dumbbell}. In genus $0$ this configuration was called a ``dumbbell'' in~\cite{Athreya:Eskin:Zorich} and the corresponding Siegel--Veech constant was computed. We apply Goujard's Formula~\eqref{eq:carea} in the case of arbitrary genus.
\medskip

\begin{figure}[htb]
\centering
\includegraphics[scale=0.8]{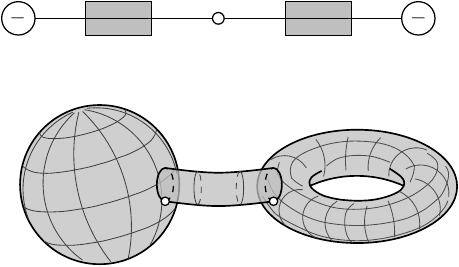}
\begin{picture}(0,0)(370,-150) % (370,-140)
\put(185,-45){$\scriptstyle g_0$} % (192,-45)
\put(211,-61){$\scriptstyle d_{i_s}$}
\put(250,-61){$\scriptstyle 0$}
\put(293,-61){$\scriptstyle 0$}
\put(326,-61){$\scriptstyle n$}
\put(115,-75){$\scriptstyle \{d_{i_1},\dots,d_{i_{s-1}},d_{i_{s+1}},\dots,d_{i_{k_0}},-1^{D_0-4g_0+2}\}$}
\put(355,-45){$\scriptstyle g_1$} % (348,-45)
\put(335,-75){$\scriptstyle \{d_{j_1},\dots,d_{j_{k_1}},-1^{n+D_1-4g_1+2}\}$}
\end{picture}
\caption{
\label{fig:dumbbell}
Dumbbell configuration. The bottom picture illustrates a particular case
when $g_0=0$ and $g_1=1$}
\end{figure}

We start our consideration with a dumbbell having a zero of degree $n$ at the base.

Given a stratum
\(
\cQ_g(n,d_1,\dots,d_k,-1^{n+D-4g+4})
\)
consider a partition of the set with multiplicities $\{d_1,\dots,d_k\}$
into two complementary subsets with multiplicities
\[
\{d_1,\dots,d_k\}=
\{d_{i_1},\dots,d_{i_{k_0}}\}\sqcup
\{d_{j_1},\dots,d_{j_{k_1}}\}\,,
\]
where
\[
\{1,\dots,k\}=
\{i_1,\dots,i_{k_0}\}\sqcup
\{j_1,\dots,j_{k_1}\}\,,
\]
and a partition of $g$ into a sum $g=g_0+g_1$ of nonnegative integers.
Let
\begin{align*}
  D_0 &= d_{i_1}+\cdots+d_{i_{k_0}}\,,
\\
  D_1 &= d_{j_1}+\cdots+d_{j_{k_1}}\,,
\end{align*}
which implies that $D_0+D_1=D$. We consider only those partitions which satisfy
\begin{align*}
\begin{cases}
D_0\ge 4g_0-2\,,\\
n+D_1\ge 4g_1-2\,.
\end{cases}
\end{align*}

\begin{proposition}
\label{prop:dumbbell:general:init}
Assume that the ambient stratum and both boundary strata displayed below are nonempty, that the boundary data are admissible, and that the conventions of Remark~\ref{rm:connected:strata} apply. Retain both symmetry factors $|\Gamma(v_0)|$ and $|\Gamma(v_1)|$ from~\eqref{eq:Gamma}.

The total contribution of dumbbell configurations
as in Figure~\ref{fig:dumbbell}
has the following form:
\begin{multline}
\label{eq:dumbbell:general:init}
c_{area}^{dumbbell,n,d_{i_s}}\big(\cQ_g(n,d_1,\dots,d_k,-1^{n+D-4g+4})\big)
=\frac{1}{|\Gamma(v_0)|\,|\Gamma(v_1)|}
\\
\times\frac{d_{i_s}\cdot n}{4}
\cdot
\frac{(D_0+k_0-2g_0-1)!}{(D_0-4g_0+2)!}
\cdot
\frac{(n+D_1+k_1-2g_1)!}{(n+D_1-4g_1+2)!}
\cdot
\frac{(n+D-4g+4)!}{(n+D+k-2g+2)!}
\\
\times
\Vol\cQ_{g_0}(d_{i_1},\dots,d_{i_{s-1}},d_{i_s}-2,d_{i_{s+1}},\dots,d_{i_{k_0}},-1^{D_0-4g_0+2})
\\
\times
\frac{\Vol\cQ_{g_1}(n-2,d_{j_1},\dots,d_{j_{k_1}},-1^{n+D_1-4g_1+2})}
{\Vol\cQ_g(n,d_1,\dots,d_k,-1^{n+D-4g+4})}\,.
\end{multline}
\end{proposition}
\begin{proof}
We consider a graph of type ``a'' with two $\ominus$-vertices
$v_0,v_1$ both
of types $-1.1$,
with a single $\circ2.2$ vertex between them and with no $\oplus$-vertices.
The boundary singularity at the vertex $v_0$ is $d_{i_s}$; the interior singularities are
\[
\{d_{i_1},\dots,d_{i_{s-1}},\widehat{d_{i_s}},
d_{i_{s+1}},\dots,d_{i_{k_0}},
-1^{D_0+2-4g_0}\}\,.
\]
The boundary singularity at the vertex $v_1$ is $n$; the interior singularities are
\[
\{d_{j_1},\dots,d_{j_{k_1}},-1^{n+D_1-4g_1+2}\}\,.
\]
The dimensions of the original stratum and of the two principal boundary strata
have the following form:
\begin{align*}
a&
= \dim_{\C}\cQ_g(n,d_1,\dots,d_k,-1^{n+D-4g+4})
= n+D+k-2g+3,
\\
a_0
&
= \dim_{\C}\cQ_{g_0}
\bigl(
d_{i_1},\dots,d_{i_{s-1}},d_{i_s}-2,
d_{i_{s+1}},\dots,d_{i_{k_0}},
-1^{D_0+2-4g_0}
\bigr)
= D_0+k_0-2g_0,
\\
a_1
&= \dim_{\C}\cQ_{g_1}
\bigl(
n-2,d_{j_1},\dots,d_{j_{k_1}},
-1^{n+D_1-4g_1+2}
\bigr)
= n+D_1+k_1-2g_1+1\,.
\end{align*}
For this configuration we have
\[
q_1=1\,,\qquad
q_2=0\,,\qquad
m=2\,,
\]
so Goujard's Formula~\eqref{eq:carea} gives:
\begin{multline*}
c_{area}^{\textit{dumbbell},n,d_{i_s};\textit{fixed poles}}
\big(\cQ_g(n,d_1,\dots,d_k,-1^{n+D-4g+4})\big)
 \\=
 M\cdot\frac{4\cdot 1+0}{2^{2+1+3}}\cdot
 \frac{(D_0+k_0-2g_0-1)!\cdot(n+D_1+k_1-2g_1)!}{(n+D+k-2g+2)!}
 \\
 \times
 \Vol\cQ_{g_0}(d_{i_1},\dots,d_{i_{s-1}},d_{i_s}-2,d_{i_{s+1}},\dots,d_{i_{k_0}},-1^{D_0-4g_0+2})
 \\
 \times
 \frac{\Vol\cQ_{g_1}(n-2,d_{j_1},\dots,d_{j_{k_1}},-1^{n+D_1-4g_1+2})}
 {\Vol\cQ_g(n,d_1,\dots,d_k,-1^{n+D-4g+4})}\,.
\end{multline*}

Let us compute now the constant $M$ for this configuration (see Section~\ref{ss:Formula:for:the:combinatorial:constant}). We have $m=2$. Since we have a diagram of type ``a'', we get
\[
M_c =4^1\,,\qquad M_t =1\,.
\]
We have $K(v_1)=n$. Its symmetry factor is given by~\eqref{eq:Gamma}; in particular, if the associated boundary stratum has at least three simple poles, then
\[
K(v_1)=(0+n)=n\,,\qquad |\Gamma(v_1)|=1\,,
\]
We have
\[
K(v_0)=(0+d_{i_s})=d_{i_s}\,.
\]
Finally, $|\Gamma(v_0)|$ is given by Formula~\eqref{eq:Gamma}.
In our language it translates as
\begin{equation}
\label{eq:Gamma:S:0}
|\Gamma(v_0)|
=\begin{cases}
2, & \text{if}\quad g_0=1,\quad k_0=1,\\
&\text{and}\quad (d_{i_1},\dots,d_{i_{k_0}})=(4)\,;\\
2, & \text{if}\quad g_0=2,\quad k_0=2,\quad d_{i_s}=4,\\
&\text{and}\quad  (d_{i_1},\dots,d_{i_{s-1}}, d_{i_{s+1}},\dots,d_{i_{k_0}})=(2)\,,\\
2, & \text{if}\quad g_0=2,\quad k_0=3,\quad d_{i_s}=4,\\
&\text{and}\quad  (d_{i_1},\dots,d_{i_{s-1}}, d_{i_{s+1}},\dots,d_{i_{k_0}})=(1,1)\,.\\
1, & \mbox{otherwise}\,.
\end{cases}
\end{equation}
Hence $M_S=\frac{K(v_0)}{|\Gamma(v_0)|}\cdot\frac{K(v_1)}{|\Gamma(v_1)|}
=\frac{d_{i_s}\cdot n}{|\Gamma(v_0)|\,|\Gamma(v_1)|}$
and thus
\begin{equation*}
%\label{eq:M:Aa}
M=\frac{M_S M_c}{M_t}
=4\cdot \frac{d_{i_s}\cdot n}{|\Gamma(v_0)|\,|\Gamma(v_1)|}\,.
\end{equation*}

There are
\begin{equation}
\label{eq:binomial}
\binom{n+D-4g+4}{D_0-4g_0+2}=\frac{(n+D-4g+4)!}{(D_0-4g_0+2)!(n+D_1-4g_1+2)!}
\end{equation}
ways to choose $D_0-4g_0+2$ simple poles which go to the surface of genus $g_0$ out of $n+D-4g+4$ initial poles,
so we get the contribution~\eqref{eq:dumbbell:general:init}
of all such dumbbell configurations.
The proof of Proposition~\ref{prop:dumbbell:general:init} is completed.
\end{proof}

% ------------------------------------------------------------------
\subsection{Dumbbell configurations carrying dominating contributions}

\begin{proposition}
\label{prop:distinguished:dumbbell}
Consider a special dumbbell configuration as in Figure~\ref{fig:dumbbell} with fixed $(g,d)$ and with an admissible decoration. Let $g_0=0$ and $d_{i_s}=d_i$. Assume, moreover, that the genus zero component has only simple poles (no zeroes) as interior singularities. The contribution of this specific dumbbell configuration has the following form in renormalized variables $W_g(n;d)$ and $C_g^\cC(n;d)$ defined in~\eqref{eq:W:g:n} and~\eqref{eq:C:g:n} as $n\to+\infty$:
\begin{multline}
\label{eq:dumbell:genus:g}
C_g^{dumbbell,n,d_i}(n;d_1,\dots,d_k)
\\
=
\frac{2}{d_i+2}\cdot
W_g(n-2;d_1,\dots,\widehat{d_i},\dots,d_k)
\left(1+O_{g,d}\left(\frac1n\right)\right)\,.
\end{multline}
\end{proposition}
\begin{proof}
Since $n\to+\infty$, the $\ominus$-vertex $v_1$ containing the zero of degree $n-2$ has at least three simple poles, so by~\eqref{eq:Gamma} $|\Gamma(v_1)|=1$. We have $g_0=0$, so by~\eqref{eq:Gamma:S:0} we have $|\Gamma(v_0)|=1$.
We have $i=i_s$. Thus
\[
D_0=d_i,
\qquad
D_1=D-d_i,
\qquad
g_1=g,
\qquad
k_1=k-1.
\]
The first factorial ratio becomes
\[
\frac{(D_0+k_0-2g_0-1)!}{(D_0-4g_0+2)!}
=
\frac{d_i!}{(d_i+2)!}
=
\frac{1}{(d_i+1)(d_i+2)}.
\]
The genus-zero component is
$\cQ_0(d_i-2,-1^{d_i+2})$,
so by Formula~\eqref{eq:Vol:minimal:g0} we get
\[
\Vol\cQ_0(d_i-2,-1^{d_i+2})
=
2\pi^2\cdot v(d_i-2).
\]
Using
\[
\frac{v(d_i-2)}{v(d_i)}
=
\frac{1}{\pi^2}\cdot
\frac{d_i+1}{d_i},
\]
we get
\[
\Vol\cQ_0(d_i-2,-1^{d_i+2})
=
2\cdot\frac{d_i+1}{d_i}\cdot v(d_i).
\]
Hence
\[
\frac{d_i n}{4}
\cdot
\frac{1}{(d_i+1)(d_i+2)}
\cdot
\Vol\cQ_0(d_i-2,-1^{d_i+2})
=
\frac{n\cdot v(d_i)}{2(d_i+2)}.
\]

Therefore, in the case $k_0=1$ and $g_0=0$, we obtain
\begin{multline*}
c_{area}^{dumbbell,n,d_i}
\big(\cQ_g(n,d_1,\dots,d_k,-1^{n+D-4g+4})\big)
\\
=
\frac{n\cdot v(d_i)}{2(d_i+2)}
\cdot
\frac{(n+D-d_i+k-1-2g)!}
{(n+D-d_i-4g+2)!}
\cdot
\frac{(n+D-4g+4)!}
{(n+D+k-2g+2)!}
\\
\times
\frac{
\Vol\cQ_g(n-2,d_1,\dots,\widehat{d_i},\dots,d_k,
-1^{n+D-d_i-4g+2})
}
{
\Vol\cQ_g(n,d_1,\dots,d_k,-1^{n+D-4g+4})
}.
\end{multline*}

Passing to the renormalized contribution
$C_g^{dumbbell,n,d_i}$
defined in~\eqref{eq:C:g:n} we get
\begin{multline*}
C_g^{dumbbell,n,d_i}(n;d_1,\dots,d_k)
\\
=
\frac{2(n+1)}{d_i+2}
\left(\frac{n-2}{n}\right)^g
\cdot
\frac{(n+D-d_i+k-1-2g)!}
{(n+D-d_i-4g+2)!}
\cdot
\frac{(n+D-4g+4)!}
{(n+D+k-2g+2)!}
\\
\times
W_g(n-2;d_1,\dots,\widehat{d_i},\dots,d_k).
\end{multline*}
Indeed,
\[
W_g(n-2;d_1,\dots,\widehat{d_i},\dots,d_k)
=
\frac{\Vol\cQ_g(n-2,d_1,\dots,\widehat{d_i},\dots,d_k,-1^{n+D-d_i-4g+2})}
{(n-2)^g\cdot v(n-2)\cdot \prod_{j\ne i}v(d_j)}\,.
\]
Therefore
\begin{multline*}
\frac{\Vol\cQ_g(n-2,d_1,\dots,\widehat{d_i},\dots,d_k,-1^{n+D-d_i-4g+2})}
{n^g\cdot v(n)\cdot \prod_{j\ne i}v(d_j)}
\\
=
\left(\frac{n-2}{n}\right)^g
\frac{v(n-2)}{v(n)}\cdot
W_g(n-2;d_1,\dots,\widehat{d_i},\dots,d_k).
\end{multline*}
Using one more time~\eqref{eq:v:n:minus:2} for
\(
\frac{v(n-2)}{v(n)}
=
\frac{1}{\pi^2}\cdot
\frac{n+1}{n}
\)
we obtain the factor
\(
\frac{2(n+1)}{d_i+2}
\left(\frac{n-2}{n}\right)^g
\)
in the renormalized expression above.

Finally, the product of the two factorial ratios satisfies
\[
\frac{(n+D-d_i+k-1-2g)!}
{(n+D-d_i-4g+2)!}
\cdot
\frac{(n+D-4g+4)!}
{(n+D+k-2g+2)!}
=
\frac{1}{n}
\left(1+O_{g,d}\left(\frac1n\right)\right).
\]
Proposition~\ref{prop:distinguished:dumbbell} is proved.
\end{proof}

%-------------------------------------------------------------------
\subsection{%
  \texorpdfstring{%
    \for{toc}{Contributions of graphs of type ``a'' in genus $1$}%
    \except{toc}{%
      \boldmath Contributions of graphs of type ``a'' in genus $1$%
    }%
  }{Contributions of graphs of type a in genus 1}%
}

Throughout this subsection, we assume $n\ge 5$.

\begin{lemma}
\label{lm:a:graphs:contributions:Q:n:1}
Contributions of the configurations corresponding to graphs
of types \textit{i} and \textit{ii} in Figure~\ref{fig:graphs:n:1:genus:1}
to the area Siegel--Veech constant of the strata $\cQ_1(n,1,-1^{n+1})$
have the following form in renormalized variables:
\begin{equation}
\label{eq:a:graphs:contributions:Q:n:1}
\begin{aligned}
&C_1^{dumbbell,n,1}(n;1)
=\frac{2(n+1)(n-2)}{3n(n+2)}\cdot W_1(n-2)\,,
\\
&C_1^{\textit{dumbbell-through-torus},n,1}(n;1)
=\frac{8(n+1)(n-1)}{9n^2(n-2)^2(n+2)}\,.
\end{aligned}
\end{equation}
\end{lemma}
\begin{proof}
By~\eqref{eq:dumbbell:general:init} we have
\begin{multline}
\label{eq:n:1:i}
c_{area}^{dumbbell,n,1}\big(\cQ_1(n,1,-1^{n+1})\big)
\\=
\frac{1}{24}
\cdot\frac{n}{(n+2)}
\cdot
\frac{\Vol\cQ_0(-1^4)\cdot\Vol\cQ_1(n-2,-1^{n-2})}
{\Vol\cQ_1(n,1,-1^{n+1})}
\,.
\end{multline}
From~\eqref{eq:W:g:n} we have
\[
\Vol\cQ_1(n-2,-1^{n-2})=(n-2)\cdot v(n-2)\cdot W_1(n-2)\,.
\]
Applying relation~\eqref{eq:v:n:minus:2} for the ratio $\frac{v(n-2)}{v(n)}$
and using the value $v(1)=\frac{\pi^2}{2}$ given in~\eqref{eq:v:for:small:n}, we get
\begin{multline*}
\frac{\Vol\cQ_0(-1^4)\cdot\Vol\cQ_1(n-2,-1^{n-2})}{n\cdot v(n)\cdot v(1)}
\\
=\frac{2\pi^2}{v(1)}
\cdot\frac{n-2}{n}
\cdot\frac{v(n-2)}{v(n)}
\cdot W_1(n-2)
=\frac{4(n+1)(n-2)}{\pi^2n^2}\cdot W_1(n-2).
\end{multline*}
Thus,
\begin{multline*}
%\label{eq:contribution:i}
C_1^{dumbbell,n,1}(n;1)
=
\frac{4\pi^2}{n\cdot v(n)\cdot v(1)}
\cdot\frac{n}{24(n+2)}
\\
\times\Vol\cQ_0(-1^4)\Vol\cQ_1(n-2,-1^{n-2})
=\frac{2(n+1)(n-2)}{3n(n+2)}\cdot W_1(n-2)\,.
\end{multline*}
The first equality in~\eqref{eq:a:graphs:contributions:Q:n:1} is proved.

In analogy to the graph \textit{i}, Goujard's Formula~\eqref{eq:carea} applied to the configuration represented by the graph \textit{ii} gives
\begin{multline}
\label{eq:n:1:ii}
c_{area}^{\textit{dumbbell-through-torus},n,1}\big(\cQ_1(n,1,-1^{n+1})\big)
\\=
\frac{n-4}{12(n-3)(n-2)(n+2)}
\cdot
\frac{\Vol\cQ_0(-1^4)\cdot\Vol\cH(0)\cdot\Vol\cQ_0(n-6,-1^{n-2})}
{\Vol\cQ_1(n,1,-1^{n+1})}
\,.
\end{multline}
Applying relation~\eqref{eq:v:n:minus:6} for the ratio $\frac{v(n-6)}{v(n)}$,
using the closed Formula~\eqref{eq:Vol:minimal:g0} for the volumes
of the minimal strata in genus $0$, and using the value $\Vol\cH(0)=\frac{\pi^2}{3}$, we get
\begin{multline*}
\frac{\Vol\cQ_0(-1^4)\cdot\Vol\cH(0)\cdot\Vol\cQ_0(n-6,-1^{n-2})}
{n\cdot v(n)\cdot v(1)}
\\
=\frac{4\pi^6}{3n\cdot v(1)}
\cdot\frac{v(n-6)}{v(n)}
=\frac{8(n+1)(n-1)(n-3)}
{3\pi^2n^2(n-2)(n-4)}\,.
\end{multline*}
Thus,
\begin{multline*}
%\label{eq:contribution:ii}
C_1^{\textit{dumbbell-through-torus},n,1}(n;1)
\\=
\frac{4\pi^2}{n\cdot v(n)\cdot v(1)}
\cdot\frac{n-4}{12(n-3)(n-2)(n+2)}
\\
\times
\Vol\cQ_0(-1^4)\cdot\Vol\cH(0)\cdot\Vol\cQ_0(n-6,-1^{n-2})
=\frac{8(n+1)(n-1)}{9n^2(n-2)^2(n+2)}\,.
\end{multline*}
Lemma~\ref{lm:a:graphs:contributions:Q:n:1} is proved.
\end{proof}

% ###################################################################
% ###################################################################
% ###################################################################
\section{Contributions of graphs of type ``b''}
\label{s:Contributions:of:graphs:b}

%--------------------------------------------------------------------
\subsection{Surface with a handle}

As in the previous section we start with the simplest graphs of type ``b''
contributing to the area Siegel--Veech constant, namely with those
having a single $\ominus$-vertex and a single $\circ$-vertex,
see Figure~\ref{fig:Ba:and:Bb}.
The $\ominus$-vertex has valency $2$, so it has one of the local types
$-2.1$ or $-2.2$ in terms of the classification of Figure~\ref{fig:embedded:local:ribbon:graphs}.
Analogous configurations exist for Abelian differentials and were studied in~\cite{Eskin:Masur:Zorich} where they were obtained by ``figure-eight'' and ``parallelogram'' constructions respectively and where similar Siegel--Veech constants were computed. We apply Goujard's Formula~\eqref{eq:carea} in the case of arbitrary genus.

\begin{figure}[htb]
\centering
\includegraphics[scale=0.4]{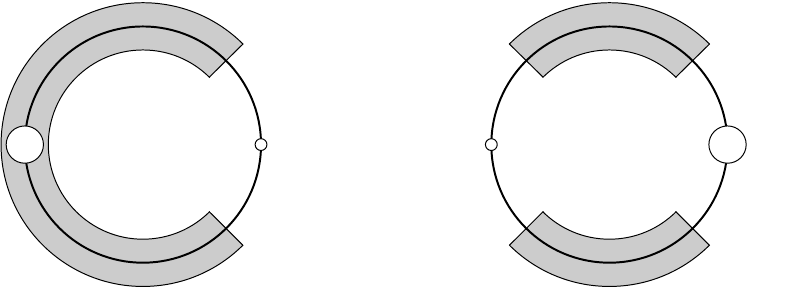}
\vspace*{5pt}
\begin{picture}(0,0)(332,-76)%(350,-76)
\put(95,-60){$\scriptstyle \{d_1\dots,d_k,-1^{n+D-4g+4}\}$}
%\put(160,-40){$\scriptstyle g-1$}
\put(225,-65){$\scriptstyle 0$}
\put(225,-35){$\scriptstyle 0$}
\put(167,-50.5){$\scriptstyle n'$}
\put(178.5,-50.5){$\scriptstyle -$}
\put(190,-50.5){$\scriptstyle n''$}
\put(175,-88){\small ``Figure-eight''}
\end{picture}
\begin{picture}(0,0)(250,-76)
\put(234,-60){$\scriptstyle\{d_1,\dots,\widehat{d_j},\dots,d_k,-1^{n+D-4g+4}\}$}
%\put(232,-42){$\scriptstyle g-1$}
\put(184,-65){$\scriptstyle 0$}
\put(184,-35){$\scriptstyle 0$}
\put(228.75,-50.5){$\scriptstyle -$}
\put(228.75,-65){$\scriptstyle n$}
\put(229,-35){$\scriptstyle d_j$}
\put(180,-88){\small ``Parallelogram''}
\end{picture}
\vspace*{5pt}
\caption{A surface of genus $g-1$ with an attached handle.}
\label{fig:Ba:and:Bb}
\end{figure}

Geometrically, the left graph in Figure~\ref{fig:Ba:and:Bb} represents the following configuration. Take a surface $S'$ in the stratum $\cQ_{g-1}(n-4,d,-1^{n+D-4g+4})$ and break the zero of degree $n-4$ into two zeroes joined by a saddle connection $\gamma$. (One of the newborn zeroes is, actually, allowed to be a marked point or even a pole.) Make a slit along $\gamma$ and join the endpoints. We get a surface with boundary in the shape of a figure-eight. Pasting a cylinder to this boundary we get a surface $S$ in the stratum $\cQ_g(n,d,-1^{n+D-4g+4})$.

Geometrically, the right graph in Figure~\ref{fig:Ba:and:Bb} represents the following configuration. Take a surface $S'$ in the stratum
\[
\cQ_{g-1}(n-2,d_1,\dots,d_{j-1},d_j-2,d_{j+1},\dots,d_k,-1^{n+D-4g+4})\,.
\]
Consider a saddle connection $\gamma$ joining the zeroes of degrees $n-2$ and $d_j-2$
and consider a small parallel translate $\gamma'$ of this saddle connection in a transverse direction. Remove the parallelogram spanned by $\gamma$ and $\gamma'$ and identify the pair of opposite sides represented by $\gamma$ and $\gamma'$
of the resulting hole of parallelogram shape.
The other pair of sides of the original parallelogram hole thus gives rise
to two distinct boundary components of the resulting surface.
Pasting a cylinder to this pair of boundary components we get a surface
$S$ in the stratum $\cQ_g(n,d,-1^{n+D-4g+4})$.

%------------------------------------------------------------------------------
\subsection{Figure-eight configuration}

We start with the case when the
$\ominus$-vertex has local type $-2.1$ as in the left picture
in Figure~\ref{fig:Ba:and:Bb} and when in addition, the zero of degree $n$ is \textit{not} an interior singularity.

\begin{proposition}
\label{prop:figure:eight:general:init}
The total contribution of admissible ``figure-eight'' configurations represented by the graphs as in the left picture in Figure~\ref{fig:Ba:and:Bb}
with $g\geq1$, $n\ge 3$, and $n+D-4g+4\ge 3$,
has the following form:
\begin{multline}
\label{eq:carea:figure:eight}
c_{area}^{\textit{figure-eight},n}\big(\cQ_g(n,d_1,\dots,d_k,-1^{n+D-4g+4})\big)
\\=
\frac{(n-2)(n-1)}{(n+D+k-2g+1)(n+D+k-2g+2)}
\cdot\frac{\Vol\cQ_{g-1}(n-4,d,-1^{n+D-4g+4})}{\Vol\cQ_g(n,d,-1^{n+D-4g+4})}\,,
\end{multline}
\end{proposition}
\begin{proof}
The assumptions of the Proposition imply that the principal boundary stratum has at least three simple poles. By Remark~\ref{rm:connected:strata} this implies connectedness of both the boundary stratum and the ambient stratum. By~\eqref{eq:Gamma} this also implies that $|\Gamma(v_1)|=1$.

The graph has a single $\ominus$-vertex; the associated local ribbon graph is of type $-2.1$ and has a single connected component with two boundary singularities of nonnegative degrees $n',n''$ satisfying $n'+n''=n-2$. The
global ribbon graph has a single connected component with four boundary singularities of degrees $n',n'', 0,0$.
Formula~\eqref{eq:Dji} gives degree
\[
\delta=n'+n''-2=n-4
\]
for the resulting singularity of the boundary stratum.
Together with interior singularities
$\{d_1\dots,d_k,-1^{n+D-4g+4}\}$ this gives us the boundary stratum
\[
\cQ(\alpha')=\cQ_{g-1}(n-4,d,-1^{n+D-4g+4})\,.
\]
We have
\[
\begin{array}{lll}
a&= \dim_{\C}\cQ_g(n,d_1,\dots,d_k,-1^{n+D-4g+4})&= n+D+k-2g+3\,,
\\
a_1&= \dim_{\C}\cQ_{g-1}(n-4,d_1,\dots,d_k,-1^{n+D-4g+4})&= n+D+k-2g+1\,,
\end{array}
\]
For this configuration we have a single cylinder, and it is thin, see Lemma~\ref{lm:thin:and:thick:cylinders}. We get the following values for the number $q_1$ of thin cylinders and the number $q_2$ of thick cylinders and for the number $m$ of $\oplus$ and $\ominus$-vertices counted together:
\[
q_1=1\,,\qquad
q_2=0\,,\qquad
m=1\,.
\]
Hence~\eqref{eq:carea} gives the following value for a fixed pair $n', n''$:
\begin{multline}
\label{eq:figure:eight:before:weight}
c_{area}^{\textit{figure-eight},n,\textit{fixed }n',n''}\big(\cQ_g(n,d_1,\dots,d_k,-1^{n+D-4g+4})\big)
\\=
M\cdot\frac{4\cdot 1+0}{2^{1+1+3}}\cdot
\frac{(n+D+k-2g)!}{(n+D+k-2g+2)!}
\cdot\frac{\Vol\cQ_{g-1}(n-4,d,-1^{n+D-4g+4})}{\Vol\cQ_g(n,d,-1^{n+D-4g+4})}\,.
\end{multline}

The degrees $n', n''$ of boundary singularities form part of the decoration of the global ribbon graph. However, by Goujard's Formula~\eqref{eq:carea}, for a given $n$ all partitions  $n'+n''=n-2$ contribute the same amount to the area Siegel--Veech constant. There are $n-1$ ordered partitions of $n-2$ into a sum of nonnegative integers. However, the global ribbon graphs with $(n', n'')$ and $(n'', n')$ are symmetric, so they represent the same configuration. Thus, taking a sum over all ordered partitions of $n-2$ we will count the contribution of each configuration twice. We get a factor
\begin{equation}
\label{eq:combinatorial:constant:figure:eight}
\frac{n-1}{2}\,.
\end{equation}

\begin{remark}
When $n$ is even, the global ribbon graph with $n'=n''=(n-2)/2$ has symmetry of order $2$ and the contribution of such a pair is not counted twice
compared to graphs with $n'\neq n''$. The weight $\frac{1}{2}$ with which we count this contribution is the inverse of the order of the symmetry group of the corresponding decorated ribbon graph.
\end{remark}

For this configuration we have
$$
  M_c =4^{1+1}\,,\qquad\text{and}\qquad
  M_t =1\,,
$$
see Section~\ref{ss:Formula:for:the:combinatorial:constant}, and
\[
K(v_1) =n'+n''=n-2\,,\qquad
  |\Gamma(v_1)|=1\,.
\]
Hence $M_s=\frac{K(v_1)}{|\Gamma(v_1)|}=n-2$ and we get
\[
M=\frac{M_s M_c}{M_t}
=\frac{(n-2)\cdot 2^4}{1}\,.
\]
Plugging this value of $M$ into~\eqref{eq:figure:eight:before:weight}
and multiplying the result by the combinatorial factor~\eqref{eq:combinatorial:constant:figure:eight}
we get
the following contribution of
all ``figure-eight'' configurations as in the left picture
in Figure~\ref{fig:Ba:and:Bb}:
\begin{multline*}
c_{area}^{\textit{figure-eight},n}\big(\cQ_g(n,d_1,\dots,d_k,-1^{n+D-4g+4})\big)
=\frac{n-1}{2}\cdot \big((n-2)\cdot 2^4\big)
\\
\times
\frac{2^{-3}}{(n+D+k-2g+1)(n+D+k-2g+2)}
\cdot\frac{\Vol\cQ_{g-1}(n-4,d,-1^{n+D-4g+4})}{\Vol\cQ_g(n,d,-1^{n+D-4g+4})}\,.
\end{multline*}
Proposition~\ref{prop:figure:eight:general:init} is proved.
\end{proof}

\begin{corollary}
\label{cor:figure:eight:general:renormalized}
Under the hypotheses of Proposition~\ref{prop:figure:eight:general:init}
with an extra condition $n\ge 5$,
the contribution of the ``figure-eight'' configurations represented by the graphs as in the left picture in Figure~\ref{fig:Ba:and:Bb} has the following form in renormalized variables:
\begin{multline}
\label{eq:figure:eight:genus:g:renormalized}
C_g^{\textit{figure-eight},n}(n;d_1,\dots,d_k)
\\=
\frac{4}{\pi^2}
\left(\frac{n-4}{n}\right)^{g-1}
\frac{(n+1)(n-1)^2}
{n^2(n+D+k-2g+1)(n+D+k-2g+2)}
W_{g-1}(n-4;d)\,.
\end{multline}
\end{corollary}
\begin{proof}
By Definition~\eqref{eq:W:g:n} of the renormalized volume,
\[
W_{g-1}(n-4;d)
=
\frac{\Vol\cQ_{g-1}(n-4,d_1,\dots,d_k,-1^{n+D-4g+4})}
{(n-4)^{g-1}\cdot v(n-4)\cdot v(d_1)\cdots v(d_k)}\,,
\]
where $d=(d_1,\dots,d_k)$ and $D=d_1+\cdots+d_k$.
Consequently,
\begin{multline*}
\frac{\Vol\cQ_{g-1}(n-4,d,-1^{n+D-4g+4})}
{n^g\cdot v(n)\prod_{j=1}^k v(d_j)}
=
\frac{(n-4)^{g-1}}{n^g}
\cdot
\frac{v(n-4)}{v(n)}
\cdot
W_{g-1}(n-4;d)
\\
=
\frac{1}{\pi^4}
\left(\frac{n-4}{n}\right)^{g-1}
\frac{(n+1)(n-1)}{n^2(n-2)}
W_{g-1}(n-4;d)
\,,
\end{multline*}
where we used~\eqref{eq:v:n:minus:4} for
$\frac{v(n-4)}{v(n)}$.

Applying Definition~\eqref{eq:C:g:n}  of the renormalized area Siegel--Veech constant, substituting the ratio in the left-hand side of the above relation by the expression in the right-hand side, and cancelling out the common factor $n-2$ in the numerator and in the denominator of the resulting expression, we obtain
\begin{multline*}
C_g^{\textit{figure-eight},n}(n;d)
=
4\pi^2\cdot W_g(n;d)\cdot
c_{area}^{\textit{figure-eight},n}
\bigl(\cQ_g(n,d,-1^{n+D-4g+4})\bigr)
\\
=
4\pi^2
\frac{(n-2)(n-1)}{(n+D+k-2g+1)(n+D+k-2g+2)}
\frac{\Vol\cQ_{g-1}(n-4,d,-1^{n+D-4g+4})}
{n^g\cdot v(n)\prod_{j=1}^k v(d_j)}
\\=
\frac{4}{\pi^2}
\left(\frac{n-4}{n}\right)^{g-1}
\frac{(n+1)(n-1)^2}
{n^2(n+D+k-2g+1)(n+D+k-2g+2)}
W_{g-1}(n-4;d)\,.
\end{multline*}
which completes the proof of Corollary~\ref{cor:figure:eight:general:renormalized}.
\end{proof}

%------------------------------------------------------------------------------
\subsection{Parallelogram configuration}

We now treat the case when the
$\ominus$-vertex has local type $-2.2$ as in the right picture
in Figure~\ref{fig:Ba:and:Bb} and when in addition, the zero of degree $n$ is \textit{not} an interior singularity.

\begin{proposition}
\label{prop:parallelogram:general:init}
Assume that all of $g,n,d_j$ are strictly positive and that $n+D-4g+4\ge 3$.
The contribution of ``parallelogram'' configurations
represented by the graphs as
in the right picture in Figure~\ref{fig:Ba:and:Bb}
with admissible decoration has the following form:
\begin{multline}
\label{eq:carea:parallelogram}
c_{area}^{\textit{parallelogram},n,d_j}\big(\cQ_g(n,d,-1^{n+D-4g+4})\big)
\\
=\frac{2d_j\cdot n}{(n+D+k-2g+1)(n+D+k-2g+2)}
\\
\times\frac{\Vol\cQ_{g-1}(n-2,d_1,\dots,d_{j-1},d_j-2,d_{j+1},\dots,d_k,-1^{n+D-4g+4})}
{\Vol\cQ_g(n,d,-1^{n+D-4g+4})}\,.
\end{multline}
\end{proposition}
\begin{proof}
The proof is analogous to the proof of Proposition~\ref{prop:figure:eight:general:init}.
We outline only the differences. This time we have
\[
\cQ(\alpha')=\cQ_{g-1}(n-2,d_1,\dots,d_{j-1},d_j-2,d_{j+1},\dots,d_k,-1^{n+D-4g+4})\,,
\]
We have the same values of dimensions $a$ and $a_1$ as in the proof of
Proposition~\ref{prop:figure:eight:general:init}, the same numbers $q_1$
and $q_2$ of
thin and thick cylinders and the same number $m$ of $\oplus$ and $\ominus$ vertices counted together, so Goujard's Formula~\eqref{eq:carea}
gives the contribution analogous to~\eqref{eq:figure:eight:before:weight}
with the principal boundary stratum changed.
This time there are no varying parameters similar to
$n'$ and $n''$ in the figure-eight case. The constants $M_c$ and $M_t$
do not change.

The constant $M_s$ changes. Now the local ribbon graph
has two distinct connected components; each component has a single boundary singularity. The orders of boundary singularities
are $n$ and $d_j$, so $K(v)=d_j\cdot n$. Since the boundary stratum
has at least $3$ simple poles, we have $|\Gamma(v_1)|=1$. Hence $M_s=\frac{K(v_1)}{|\Gamma(v_1)|}=d_j\cdot n$ and we get
\[
M=\frac{M_s M_c}{M_t}
=\frac{d_j\cdot n\cdot 2^4}{1}\,.
\]
Collecting, we get~\eqref{eq:carea:parallelogram}.
Proposition~\ref{prop:parallelogram:general:init} is proved.
\end{proof}

\begin{corollary}
%\label{cor:parallelogram:general:renormalized}
Assume the hypotheses of Proposition~\ref{prop:parallelogram:general:init} and $n\ge 3$. The contribution of the ``parallelogram'' configuration represented by the graph as in the right picture in Figure~\ref{fig:Ba:and:Bb} has the following form in renormalized variables:
\begin{multline}
\label{eq:parallelogram:genus:g:renormalized}
C_g^{\textit{parallelogram},n,d_j}(n;d)
\\
=
\frac{8}{\pi^2}
\left(\frac{n-2}{n}\right)^{g-1}
\frac{(n+1)(d_j+1)}
{n(n+D+k-2g+1)(n+D+k-2g+2)}
\\ \times
W_{g-1}(n-2;d_1,\dots,d_{j-1},d_j-2,d_{j+1},\dots,d_k)\,,
\end{multline}
where for $d_j=1,2$ the quantity $W_{g-1}(n-2;d_1,\dots,d_{j-1},d_j-2,d_{j+1},\dots,d_k)$ is defined in~\eqref{eq:W:of:minus:1} and~\eqref{eq:W:of:0} respectively.
\end{corollary}
\begin{proof}
Denote
\[
\hat d^{(j)}
:=
(d_1,\dots,d_{j-1},d_j-2,d_{j+1},\dots,d_k).
\]
By Definition~\eqref{eq:W:g:n} of the renormalized volume, we have
\[
W_{g-1}(n-2;\hat d^{(j)})
=
\frac{\Vol\cQ_{g-1}\big(n-2,\hat d^{(j)},-1^{n+D-4g+4}\big)}
{(n-2)^{g-1}\cdot v(n-2)\cdot v(d_j-2)\cdot
\prod_{\substack{1\leq i\leq k\\ i\neq j}}v(d_i)}\,.
\]
Consequently,
\begin{multline*}
\frac{\Vol\cQ_{g-1}\big(n-2,\hat d^{(j)},-1^{n+D-4g+4}\big)}
{n^g\cdot  v(n)\cdot \prod_{i=1}^k v(d_i)}
\\
=
\frac{(n-2)^{g-1}}{n^g}
\cdot \frac{v(n-2)}{v(n)}
\cdot \frac{v(d_j-2)}{v(d_j)}
\cdot W_{g-1}(n-2;\hat d^{(j)})\,.
\end{multline*}
Applying~\eqref{eq:v:n:minus:2} to both ratios of $v$ gives
\[
\frac{v(n-2)}{v(n)}
=
\frac{1}{\pi^2}\frac{n+1}{n}\,,
\qquad
\frac{v(d_j-2)}{v(d_j)}
=
\frac{1}{\pi^2}\frac{d_j+1}{d_j}.
\]
Hence
\begin{multline}
\label{eq:parallelogram:boundary:volume:ratio}
\frac{
\Vol\cQ_{g-1}(n-2,\hat d^{(j)},-1^{n+D-4g+4})
}{
n^g\cdot v(n)\cdot \prod_{i=1}^k v(d_i)
}
\\
=
\frac{1}{\pi^4}
\left(\frac{n-2}{n}\right)^{g-1}
\frac{(n+1)(d_j+1)}{n^2\cdot d_j}
\cdot W_{g-1}(n-2;\hat d^{(j)})\,.
\end{multline}

By Definition~\eqref{eq:C:g:n} of the renormalized area
Siegel--Veech constant and by
Formula~\eqref{eq:carea:parallelogram},
for the Siegel--Veech constant before the normalization,
we obtain
\begin{multline*}
C_g^{\textit{parallelogram},n,d_j}(n;d)
=
4\pi^2\cdot W_g(n;d)\cdot
c_{area}^{\textit{parallelogram},n,d_j}
\bigl(\cQ_g(n,d,-1^{n+D-4g+4})\bigr)
\\
=
4\pi^2
\frac{2d_j n}
{(n+D+k-2g+1)(n+D+k-2g+2)}
\\
\times
\frac{
\Vol\cQ_{g-1}(n-2,\hat d^{(j)},-1^{n+D-4g+4})
}{
n^g\cdot v(n)\cdot \prod_{i=1}^k v(d_i)
}.
\end{multline*}
Substituting~\eqref{eq:parallelogram:boundary:volume:ratio},
cancelling $d_j$, and cancelling one factor of $n$, we get~\eqref{eq:parallelogram:genus:g:renormalized}.
\end{proof}

% ------------------------------------------------------------------
\subsection{\texorpdfstring{Contributions of graphs of type ``b'' in genus $1$.}{Contributions of graphs of type b in genus 1.}}Throughout this subsection, we assume $n\ge 5$.
\begin{corollary}
The contributions of the figure-eight configurations to the renormalized
Siegel--Veech constants of the strata $\Vol\cQ_1(n,-1^n)$
and $\Vol\cQ_1(n,1,-1^{n+1})$ have the following form:
\begin{equation}
\label{eq:figure:eight:g:1}
\begin{aligned}
&C_1^{\textit{figure-eight},n}(n)=\frac{8(n+1)(n-1)}{n^3}\,,
\\
&C_1^{\textit{figure-eight},n}(n;1)=
\frac{8(n-1)^2}{n^2(n+2)}\,.
\end{aligned}
\end{equation}
\end{corollary}
\begin{proof}
Apply~\eqref{eq:figure:eight:genus:g:renormalized} with $g=1$ for these two specific values of $d$ and use~\eqref{eq:W:in:genus:0} for the explicit constant value $2\pi^2$ of $W_0$.
\end{proof}

\begin{corollary}
The contribution of the parallelogram configurations to the renormalized
Siegel--Veech constant of the stratum $\Vol\cQ_1(n,1,-1^{n+1})$ has the following form:
\begin{equation}
\label{eq:parallelogram:g:1}
C_1^{\textit{parallelogram},n}(n;1)=
\frac{32}{n(n+2)}\,.
\end{equation}
\end{corollary}
\begin{proof}
Apply~\eqref{eq:parallelogram:genus:g:renormalized} with $g=1$ and $d=\{1\}$ and use~\eqref{eq:W:in:genus:0} for the explicit constant value $2\pi^2$ of $W_0$.
\end{proof}

% ###############################################################
% ###############################################################
% ###############################################################
\section{Contributions of graphs of type ``c''}
\label{s:Contributions:of:graphs:c}

% ---------------------------------------------------------------------
\subsection{Pocket (ear) configuration}
\label{ss:pocket:configuration}

We start with the configuration encoded by the decorated ribbon graph as in the left picture in Figure~\ref{fig:pocket}. It consists of a flat cylinder having the distinguished zero of degree $n$ at the base and a saddle connection joining a pair of simple poles on the other side of the cylinder. This configuration was called a ``pocket'' in~\cite{Athreya:Eskin:Zorich}; it is sometimes also called an ``ear''.
In the case of genus $0$ the corresponding Siegel--Veech constant was computed in~\cite{Athreya:Eskin:Zorich}. We apply Goujard's Formula~\eqref{eq:carea} in the case of arbitrary genus.

\begin{figure}[htb]
\centering
\includegraphics[scale=0.9]{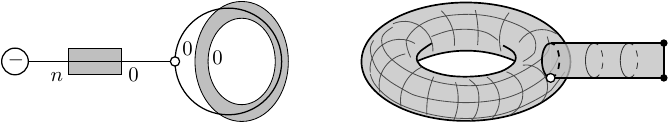}
\begin{picture}(0,0)(350,-76)
\put(30,-72){$\{d,-1^{n+D-4g+2}\}$}
\end{picture}
\caption{Pocket (ear) configuration.
The picture on the right illustrates a particular case
when $g=1$}
\label{fig:pocket}
\end{figure}

\begin{proposition}
\label{prop:pocket}
Assume that $n\ge 1$ and that $n+D-4g\ge 1$. The contribution of the ``pocket'' (``ear'') configuration represented by the graph in Figure~\ref{fig:pocket} with admissible decoration has the following form:
\begin{multline}
\label{eq:pocket}
c_{area}^{pocket,n}\big(\cQ_g(n,d_1,\dots,d_k,-1^{n+D-4g+4})\big)
\\
=\frac{n}{4}
\cdot\frac{(n+D-4g+4)(n+D-4g+3)}{(n+D+k-2g+2)(n+D+k-2g+1)}
\\
\times
\frac{\Vol\cQ_g(n-2,d_1,\dots,d_k,-1^{n+D-4g+2})}{\Vol\cQ_g(n,d_1,\dots,d_k,-1^{n+D-4g+4})}\,.
\end{multline}
\end{proposition}
\begin{proof}
The graph does not have any $\oplus$-vertices and has a single $\ominus$-vertex; the associated local ribbon graph is of type $-1.1$ and has a single connected component with a single conical singularity of degree $n$, so Formula~\eqref{eq:Dji} gives $\delta=n-2$. Merging it with the interior singularities $\{d,-1^{n+D-4g+2}\}$ we get the principal boundary stratum
\[
\cQ(\alpha')=\cQ_g(n-2,d,-1^{n+D-4g+2})
\]
We have
\begin{align*}
a&= \dim_{\C}\cQ_g(n,d_1,\dots,d_k,-1^{n+D-4g+4})= n+D+k-2g+3\,,
\\
a_1&= \dim_{\C}\cQ_g(n-2,d_1,\dots,d_k,-1^{n+D-4g+2})= n+D+k-2g+1\,,
\end{align*}
For this configuration we have a single cylinder, and it is thick, see Lemma~\ref{lm:thin:and:thick:cylinders}. We get the following values for the number $q_1$ of thin cylinders and the number $q_2$ of thick cylinders and for the number $m$ of $\oplus$ and $\ominus$-vertices counted together:
\[
q_1=0\,,\qquad
q_2=1\,,\qquad
m=1\,,
\]
so for a fixed labeled pair of poles~\eqref{eq:carea} gives:
\begin{multline}
\label{eq:Ca:before:weight}
c_{area}^{pocket,n,\textit{fixed poles}}\big(\cQ_g(n,d_1,\dots,d_k,-1^{n+D-4g+4})\big)
\\=
M\cdot\frac{4\cdot 0+1}{2^{1+1+3}}\cdot
\frac{(n+D+k-2g)!}{(n+D+k-2g+2)!}
\cdot\frac{\Vol\cQ(n-2,d,-1^{n+D-4g+2})}{\Vol\cQ(n,d,-1^{n+D-4g+4})}
\\=
M\cdot \frac{2^{-5}}{(n+D+k-2g+1)(n+D+k-2g+2)}
\cdot\frac{\Vol\cQ(n-2,d,-1^{n+D-4g+2})}{\Vol\cQ(n,d,-1^{n+D-4g+4})}\,.
\end{multline}

Note that in Goujard's formula~\eqref{eq:carea} all the zeroes and poles involved in the configuration are labeled. Thus, there are $\binom{n+D-4g+4}{2}$ ways to choose two distinguished simple poles located on top of the ear out of $n+D-4g+4$ poles. Hence, the number of all pocket configurations with a zero of degree $n$ at the base is
\begin{equation}
\label{eq:combinatorics:Ca}
\frac{(n+D-4g+4)(n+D-4g+3)}{2}\,.
\end{equation}

Let us compute now the constant $M$ for this configuration (see Section~\ref{ss:Formula:for:the:combinatorial:constant} for definitions). We have $m=1$ and
\[
M_c =4^{1+1}\,,\qquad M_t =1\,.
\]
The condition $n+D-4g\ge 1$ implies that the principal boundary stratum
has at least three simple poles, so
\[
K(v_1)=(0+n)=n\,,\qquad |\Gamma(v_1)|=1\,,
\]
so $M_s=\frac{K(v_1)}{|\Gamma(v_1)|}=n$
and thus
\begin{equation}
\label{eq:M:Ca}
M=\frac{M_s M_c}{M_t}
=\frac{n\cdot 2^4}{1}\,.
\end{equation}

Combining~\eqref{eq:Ca:before:weight}, \eqref{eq:combinatorics:Ca} and~\eqref{eq:M:Ca} we get~\eqref{eq:pocket}. Proposition~\ref{prop:pocket} is proved.
\end{proof}

\begin{corollary}
\label{cor:pocket:renormalized}
Assume the hypotheses of Proposition~\ref{prop:pocket} and assume now that $n\ge 3$. The contribution of the ``pocket''
(``ear'') configuration represented by the graph in Figure~\ref{fig:pocket} has the following form in renormalized variables:
\begin{multline}
\label{eq:Ca:total:renormalized:g}
C_g^{pocket,n}(n;d)
\\=
(n+1)\cdot
\left(\frac{n-2}{n}\right)^g\cdot
\frac{(n+D-4g+4)(n+D-4g+3)}{(n+D+k-2g+2)(n+D+k-2g+1)}
W_g(n-2;d)\,.
\end{multline}
\end{corollary}
\begin{proof}
Applying~\eqref{eq:v:n:minus:2} for the ratio $\frac{v(n-2)}{v(n)}$
and applying the definition~\eqref{eq:W:g:n} of the renormalized volumes $W_g(n;d)$ we get
\begin{multline*}
\frac{\Vol\cQ_g(n-2,d,-1^{n+D-4g+2})}
{n^g\cdot v(n)\cdot v(d_1)\cdots v(d_k)}
\\
=\frac{\Vol\cQ_g(n-2,d,-1^{n+D-4g+2})}
{(n-2)^g\cdot v(n-2)\cdot v(d_1)\cdots v(d_k)}
\cdot\left(\frac{n-2}{n}\right)^g\cdot\frac{v(n-2)}{v(n)}
\\
=\frac{1}{\pi^2}\cdot
\left(\frac{n-2}{n}\right)^g\cdot\frac{n+1}{n}\cdot W_g(n-2;d)
\,,
\end{multline*}
which allows us to rewrite the contribution of the pocket (ear) configuration to the sum in~\eqref{eq:EKZ:in:W:g} as~\eqref{eq:Ca:total:renormalized:g}.
\end{proof}

\begin{corollary}
\label{cor:pocket:d}
Assume the hypotheses of Proposition~\ref{prop:pocket} and assume that $d_j\geq1$, $n+D-4g+2\ge 0$. The contribution of the ``pocket'' (``ear'') configuration represented by the graph as in Figure~\ref{fig:pocket} but having $n$ as an interior singularity and $d_j$ as a boundary singularity has the following form:
\begin{multline}
\label{eq:pocket:d}
c_{area}^{pocket,d_j}\big(\cQ_g(n,d_1,\dots,d_k,-1^{n+D-4g+4})\big)
\\
=\frac{d_j}{4}
\cdot\frac{(n+D-4g+4)(n+D-4g+3)}{(n+D+k-2g+2)(n+D+k-2g+1)}
\\
\times
\frac{\Vol\cQ_g(n,d_1,\dots,d_{j-1},d_j-2,d_{j+1},\dots,d_k,-1^{n+D-4g+2})}
{\Vol\cQ_g(n,d_1,\dots,d_k,-1^{n+D-4g+4})}\,.
\end{multline}
\end{corollary}

\begin{remark}
\label{rm:entry:minus:1}
In the particular case when $d_j=1$, we get $d_j-2=-1$. In this case we use the following notational convention
\begin{multline*}
\cQ_g(n,d_1,\dots,d_{j-1},-1,d_{j+1},\dots,d_k,-1^{n+D-4g+2})
\\ :=
\cQ_g(n,d_1,\dots,d_{j-1},d_{j+1},\dots,d_k,-1^{n+D-4g+3})
\end{multline*}
In the particular case when $d_j=2$, we get $d_j-2=0$. We keep the resulting marked point in
$\cQ_g(n,d_1,\dots,d_{j-1},0,d_{j+1},\dots,d_k,-1^{n+D-4g+2})$
as part of the singularity data. Recall that it affects the value of the
volume before normalization and also affects the dimension of the boundary stratum.
\end{remark}

\begin{proof}[Proof of Corollary~\ref{cor:pocket:d}]
The proof is parallel to the proof of Proposition~\ref{prop:pocket}.
The only difference with
the previous configuration is notational: the boundary singularity
is now denoted by $d_j$ and the interior singularities by
\[
\{n,d_1,\dots,d_{j-1},\widehat{d_j},d_{j+1},\dots,d_k,-1^{n+D-4g+2}\},
\]
so $n$ and $d_j$ are interchanged. This interchange does not affect
the sum $n+D$, so we get~\eqref{eq:pocket:d}.
\end{proof}

\begin{corollary}
\label{cor:pocket:d:renormalized}
Assume the hypotheses of Corollary~\ref{cor:pocket:d}. The contribution of the ``pocket'' (``ear'') configuration represented by the graph as in Figure~\ref{fig:pocket} but having $n$ as an interior singularity and $d_j$ as a boundary singularity has the following form in renormalized variables as $n\to+\infty$:
\begin{multline}
\label{eq:Ca:total:renormalized:dj:g}
C_g^{pocket,d_j}(n;d)
=
(d_j+1)
\cdot\frac{(n+D-4g+4)(n+D-4g+3)}{(n+D+k-2g+2)(n+D+k-2g+1)}
\\
\times
W_g(n;d_1,\dots,d_{j-1},d_j-2,d_{j+1},\dots,d_k)
\\
=(d_j+1)
\cdot
W_g(n;d_1,\dots,d_{j-1},d_j-2,d_{j+1},\dots,d_k)
\cdot\left(1+O_{g,d}\left(\frac1n\right)\right)\,.
\end{multline}
\end{corollary}

\begin{remark}
\label{rm:entry:minus:1:in:W}
In the particular cases when $d_j=1$, or when $d_j=2$, we get $d_j-2=-1$
or $d_j-2=0$ respectively. Recall that by~\eqref{eq:W:of:0}
and by~\eqref{eq:W:of:minus:1} we have
\[
\begin{aligned}
W_g(n;d_1,\dots,d_{j-1},-1,d_{j+1},\dots,d_k)=
W_g(n;d_1,\dots,d_{j-1},d_{j+1},\dots,d_k)\,,
\\
W_g(n;d_1,\dots,d_{j-1},0,d_{j+1},\dots,d_k)=
W_g(n;d_1,\dots,d_{j-1},d_{j+1},\dots,d_k)\,.
\end{aligned}
\]
The conventions for the resulting principal boundary strata in these two cases are described in Remark~\ref{rm:entry:minus:1}.
\end{remark}

\begin{proof}[Proof of Corollary~\ref{cor:pocket:d:renormalized}]
The volume normalization is not symmetric with respect to the interchange
of $n$ and $d_j$. Now we have
\begin{multline*}
\frac{\Vol\cQ_g(n,d_1,\dots,d_{j-1},d_j-2,d_{j+1},\dots,d_k,-1^{n+D-4g+2})}
{n^g\cdot v(n)\cdot v(d_1)\cdots v(d_k)}
\\=
\frac{\Vol\cQ_g(n,d_1,\dots,d_{j-1},d_j-2,d_{j+1},\dots,d_k,-1^{n+D-4g+2})}
{n^g\cdot v(n)\cdot v(d_1)\cdots v(d_{j-1})\cdot v(d_{j}-2)\cdot v(d_{j+1})\cdots v(d_k)}
\cdot\frac{v(d_j-2)}{v(d_j)}
\\
=\frac{1}{\pi^2}\cdot\frac{d_j+1}{d_j}\cdot W_g(n;d_1,\dots,d_{j-1},d_j-2,d_{j+1},\dots,d_k)
\,,
\end{multline*}
which allows us to rewrite the contribution of the pocket (ear) configuration to $\carea$ in renormalized variables as~\eqref{eq:Ca:total:renormalized:dj:g}.
\end{proof}

Note that the factor in front of the variable $W$ in~\eqref{eq:Ca:total:renormalized:g} is of order $O(n)$ while the analogous factor in~\eqref{eq:Ca:total:renormalized:dj:g} is of order $O(1)$.

% -------------------------------------------------------------------
\subsection{Contribution of graphs of type ``c'' in genus one.}Throughout this subsection, we assume $n\ge 5$.

We collect now contributions of configurations corresponding to graphs of type ``c'' to the area Siegel--Veech constants of strata $\cQ_1(n,-1^n)$ and $\cQ_1(n,1,-1^{n+1})$. We start with ``pocket'' configurations.

\begin{lemma}
%\label{lm:pocket:renormalized:g:1}
The contributions of the ``pocket'' configurations represented by the graphs
\textit{II} in Figure~\ref{fig:minimal:genus:1} and graphs \textit{v} and \textit{vi} in Figure~\ref{fig:graphs:n:1:genus:1}
to the area Siegel--Veech constants of the strata $\cQ_1(n,-1^n)$ and $\cQ_1(n,1,-1^{n+1})$
have the following forms in renormalized variables:
\begin{align}
\label{eq:pocket:n:g:1:renormalized}
&C_1^{\textit{pocket}}(n)
=\frac{(n+1)(n-2)}{n}\cdot W_1(n-2)\,,
\\
\label{eq:pocket:at:n:n:1:g:1:renormalized}
&C_1^{\textit{pocket,n}}(n;1)
=\frac{(n+1)(n-2)}{n+2}\cdot W_1(n-2;1)\,,
\\
\label{eq:pocket:at:1:n:1:g:1:renormalized}
&C_1^{\textit{pocket,1}}(n;1)
=\frac{2n}{n+2}\cdot W_1(n)\,.
\end{align}
\end{lemma}
\begin{proof}
Equation~\eqref{eq:pocket:n:g:1:renormalized} is obtained as a particular case of~\eqref{eq:Ca:total:renormalized:g}, where we let $g=1$ and $D=k=0$ and $d=\emptyset$.
Equation~\eqref{eq:pocket:at:n:n:1:g:1:renormalized} is obtained as a particular case of~\eqref{eq:Ca:total:renormalized:g}, where we let
\[
g=k=D=d_1=1,\qquad d=(1)\,.
\]
Finally, Equation~\eqref{eq:pocket:at:1:n:1:g:1:renormalized}
is obtained as a particular case of the first equality in~\eqref{eq:Ca:total:renormalized:dj:g}, where we let $g=D=k=j=d_j=1$,
and where we use convention~\eqref{eq:W:of:minus:1}
according to which $W_1(n;-1)=W_1(n)$.
\end{proof}

We compute now contributions of configurations represented by  the graph \textit{III} in Figure~\ref{fig:minimal:genus:1} and by the graph \textit{vii} in Figure~\ref{fig:graphs:n:1:genus:1}. These contributions will be needed in computation of volumes of strata $\cQ_1(n,-1^n)$ and $\cQ_1(n,1,-1^{n+1})$ respectively. We perform a detailed computation for the stratum $\cQ_1(n,-1^n)$. The computation for the stratum $\cQ_1(n,1,-1^{n+1})$ is completely analogous.

\begin{proposition}
\label{prop:pocket:through:torus}
For $n\ge 5$
the contributions of the ``pocket-through-torus'' configurations represented by the graphs
\textit{III} in Figure~\ref{fig:minimal:genus:1} and \textit{vii} in Figure~\ref{fig:graphs:n:1:genus:1}
to the area Siegel--Veech constants of the strata $\cQ_1(n,-1^n)$ and $\cQ_1(n,1,-1^{n+1})$ respectively
have the following forms:
\begin{align}
\label{eq:pocket:through:torus:n:g:1}
\carea^{\textit{pocket-through-torus}}&\big(\cQ_1(n,-1^n)\big)
\\
\notag
&=\frac{n-4}{2(n-2)(n-3)}
\cdot\frac{\Vol\cH(0)\cdot\Vol\cQ_0(n-6,-1^{n-2})}{\Vol\cQ_1(n,-1^n)}\,.
\\
\label{eq:pocket:through:torus:n:1:g:1}
\carea^{\textit{pocket-through-torus}}&
\big(\cQ_1(n,1,-1^{n+1})\big)
\\
\notag
&=\frac{n-4}{2(n+2)(n-1)}
\cdot\frac{\Vol\cH(0)\cdot\Vol\cQ_0(n-6,1,-1^{n-1})}{\Vol\cQ_1(n,1,-1^{n+1})}\,.
\end{align}
\end{proposition}
\begin{proof}
This time we have
\begin{align*}
a&=\dim_{\C}\cQ_1(n,-1^n)=n+1\,,\\
a_1&=\dim_{\C}\cQ_0(n-6,-1^{n-2})=n-3\,,\\
b_1&=\dim_{\C}\cH(0)=2\,,\\
\end{align*}
and
\[
q_1=0\,,\qquad
q_2=1\,,\qquad
m=2\,,
\]
so~\eqref{eq:carea} gives:
\begin{multline}
c_{area}^{\textit{pocket-through-torus; fixed poles}}\big(\cQ_1(n,-1^n)\big)
\\=
M\cdot\frac{4\cdot 0+1}{2^{2+1+3}}\cdot
\frac{\bigg((2-1)!\cdot 2^2\cdot\Vol\cH(0)\bigg)
\bigg(((n-3)-1)! \Vol\cQ_0(n-6,-1^{n-2})\bigg)}
{((n+1)-1)!\Vol\cQ_1(n,-1^n)}
\\=
\label{eq:Cb:before:weight:minimal:g:1}
M\cdot \frac{2^{-4}}{n(n-1)(n-2)(n-3)}
\cdot\frac{\Vol\cH(0)\cdot\Vol\cQ_0(n-6,-1^{n-2})}{\Vol\cQ_1(n,-1^n)}\,.
\end{multline}
There are $\binom{n}{2}$ ways to choose two simple poles
out of $n$, so the combinatorial factor here is
\begin{equation}
\label{eq:combinatorics:Cb:minimal:g:1}
\frac{n(n-1)}{2}\,.
\end{equation}
For this configuration we have $q=q_1+q_2=1$ and the graph is not of type ``a'', so we have
\[
M_c=4^{1+1}\,,\qquad
M_t=1
\]
(see Section~\ref{ss:Formula:for:the:combinatorial:constant}).
We also have
\[
  K(v_1) =n-4\,,\qquad
  |\Gamma(v_1)|=1
\]
and
\[
  K(v_2) =2\cdot\left(\frac{1}{2}+\frac{1}{2}\right) = 2\,,\qquad
  |\Gamma(v_2)|=2\,,
\]
where $v_1$ represents the stratum $\cQ_0(n-6,-1^{n-2})$ and $v_2$ represents the stratum $\cH(0)$, so
\[
M_s=\frac{n-4}{1}\cdot\frac{2}{2}=n-4\,,
\]
and
\begin{equation}
\label{eq:M:Cb:minimal:g:1}
M=\frac{M_s M_c}{M_t}
=\frac{(n-4)\cdot 2^4}{1}=2^4\cdot(n-4)\,.
\end{equation}
Combining~\eqref{eq:Cb:before:weight:minimal:g:1}, \eqref{eq:combinatorics:Cb:minimal:g:1}
and~\eqref{eq:M:Cb:minimal:g:1} we get:
\begin{multline*}
c_{area}^{\textit{pocket-through-torus}}\big(\cQ_1(n,-1^n)\big)
= \big(2^4\cdot (n-4)\big)
\\ \times
\left(\frac{n(n-1)}{2}\right)\cdot
\frac{2^{-4}}{n(n-1)(n-2)(n-3)}
\cdot\frac{\Vol\cH(0)\cdot\Vol\cQ_0(n-6,-1^{n-2})}{\Vol\cQ_1(n,-1^n)}\,,
\end{multline*}
which proves~\eqref{eq:pocket:through:torus:n:g:1}.

Evaluation of contribution of the configuration \textit{vii} in Figure~\ref{fig:graphs:n:1:genus:1} is completely analogous. There is a change in dimensions: now we have $a=n+3$ and $a_1=n-1$, which changes the ratio $\frac{(a_1-1)!}{(a-1)!}$. There is one more pole now, so the number of choices of a pair of poles becomes $\frac{(n+1)n}{2}$.
The combinatorial constant $M$ remains unchanged. We get
\begin{multline*}
c_{area}^{\textit{pocket-through-torus}}\big(\cQ_1(n,1,-1^{n+1})\big)
= \big(2^4\cdot (n-4)\big)
\\ \times
\left(\frac{(n+1)n}{2}\right)\cdot
\frac{2^{-4}}{(n+2)(n+1)n(n-1)}
\cdot\frac{\Vol\cH(0)\cdot\Vol\cQ_0(n-6,1,-1^{n-1})}{\Vol\cQ_1(n,1,-1^{n+1})}\,,
\end{multline*}
which proves~\eqref{eq:pocket:through:torus:n:1:g:1}
and thus completes the proof of Proposition~\ref{prop:pocket:through:torus}.
\end{proof}

\begin{corollary}
%\label{cor:pocket:through:torus:renormalized}
  %
For $n\ge 5$,
the contributions of the ``pocket-through-torus'' configurations represented by the graphs
\textit{III} in Figure~\ref{fig:minimal:genus:1} and \textit{vii} in Figure~\ref{fig:graphs:n:1:genus:1}
to the area Siegel--Veech constants of the strata $\cQ_1(n,-1^n)$ and $\cQ_1(n,1,-1^{n+1})$ respectively
have the following forms in renormalized variables:
\begin{align}
\label{eq:pocket:through:torus:n:g:1:renormalized}
&C_1^{\textit{pocket-through-torus}}(n)
=\frac{4(n+1)(n-1)}{3n^2(n-2)^2}\,,
\\
\label{eq:pocket:through:torus:n:1:g:1:renormalized}
&C_1^{\textit{pocket-through-torus}}(n;1)
=\frac{4(n+1)(n-3)}{3n^2(n+2)(n-2)}\,.
\end{align}
\end{corollary}
\begin{proof}
Using Formula~\eqref{eq:Vol:minimal:g0} for $\Vol\cQ_0(n-6,-1^{n-2})$ and the value $\Vol\cH(0)=\frac{\pi^2}{3}$, which can be found, for example, in~\cite{Eskin:Okounkov:Abelian} and in~\cite{Zorich:square:tiled} and applying Formula~\eqref{eq:v:n:minus:6} for the ratio $\frac{v(n-6)}{v(n)}$ we get
\[\frac{\Vol\cQ_0(n-6,-1^{n-2})\cdot\Vol\cH(0)}{n\cdot v(n)}
=\frac{2\pi^4}{3n}\frac{v(n-6)}{v(n)}
=\frac{2}{3\pi^2}
\frac{(n+1)(n-1)(n-3)}{n^2(n-2)(n-4)}\,.
\]
Therefore, passing to the renormalized variable
\begin{multline*}
C_1^{\textit{pocket-through-torus}}(n)
\\=
\frac{4\pi^2}{n\cdot v(n)}\cdot\Vol(\cQ_1(n,-1^n))\cdot
\carea^{\textit{pocket-through-torus}}\big(\cQ_1(n,-1^n)\big)
\end{multline*}
and using~\eqref{eq:pocket:through:torus:n:g:1} for
$\carea^{\textit{pocket-through-torus}}\big(\cQ_1(n,-1^n)\big)$
we get
\begin{multline*}
C_1^{\textit{pocket-through-torus}}(n)
=
\frac{4\pi^2}{n\cdot v(n)}
\\
\times\frac{n-4}{2(n-2)(n-3)}
\Vol\cQ_0(n-6,-1^{n-2})\cdot\Vol\cH(0)
=
\frac{4(n+1)(n-1)}{3n^2(n-2)^2}\,,
\end{multline*}
which corresponds to~\eqref{eq:pocket:through:torus:n:g:1:renormalized}.

Similarly to the case of the minimal stratum we have
\[
\frac{\Vol\cH(0)\cdot\Vol\cQ_0(n-6,1,-1^{n-1})}
{n\cdot v(n)\cdot v(1)}
=\frac{2\pi^4}{3n}
\frac{v(n-6)}{v(n)}
=\frac{2}{3\pi^2}
\frac{(n+1)(n-1)(n-3)}{n^2(n-2)(n-4)}\,.
\]
Thus we get the following normalized contribution
\begin{multline*}
C_1^{\textit{pocket-through-torus}}(n;1)
=\frac{4\pi^2}{n\cdot v(n)\cdot v(1)}
\\
\times
\frac{n-4}{2(n+2)(n-1)}
\cdot\Vol\cH(0)\cdot\Vol\cQ_0(n-6,1,-1^{n-1})
=\frac{4(n+1)(n-3)}{3n^2(n+2)(n-2)}\,,
\end{multline*}
which corresponds to~\eqref{eq:pocket:through:torus:n:1:g:1:renormalized}.
\end{proof}

% \begin{figure}[htb]
% \centering
% \includegraphics[scale=0.8, angle=90]{pocket_through_torus.pdf}
%    %
% \caption{
% \label{fig:pocket:through:torus}
% ``Pocket-through-torus'' configuration}
% \end{figure}

\begin{remark}
Geometry of the ``pocket-through-torus'' configuration can be seen as follows. In the pocket (ear) configuration considered above we considered a surface $S'$ in the principal boundary stratum, broke the zero of degree $n-2$ (respectively of degree $d_j-2$) into a zero of degree $n-1$ (respectively of degree $d_j-1$) and a simple pole, slit the saddle connection between the newborn zero and the simple pole and pasted a pocket to the resulting hole. Now we start with a sphere $S'$ in the stratum $\cQ_0(n-6,-1^{n-2})$ (respectively in the stratum $\cQ_0(n-6,1,-1^{n-1})$), and create a hole at the zero of degree $n-6$. However, instead of pasting a cylinder directly to the hole, we do it through an auxiliary flat torus. We first slit the torus along a geodesic segment and bring the two endpoints of the slit together thus creating a boundary in the shape of a figure eight. We glue one of the loops of this figure eight to the sphere and we paste a pocket (ear) to the other loop.
\end{remark}

% ###############################################################
% ###############################################################
% ###############################################################
\section{Dominant and negligible contributions}
\label{s:Dominant:and:negligible:contributions}

%--------------------------------------------------------------
\subsection{Order of magnitude of contribution of an individual configuration}

Fix $(g,d)$ and take $n$ sufficiently large that $n>4g$, $n>d_j$ for every $j$, so that Lemma~\ref{lm:ominus:vertex:associated:to:zero:of:degree:n} applies. Consider a configuration $\cC_{\mathit{fixed}}$. We assume that all singularities, in particular, all endpoint singularities of the ambient stratum are labeled. Denote by $v_1$ the $\ominus$-vertex
carrying the unique zero of degree $n$, see Definition~\ref{def:zero:carried:by:a:minus:vertex}. Denote by $g_1$ the genus carried by this vertex. Let $n-\ell$ be the largest order among all zeroes of the corresponding boundary stratum $\cQ_{g_1}(\alpha_1)$; denote the collection of the remaining zeroes and marked points by $d'$.
We fix the configuration allowing only the parameter $n$ to grow.
In particular, the distinguished vertex $v_1$, the genus $g_1$,
the value of $\ell$ and the collection $d'$ are fixed. The simple poles and their distribution among all vertices except $v_1$ are also fixed.
The number of poles in $\alpha_1$ grows to compensate for the growing value of $n$
while the genus $g_1$ remains fixed.

Let
\[
a=\dim_{\C}\cQ_g(n,d,-1^{n+D-4g+4})
\]
be the dimension of the ambient stratum, and let $a_1=\dim_{\C}\cQ_{g_1}(\alpha_1)$ be the dimension of the principal boundary stratum corresponding to the distinguished $\ominus$-vertex $v_1$. Let $a_\mathit{rest}$ be the sum of dimensions carried by all other vertices, no matter $\oplus$ or $\ominus$. Let $q\geq 1$ be the total number of cylinders.

\begin{lemma}
%\label{lm:order:of:magnitude:of:cC}
Fix $g\geq1$ and a collection $d=(d_1,\dots,d_k)$ of positive integers, allowing $k=0$ and $d=\emptyset$. One has
\begin{equation}
\label{eq:order:of:C}
C_{g}^{\cC_{\mathit{fixed}}}(n;d)
=O_{g,d}\bigl(n^{g_1-g-q-a_\textit{rest}-\iota}\bigr)
\cdot W_{g_1}(n-\ell;d'),
\end{equation}
where $\iota=1$ if $n$ is an interior singularity and $\iota=0$ otherwise.
\end{lemma}
\begin{proof}
Combining~\eqref{eq:carea} and~\eqref{eq:C:g:n}, cancelling out $\Vol\cQ_g(n,d,-1^{n+D-4g+4})$ present in the numerator and in the denominator, and rearranging the order of factors we get
\begin{multline*}
C_{g}^{\cC_{\mathit{fixed}}}(n;d)=
\frac{\Vol\cQ_{g_1}(\alpha_1)}
{n^g\cdot v(n)\cdot v(d_1)\cdots v(d_k)}
\cdot \frac{(a_1-1)!}{(a-1)!}
\\
\times 4\pi^2
\cdot M\cdot\frac{(4q_1+q_2)}{2^{m+q+3}}
\cdot
\prod_{\substack{i\\i\neq 1}}\big((a_i-1)!\Vol\cQ(\alpha'_i)\big)\cdot
\prod_j\big(2^{b_j}\cdot (b_j-1)! \cdot \Vol\cH(\beta'_j)\big)
\,.
\end{multline*}
The only term in the second line which might depend on $n$
is the combinatorial constant $M$. We have $M_c=O_q(1)$ and $M_t=O(1)$.
We also have $K(v_i)=O_{g,d}(1)$ for $i\neq 1$ and $\Gamma(v_i)=O(1)$
for all $i$. Finally,
\[
K(v_1)=
\begin{cases}
O_{g,d}(1)
&\text{when $n$ is an interior singularity}\,,
\\
O_{g,d}(n)
&\text{when $n$ is an endpoint singularity}\,.
\end{cases}
\]
Thus
\[
M=
\begin{cases}
O_{g,d}(1)
&\text{when $n$ is an interior singularity}\,,
\\
O_{g,d}(n)
&\text{when $n$ is an endpoint singularity}\,.
\end{cases}
\]
The product of all remaining terms in the second line is of order $O(1)$.

A direct count of independent period coordinates in the ambient stratum gives
\[
 a=a_1+a_\textit{rest}+q+1\,.
\]
Thus
\[
\frac{(a_1-1)!}{(a-1)!}=O(n^{a_1-a})=O(n^{-(a_\textit{rest}+q+1)})\,.
\]

Lemma~\ref{lm:ominus:vertex:associated:to:zero:of:degree:n} implies that the stratum $\cQ_{g_1}(\alpha_1)$ has a single zero of order $n-\ell$,
where $\ell$ does not depend on $n$. The number and the orders of other
zeroes (including marked points but excluding simple poles) do not depend on $n$. We have
\[
\frac{v(n-\ell)}{v(n)}=O_{g,d}(1)\,.
\]
Recall that we denote by $d'=(d'_1,\dots,d'_{k'})$ orders of other zeroes of
the stratum $\cQ_{g_1}(\alpha_1)$. The definition of the normalized volume gives
\begin{multline*}
\frac{\Vol\cQ_{g_1}(\alpha_1)}
{n^g\cdot v(n)\cdot v(d_1)\cdots v(d_k)}
\\
=\frac{(n-\ell)^{g_1}}{n^g}
\cdot\frac{v(n-\ell)}{v(n)}
\cdot\frac{v(d'_1)\cdots v(d'_{k'})}{v(d_1)\cdots v(d_k)}
\cdot W_{g_1}(n-\ell;d')
\\=O_{g,d}(n^{g_1-g})\cdot W_{g_1}(n-\ell;d')\,,
\end{multline*}
since the ratio of the remaining $v$-factors is bounded.
Collecting the above estimates we get~\eqref{eq:order:of:C}.
\end{proof}

%--------------------------------------------------------------
\subsection{Grouping configurations with identical contributions}

Recall that given a configuration $\cC$ of saddle connections on a flat surface $S=(X,q)$, any zero or simple pole of $q$ is either an \textit{endpoint} of $\cC$, or is invisible to $\cC$ and is located somewhere inside $S$, in which case it is called an \textit{interior singularity}, see Definition~\ref{def:endpoint}. By definition, a saddle connection cannot have conical singularities inside it.

For fixed $(g,d)$ and all sufficiently large $n$, the decorated ribbon graphs can be organized into finitely many parameterized families with uniformly bounded fixed data. Namely, we absorb the growth in the number of decorated terms into the weights described below.
\smallskip

\noindent\textbf{Labeling endpoint poles.}
The ambient stratum $\cQ(n,d,-1^{n+D-4g+4})$ has a large number of simple poles, which are all labeled. When a configuration of \^homologous saddle connections
contains poles as endpoint singularities, the labels of these distinguished poles form part of the decoration of the associated ribbon graph.
Clearly, configurations $\cC$ which differ just by the choice of
these labels carry the same contributions to the sum in~\eqref{eq:EKZ:in:W:g}.
We will consider all these configurations as one, counted with weight
\begin{equation}
\label{eq:weight:for:numbering:endpoint:poles}
E(\GG)=\binom{n+D-4g+4}{e(\GG)}\,,
\end{equation}
where $e(\GG(\cC))$ is the number of poles involved in $\cC$ as endpoint singularities. We have $e(\GG)=0$ and $E(\GG)=\binom{n+D-4g+4}{0}=1$ if all poles are interior singularities. It follows from the classification of decorated ribbon graphs~\ref{fig:graphs} and \ref{fig:embedded:local:ribbon:graphs} that in the presence of at least one $\circ$-vertex, the number $e(\GG)$ of endpoint poles can take only the following values:
\begin{equation}
\label{eq:values:of:e}
e(\GG)\in
\begin{cases}
\{0\},&\text{for graphs of types ``a'' and ``b'';}
\\
\{0, 1, 2\}, &\text{for graphs of types ``c''}\,.
\end{cases}
\end{equation}
For example for a pocket (ear) configuration we have $e(\GG)=2$
and we recognize $E(\GG)$ in~\eqref{eq:combinatorics:Ca} and~\eqref{eq:combinatorics:Cb:minimal:g:1}.
\smallskip

\noindent\textbf{Labeling interior poles.}
Recall that graphs of type ``a'' have exactly two $\ominus$-vertices,
graphs of types ``b'' and ``c'' have exactly one $\ominus$-vertex, and we do not consider other graphs.
Note that $\oplus$-vertices cannot have poles as interior singularities.
Thus, when a graph $\GG(\cC)$ has two $\ominus$-vertices (i.e. when it has type ``a'' and not ``b'' or ``c''), for each choice of $e(\cC)$ labels for the endpoint poles, the remaining $n+D-4g+4-e(\cC)$ poles are interior singularities distributed between the two $\ominus$-vertices. By Lemma~\ref{lm:ominus:vertex:associated:to:zero:of:degree:n} only one of these two $\ominus$ vertices has $n$ as an interior singularity, or has a boundary singularity depending on $n$. Denote this vertex by $v_1$. Orders of all interior and boundary singularities of the remaining $\ominus$-vertex, which we will denote by $v_0$, do not depend on $n$. The number $p_0$ of simple poles serving as interior singularities of $v_0$ satisfies a relation imposed by the decoration of the connected component of $\GG$ containing the vertex $v_0$,
and thus, can be reconstructed from the corresponding ribbon graph structure,
degrees of boundary singularities and degrees of interior zeroes.
Hence, $p_0$ does not depend on $n$.
Clearly, contributions of configurations which differ only by the choice of labels of these $p_0(\cC)$ poles are the same, so we will consider all such configurations as one configuration counted with weight
\begin{equation}
\label{eq:weight:for:partitionning:interior:poles}
I(\GG)=\binom{n+D-4g+4-e(\cC)}{p_0}\,.
\end{equation}
We have
\begin{equation}
\label{eq:no:I:for:b:c}
I(\GG)=1\quad\text{for graphs of types ``b'' and ``c''.}
\end{equation}

We recognize $I(\GG)$ in~\eqref{eq:binomial} computed in the particular case of a dumbbell configuration.
\smallskip

\noindent\textbf{Splitting an endpoint singularity of large order.}
By Lemma~\ref{lm:ominus:vertex:associated:to:zero:of:degree:n} we can have only one vertex with orders of
boundary singularities depending on $n$, and, moreover, such a vertex,
if it is present, is of $\ominus$-type. In the case when this vertex
is of local type $-2.1$, we have two boundary singularities
of nonnegative orders $k',k''$ which satisfy $k'+k''=n-\ell+2$, so there are $n-\ell+3$ ordered pairs. Indeed, the definition of $\ell$ in the preceding section and~\eqref{eq:Dji} give $n-\ell=k'+k''-2$.
Contributions of configurations corresponding to different partitions
$k'+k''=n-\ell+2$, where $k'\neq k''$ are identical. We group together all the decorated ribbon graphs which differ only by this partition. We define the  resulting weight as
\begin{equation}
\label{eq:weight:for:n:prime:n:double:prime}
P(\GG)=
\begin{cases}
\frac{n-\ell+3}{|Aut((\partial\GG)_1)|}
      &\text{if the zero of degree $n$ is an endpoint singularity}\\
      &\text{carried by a $\ominus$-vertex of $-2.1$ type}\\
%       &\text{};
% \\
1&\text{otherwise},
\end{cases}
\end{equation}
Here $|Aut((\partial\GG)_1)|$ is the order of the symmetry group
of the associated boundary component of the global decorated ribbon graph.
It is equal to $1$ when ordered pairs $(n',n'')$
and $(n'',n')$ with $n'\neq n''$ give rise to distinct
decorated ribbon graphs and is equal to $2$ when
they give rise to isomorphic decorated ribbon graphs.
When the zero of degree $n$ is an endpoint singularity
carried by a $\ominus$-vertex $v_1$ of $-2.1$ type, and when
in addition, $v_1$ has a $\circ$-vertex as a direct neighbor, we have
$|Aut((\partial\GG)_1)|=2$.

We have
\begin{equation}
\label{eq:no:P:for:a}
P(\GG)=1\quad\text{for graphs of types ``a'' and ``c''.}
\end{equation}

We recognize $P(\GG)$ in~\eqref{eq:combinatorial:constant:figure:eight} evaluated for a figure-eight configuration. In this particular case we
have $\ell=4$ and $|Aut((\partial\GG)_1)|=2$.
\smallskip

\begin{lemma}
\label{lm:finite:sum}
For any configuration $\cC$ the weights have the following growth rates as $n\to+\infty$:
\begin{equation}
\label{eq:orders:of:magnitude:of:weight}
\begin{aligned}
E(\cC)&=O\bigl(n^{e(\cC)}\bigr),\\
I(\cC)&=O\bigl(n^{p_0(\cC)}\bigr),\\
P(\cC)&=\begin{cases}
O(n),&\text{if $n$ is an endpoint at a $-2.1$ vertex},\\
1,&\text{otherwise}.
\end{cases}
\end{aligned}
\end{equation}
\end{lemma}
\begin{proof}
Recall that $\oplus$-vertices carry strictly positive genus, and the local ribbon graph structures around these vertices, and orders of their boundary singularities and interior singularities are bounded in terms of $g,d$.

By Lemma~\ref{lm:ominus:vertex:associated:to:zero:of:degree:n} exactly one $\ominus$-vertex has $n$ as an interior singularity, or has a boundary singularity depending on $n$. Denote this $\ominus$-vertex by $v_1$. This is the only vertex for which the number of interior poles depends on $n$.

If the graph is of type ``a'', the remaining $\ominus$-vertex is of local type $-1.1$. The number of choices of the corresponding boundary singularity and of interior singularities is bounded in terms of $g,d$.

Thus, the number of ways to choose a ribbon graph structure and decorations of all vertices except the vertex $v_1$ and except labels for the endpoint poles (if any) is bounded. After the grouping described above,
we do not label the endpoint poles (if any) anymore and
the number of ways to choose $\GG(v_1)$ and its decorations
becomes bounded in terms of $g,d$.
Relation~\eqref{eq:orders:of:magnitude:of:weight} is a trivial
corollary of explicit formulae~\eqref{eq:weight:for:numbering:endpoint:poles},
\eqref{eq:weight:for:partitionning:interior:poles},
\eqref{eq:weight:for:n:prime:n:double:prime}.
This completes the proof of Lemma~\ref{lm:finite:sum}.
\end{proof}

%--------------------------------------------------------------
\subsection{Configurations carrying dominant and negligible contributions}
%\label{ss:Configurations:carrying:dominant:and:negligible:contributions}

We start with the following Proposition, which will be used as the main technical tool in deriving
a convenient recursive formula from~\eqref{eq:EKZ:in:W:g}.

Without loss of generality, we may assume that the positive integers $n,d_1,\dots,d_k$ in the statement of Theorem~\ref{th:n:D:large:zero:any:g} are ordered in the decreasing order:
\[
n\ge d_1\ge d_2\ge\dots\ge d_k\,.
\]

\begin{proposition}
\label{pr:boundary:stratum:is:simpler}
Consider a stratum $\cQ_g(n,d_1,\dots,d_k,-1^{n+D-4g+4})$,
with $n>4g$,
and a configuration $\cC$ of \^homologous saddle connections contributing to the area Siegel--Veech constant of this stratum. Denote by $v_1$ the distinguished $\ominus$-vertex $v_1$ carrying the zero of degree $n$ in the sense of Lemma~\ref{lm:ominus:vertex:associated:to:zero:of:degree:n}.
The  principal boundary stratum
\[
\cQ'(\alpha_1)
=\cQ_{g'}(n',d'_1,\dots,d'_{k'},-1^{n'+D'-4g'+4})
\]
associated to the vertex $v_1$ has the following property.
Assuming that
\[
n'\ge d'_1\ge \dots\ge d'_{k'}\ge 0\,,
\]
we have
\begin{equation}
\label{eq:boundary:strata:data}
n'\le n,\qquad
g'\le g,\qquad k'\le k,\qquad d'_i \le d_i\quad\text{for}\ i=1,\dots,k'\,,
\end{equation}
and, moreover, at least one of the inequalities is strict.

For fixed $(g,d)$ and all sufficiently large $n$, the only configuration for which all the following equalities
\begin{equation}
\label{eq:all:equalities}
g'= g,\qquad k'= k,\qquad d'_i = d_i\quad\text{for}\quad i=1,\dots,k'
\end{equation}
simultaneously take place is the ``pocket'' configuration
with a zero of degree $n$ at the base, presented in Figure~\ref{fig:pocket} and studied in Section~\ref{ss:pocket:configuration}. For this configuration
we have $n'=n-2$.
\end{proposition}
Note that the assumption that a configuration contains a cylinder is essential. For example, a configuration consisting of a single saddle connection joining two zeroes of degrees $m_1$ and $m_2$ gives rise to a principal boundary stratum with a single singularity of degree $m_1+m_2$ replacing the pair of original singularities.

Note that we assume that all $d_j$ are positive for the ambient stratum, while we allow $d'_{j'}$ to be equal to zero (and thus represent marked points). When the degree $\delta$ computed by~\eqref{eq:Dji} is equal to $-1$, the resulting simple pole is registered among $n'+D'-4g'+4$ simple
poles of the boundary stratum, see Remarks~\ref{rm:entry:minus:1}
and~\ref{rm:entry:minus:1:in:W}.

\begin{proof}
Surfaces from boundary strata
are obtained from subsurfaces of the original surface $S$ by contracting boundary components. Such a surgery can only reduce the genus but not increase it, which justifies $g'\le g$.

Let us control the degrees and the number of zeroes of the boundary strata. Denote by $m'_1,\dots,m'_s$ interior singularities (including poles) associated to a $\ominus$-vertex $v_1$. By definition they form a strict subset of the set of singularities of the ambient stratum.

Suppose that the vertex $v_1$ has local type $-1.1$, see Figure~\ref{fig:embedded:local:ribbon:graphs}. Let $l$ be the degree of the corresponding single boundary singularity. The local ribbon graph $\GG(v_1)$
of a vertex of type $-1.1$
intersects with a single connected component of the boundary $\partial\GG$ of the global ribbon graph $\GG$. This
connected component $(\partial\GG)_m$ of $\partial\GG$ contains at least one other boundary singularity. Orders of all boundary singularities are nonnegative by definition. Thus, the boundary component $(\partial\GG)_m$ corresponds to an endpoint singularity of the original stratum $\cQ_g(n,d,-1^{n+D-4g+4})$ of degree $b_m$ satisfying $b_m\ge -2+(l+1)+(0+1)=l$, see~\eqref{eq:degree:of:singularity}. Since the vertex $v_1$ has local type $-1.1$ and, hence, the local ribbon graph is connected, the principal boundary stratum associated to $v_1$ has the form $\cQ_{g'}(\alpha'_1)=\cQ_{g'}(m'_1, \dots, m'_s, l-2)$ by Formulae~\eqref{eq:Dji} and~\eqref{eq:Q:boundary:stratum}. Note that the singularities $m'_1, \dots, m'_s$ are interior singularities of the ambient stratum. The remaining singularity has degree $l-2$ strictly smaller than the degree $b_m$ of the endpoint singularity of the ambient stratum considered above. Thus, suppressing entries $-1$
in the list $(m'_1, \dots, m'_s, l-2)$ and ordering the remaining
elements in the decreasing order, we get a list $(n',d'_1,\dots,d'_{k'})$
satisfying~\eqref{eq:boundary:strata:data},
and at least one of the inequalities is strict.
For $\ominus$-vertices of local type $-1.1$ the first part of the statement of the Proposition is justified.

Suppose that the vertex $v_1$ has local type $-2.1$, see Figure~\ref{fig:embedded:local:ribbon:graphs}, and that the boundary singularities are $l'$ and $l''$. Since a $-2.1$-vertex can be located only at a graph of type ``b'' (see Figure~\ref{fig:graphs}) and since the graph contains at least one $\circ$-vertex, this $\circ$-vertex has type $\circ2.2$, which implies that both boundary singularities of $v_1$ belong to the same boundary component $(\partial\GG)_m$ of the global ribbon graph $\GG$, and, moreover, that this boundary component contains at least two more boundary singularities. Hence, by Formula~\eqref{eq:degree:of:singularity} the degree $b_m$ of the associated endpoint zero of the ambient stratum is at least $l'+l''+2$. Since the local ribbon graph $\GG(v_1)$
corresponding to a vertex of local type $-2.1$
is connected,  the principal boundary stratum associated to $v_1$ has the form $\cQ_{g'}(\alpha'_1)=\cQ_{g'}(m'_1, \dots, m'_s, l'+l''-2)$ by Formulae~\eqref{eq:Dji} and~\eqref{eq:Q:boundary:stratum}. Thus, for $\ominus$-vertices of local type $-2.1$ the statement of the Proposition is also justified.

Consider now the remaining case when the vertex $v_1$ has local type $-2.2$, see Figure~\ref{fig:embedded:local:ribbon:graphs}, and the boundary singularities are $l'$ and $l''$. As in the previous case, a $-2.2$-vertex can be located only at a graph of type ``b'' (see Figure~\ref{fig:graphs}). Since the graph contains at least one $\circ$-vertex, this $\circ$-vertex has type $\circ2.2$, which this time implies that the two boundary singularities of $v_1$ \textit{do not} belong to the same boundary component of the global ribbon graph $\GG$. Each of these two boundary components contains at least one more boundary singularity. Hence, by Formula~\eqref{eq:degree:of:singularity} the degrees of the associated distinct endpoint zeroes of the ambient stratum are at least $l'$ and $l''$. This time the local ribbon graph $\GG(v_1)$ has two connected components. The boundary stratum associated to $v_1$ has the form $\cQ_{g'}(\alpha'_1)=\cQ_{g'}(m'_1, \dots, m'_s, l'-2,l''-2)$ by Formulae~\eqref{eq:Dji} and~\eqref{eq:Q:boundary:stratum}.
This completes the proof of the first assertion of Proposition~\ref{pr:boundary:stratum:is:simpler}.

% To justify the second assertion, note that for sufficiently large $n$
% (depending on $g$ and $d$), we have $n'>m'_1$ for the distinguished $\ominus$-vertex $v_1$. !!!!!!!!

For a graph of type ``a'', a $\ominus$-vertex $v_1$ carries genus $g'=g$
only if there are no $\oplus$-vertices and the other $\ominus$-vertex $v_0$ carries genus $g_0=0$. This implies
that our graph of type ``a'' represents
a ``dumbbell'' configuration, as in Figure~\ref{fig:dumbbell}.
We have seen that for a dumbbell configuration we have $k'<k$,
so graphs of type ``a'' cannot satisfy~\eqref{eq:all:equalities}.

When a graph of type ``b'' contains a $\circ$-vertex, the genus $g'$
carried by its single $\ominus$-vertex satisfies $g'<g$, so graphs of type ``b'' cannot satisfy~\eqref{eq:all:equalities}.

When a graph of type ``c'' contains at least one $\oplus$-vertex, the genus $g'$ carried by its single $\ominus$ vertex satisfies $g'<g$. It follows from
the classification of decorated ribbon graphs, see Figures~\ref{fig:graphs} and~\ref{fig:embedded:local:ribbon:graphs}, that there are only two types of graphs of type ``c'' which do not have $\oplus$-vertices, namely
a pocket as in Figure~\ref{fig:pocket} and the two-cylinder graph in Figure~\ref{fig:extra:graph:c:with:no:plus:vertices}. The
genus $g'$ of the only $\ominus$-vertex
of the two-cylinder graph in Figure~\ref{fig:extra:graph:c:with:no:plus:vertices}
satisfies $g'=g-1<g$, so
this graph does not fit~\eqref{eq:all:equalities}.

If the pocket has one of the zeroes $d_j$, where $j\in\{1,\dots,k\}$ at the base of the cylinder, then we have seen in Corollaries~\ref{cor:pocket:d} and~\ref{cor:pocket:d:renormalized} that
\[
D'=
\begin{cases}
D-1,&d_j=1,\\
D-2,&d_j\geq2.
\end{cases}
\]
(The case $d_j=1$ is special, because a simple zero gets transformed into a simple pole after a surgery, and this simple pole is recorded among other simple poles, see Remark~\ref{rm:entry:minus:1}.) Thus, no such pockets can satisfy~\eqref{eq:all:equalities}. If the pocket has the zero of degree $n$ at the base of the cylinder, then we have seen in Proposition~\ref{prop:pocket} that for such a pocket all conditions~\eqref{eq:all:equalities} are satisfied and that, moreover, $n'=n-2$. The proof of Proposition~\ref{pr:boundary:stratum:is:simpler} is completed.
\end{proof}

All estimates in the remaining part of this section are
given for fixed $(g,d)$ and for sufficiently large $n$.

\begin{proposition}
\label{prop:negligible:a}
Assuming that all normalized volumes $W_{g'}(n'; d')$ are bounded for every boundary datum satisfying~\eqref{eq:boundary:strata:data} with $g'<g$, or with $g'=g$ and $D'<D$, the total contribution
to the sum in~\eqref{eq:EKZ:in:W:g}
of graphs of type ``a''
distinct from the ones considered in Proposition~\ref{prop:distinguished:dumbbell}
is of order $O_{g,d}\left(\frac{1}{n}\right)$.
\end{proposition}
\begin{proof}
By combination of~\eqref{eq:weight:for:numbering:endpoint:poles} and~\eqref{eq:values:of:e} we have $E(\GG)=1$ for graphs of type ``a''. By~\eqref{eq:no:P:for:a} we have $P(\GG)=1$ for graphs of type ``a''.
By Lemma~\ref{lm:ominus:vertex:associated:to:zero:of:degree:n} exactly one $\ominus$-vertex has $n$ as an interior singularity, or has a boundary singularity depending on $n$. Denote this $\ominus$-vertex by $v_1$
and the other $\ominus$-vertex by $v_0$.

Denote by $g_0$ the genus carried by $v_0$.
Denote by $p_0$ the number of simple poles serving as interior singularities of $v_0$ and by $z_0$ the number of zeroes
serving as interior singularities of $v_0$. Recall that $v_0$ has boundary type
$-1.1$, so it has a single boundary singularity. Hence, the boundary stratum associated to $v_0$ has $p_0+z_0+1$ singularities (counting a marked point if the associated endpoint singularity is a marked point), and hence
\[
 a_0=p_0+z_0+1+2g_0-2\,.
\]
Denote by $b$ the sum of dimensions of Abelian boundary strata, i.e. the sum
of dimensions of strata represented by $\oplus$-vertices. By definition of
$a_\textit{rest}$ involved in the estimate~\eqref{eq:order:of:C}, we have
\begin{equation}
\label{eq:a:rest}
a_\textit{rest}=a_0+b=(p_0+z_0+1+2g_0-2)+b\,.
\end{equation}
Multiplying the contribution given by~\eqref{eq:order:of:C}
by the only remaining nontrivial weight factor $I(\GG)$ given by~\eqref{eq:weight:for:partitionning:interior:poles}
and using~\eqref{eq:orders:of:magnitude:of:weight} for its order of magnitude,
we get
\[
I(\GG)\cdot C_{g}^{\cC_{\mathit{fixed}}}(n;d)
=O(n^{\delta(\cC)})\,.
\]
where
\[
\delta(\cC)=
\begin{cases}
p_0+g_1-(q+1+a_\textit{rest}+g)
&\text{when $n$ is an interior singularity}\,,
\\
p_0+g_1-(q+a_\textit{rest}+g)
&\text{when $n$ is an endpoint singularity}\,.
\end{cases}
\]
Substituting~\eqref{eq:a:rest} we get
\begin{multline}
\label{eq:conditions:on:a}
p_0+g_1-(q+a_\textit{rest}+g)
=p_0+g_1-(q+(p_0+z_0+1+2g_0-2+b)+g)
\\=
1+g_1-q-z_0-2g_0-b-g\,.
\end{multline}
Note that the graph has at least one $\circ$-vertex, so $q\ge 1$. Note also that $g_1\le g$. Thus, the integer number~\eqref{eq:conditions:on:a} is not negative if and only if all the following conditions are satisfied:
\[
\begin{aligned}
\begin{cases}
g_1&=g\,,
\\
q&=1\,,
\\
z_0&=0\,,
\\
g_0&=0\,,
\\
b&=0\,.
\end{cases}
\end{aligned}
\]
in which case it is equal to zero. Note that~\eqref{eq:conditions:on:a} represents $\delta(\cC)$ only when the zero of degree $n$ is an endpoint singularity; otherwise $\delta(\cC)$ is smaller by $1$. We conclude that the above collection of conditions characterizes dumbbell configurations having no $\oplus$-vertices, such that one of the $\ominus$-vertices, say, $v_0$ carries genus $0$ and has no interior zeroes (and has $p_0=d_i+2$ interior poles), while the other $\ominus$-vertex has a boundary singularity of order $n$. Proposition~\ref{prop:negligible:a} is proved.
\end{proof}

\begin{proposition}
\label{prop:negligible:b}
Assuming that all normalized volumes $W_{g'}(n'; d')$ are bounded for every boundary datum satisfying~\eqref{eq:boundary:strata:data} with $g'<g$, or with $g'=g$ and $D'<D$, the following two estimates are valid.

The contribution of the ``figure-eight'' configuration
with a zero of degree $n$ at the base
(as in the left picture of Figure~\ref{fig:Ba:and:Bb})
is of order $O_{g,d}\left(\frac{1}{n}\right)$.

The total contribution
to the sum in~\eqref{eq:EKZ:in:W:g}
of all other graphs of type ``b''
is of order $O_{g,d}\left(\frac{1}{n^2}\right)$.
\end{proposition}
\begin{proof}
By combination of~\eqref{eq:weight:for:numbering:endpoint:poles}
and~\eqref{eq:values:of:e} we have $E(\GG)=1$ for graphs of type ``b''.
By~\eqref{eq:no:I:for:b:c} we have $I(\GG)=1$ for graphs of type ``b''.
Multiplying the contribution given by~\eqref{eq:order:of:C}
by the only remaining weight factor $P(\GG)$ given by~\eqref{eq:weight:for:n:prime:n:double:prime}
and using~\eqref{eq:orders:of:magnitude:of:weight} for its order of magnitude,
we get
\[
P(\GG)\cdot C_{g}^{\cC_{\mathit{fixed}}}(n;d)
=O(n^{\delta(\cC)})\,,
\]
where
\[
\delta(\cC)=
\begin{cases}
g_1-(q+1+a_\textit{rest}+g)
&\text{when $n$ is an interior singularity}\,,
\\
g_1-(q+a_\textit{rest}+g)
&\text{when $n$ is an endpoint singularity carried}
\\
 &\qquad\text{by a $\ominus$-vertex of $-2.2$-type};
\\
1+g_1-(q+a_\textit{rest}+g)
&\text{when $n$ is an endpoint singularity carried}
\\
 &\qquad\text{by a $\ominus$-vertex of $-2.1$-type}.
\end{cases}
\]
Note that in the presence of at least one $\circ$-vertex we have
$q\ge 1$ and $g_1\le g-1$. Thus $\delta(\cC)=-1$ only
if $n$ is an endpoint singularity carried by a $\ominus$-vertex of $-2.1$-type,
and when in addition to that all of the following conditions hold:
\[
\begin{aligned}
\begin{cases}
g_1&=g-1\,,
\\
q&=1\,,
\\
a_\textit{rest}&=0\,.
\end{cases}
\end{aligned}
\]
Otherwise $\delta(\cC)\le-2$.
The condition $a_\textit{rest}=0$ is equivalent to absence of $\oplus$-vertices. We proved that the maximal possible value
$\delta(\cC)=-1$ is achieved only for graphs of type ``b'' having a
single $\ominus$-vertex of type $-2.1$ carrying the zero of degree $n$
as an endpoint singularity. These are exactly the graphs corresponding to
the distinguished ``figure-eight'' configuration.
This completes the proof of Proposition~\ref{prop:negligible:b}.
\end{proof}

\begin{proposition}
\label{prop:negligible:c}
Assuming that all normalized volumes $W_{g'}(n'; d')$ are bounded for every boundary datum satisfying~\eqref{eq:boundary:strata:data} with $g'<g$, or with $g'=g$ and $D'<D$, the total contribution
to the sum in~\eqref{eq:EKZ:in:W:g}
of graphs of type ``c''
distinct from the ones considered in Corollaries~\ref{cor:pocket:renormalized}
and~\ref{cor:pocket:d:renormalized}
is of order $O_{g,d}\left(\frac{1}{n^2}\right)$.
\end{proposition}

\begin{figure}[htb]
  \centering
\hspace*{25pt}
\includegraphics[scale=0.5]{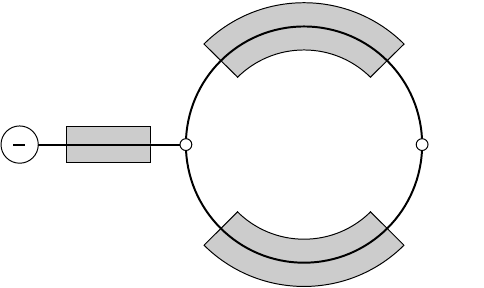}
\vspace*{10pt}
\caption{``Non-pocket'' graph of type ``c'' without $\oplus$-vertices. There are two cylinders, so $q=2$; the $\ominus$-vertex carries genus $g_1=g-1$; there are no endpoint poles, so $e=0$.}
  \label{fig:extra:graph:c:with:no:plus:vertices}
\end{figure}

\begin{proof}
By~\eqref{eq:no:I:for:b:c} and~\eqref{eq:no:P:for:a} we have
$I(\GG)=P(\GG)=1$.

Multiplying the contribution given by~\eqref{eq:order:of:C}
by the only remaining weight factor $E(\GG)$ given by combining~\eqref{eq:weight:for:numbering:endpoint:poles}
and~\eqref{eq:values:of:e} and using~\eqref{eq:orders:of:magnitude:of:weight} for its order of magnitude, we get for non-pocket configurations
\[
E(\GG)\cdot C_{g}^{\cC_{\mathit{fixed}}}(n;d)
=O(n^{\delta(\cC)})\,,
\]
where
\[
\delta(\cC)=
\begin{cases}
e(\GG)+g_1-(q+1+a_\textit{rest}+g)
&\text{when $n$ is an interior singularity}\,,
\\
e(\GG)+g_1-(q+a_\textit{rest}+g)
&\text{when $n$ is an endpoint singularity}\,.
\end{cases}
\]
Recall that by~\eqref{eq:values:of:e} we have $e(\GG)\le 2$.
If $a_\textit{rest}>0$, at least one $\oplus$-vertex is present, so $g_1\leq g-1$ and $a_\textit{rest}\ge 2$. Together with $q\geq1$ and $e\leq2$, this gives $\delta(\cC)\le -2$.

Without Abelian vertices, the graph classification leaves two cases. An ordinary pocket has
\[
e=2,\qquad q=1,\qquad g_1=g,\qquad a_\textit{rest}=0.
\]
Its exponent is $1$ when the boundary singularity of the $\ominus$-vertex
has order $n$ and $0$ when it has order $d_j$ for some $j\in\{1,\dots,k\}$. In the first case we get the volume $W_g(n-2;d)$ which is not assumed
to be bounded.

The remaining graph of type ``c'' having no $\oplus$-vertices is presented in Figure~\ref{fig:extra:graph:c:with:no:plus:vertices}. It has a $\circ$-vertex
of valency $3$ and one more $\circ$-vertex on its loop. For this graph we get
\[
e=0,\qquad q=2,\qquad g_1=g-1,\qquad a_\textit{rest}=0.
\]
Consequently $\delta(\cC)=-3$ if $n$ is a boundary singularity and $-4$ if $n$ is an interior singularity. Thus every non-pocket family of graphs of type ``c'' contributes at most $O_{g,d}(n^{-2})$.
\end{proof}

% ###############################################################
% ###############################################################
% ###############################################################
\section{Preparatory lemmas from calculus}
\label{s:Preparatory:lemmas}

In this section we establish asymptotic estimates with error $O(n^{-1})$ for two abstract recurrence relations. In the next section, we substitute the contributions computed in Sections~\ref{s:Contributions:of:graphs:a}--\ref{s:Dominant:and:negligible:contributions} into the main recurrence relation~\eqref{eq:EKZ:in:W:g} of Theorem~\ref{th:main:recurrence:relation}. After suitable rescaling, the recurrence for minimal strata takes the form considered in Lemma~\ref{lm:lemma:2:in:calculus}, while the recurrences for non-minimal strata take the form considered in Lemma~\ref{lm:lemma:in:calculus}.

Throughout this section, all sequences are real, and all asymptotic estimates are understood as $n\to+\infty$. Constants implicit in the $O$-notation may depend on the fixed parameters and initial values of the recurrence, and need not be the same at different occurrences.

\begin{lemma}
\label{lm:lemma:2:in:calculus}
Suppose that a sequence $V(n)$ satisfies
\[
V(n)=(1+\alpha_n)V(n-2)+\beta_n
\]
for all sufficiently large integers $n$, where
\[
\alpha_n=O\left(\frac1{n^2}\right),
\qquad
\beta_n=O\left(\frac1{n^2}\right).
\]
Then there exist finite constants $C^{\mathit{even}}$ and $C^{\mathit{odd}}$ such that
\begin{align*}
V(2n)&=C^{\mathit{even}}+O\left(\frac1n\right),\\
V(2n+1)&=C^{\mathit{odd}}+O\left(\frac1n\right).
\end{align*}
If, in addition, $V(n)>0$ and $\beta_n\geq0$ for all sufficiently large $n$, then both constants are strictly positive.
\end{lemma}
\begin{proof}
Fix a parity class and choose an integer $N$ in this class sufficiently large that the recurrence holds and $|\alpha_n|\leq\frac12$ for all $n\geq N$. For $n\geq N$ in the chosen parity class, set
\[
P_N:=1,
\qquad
P_n:=
\prod_{\substack{N<m\leq n\\m\equiv N(\mathrm{mod}\,2)}}
(1+\alpha_m)\,.
\]
Since $\sum_m|\alpha_m|<+\infty$, the sequence $P_n$ converges to a finite strictly positive limit $P_\infty$. In particular, $P_n$ is bounded above and bounded away from zero.

Dividing the recurrence by $P_n$ and summing gives
\[
\frac{V(n)}{P_n}
=
V(N)+
\sum_{\substack{N<m\leq n\\m\equiv N\pmod2}}
\frac{\beta_m}{P_m}.
\]
The series on the right converges absolutely. Hence $V(n)/P_n$, and therefore $V(n)$, has a finite limit along the chosen parity class. Denote the latter limit by $C$.

Since $V(n)$ is bounded along this class, the recurrence implies
\[
|V(n)-V(n-2)|
=
|\alpha_nV(n-2)+\beta_n|
\leq \frac{M}{n^2}
\]
for a suitable constant $M$ and all sufficiently large $n$ in the class. Consequently,
\begin{align*}
|V(n)-C|
&\leq
\sum_{j=1}^{\infty}
|V(n+2j)-V(n+2j-2)|\\
&\leq
M\sum_{j=1}^{\infty}\frac1{(n+2j)^2}
\leq \frac{M}{2n}.
\end{align*}
Applying this argument to both parity classes proves the stated asymptotic estimates.

Under the additional positivity assumptions, choose $N$ sufficiently large that $V(N)>0$ and $\beta_m\geq0$ for $m>N$. The displayed summation formula then gives
\[
\frac{V(n)}{P_n}\geq V(N)>0.
\]
Passing to the limit yields $C\geq P_\infty\cdot V(N)>0$. This proves strict positivity in each parity class.
\end{proof}

\begin{lemma}
\label{lm:lemma:in:calculus}
Suppose that a sequence $U(n)$ satisfies
\[
U(n)=\bigl(1-a(n)\bigr)\cdot U(n-2)+a(n)\cdot r(n)
\]
for all sufficiently large integers $n$, where
\[
a(n)=\frac cn+O\left(\frac1{n^2}\right),
\qquad
r(n)=1+O\left(\frac1n\right),
\qquad c>2.
\]
Then
\[
U(n)=1+O\left(\frac1n\right).
\]
In particular, the two parity subsequences converge to the same limit $1$, with the stated error estimate.
\end{lemma}
\begin{proofwithmanualqed}
Set $Z(n):=U(n)-1$. The recurrence becomes
\[
Z(n)=\bigl(1-a(n)\bigr)\cdot Z(n-2)
+a(n)\cdot\bigl(r(n)-1\bigr)\,.
\]
Choose a constant $c_2$ with $2<c_2<c$. By the assumptions, there exist an integer $N>2$ and a constant $B>0$ such that, for all $n\geq N$,
\[
\frac{c_2}{n}\leq a(n)<1,
\qquad
\left|a(n)\cdot\bigl(r(n)-1\bigr)\right|
\leq \frac{B}{n^2}.
\]
It follows that
\[
|Z(n)|
\leq
\left(1-\frac{c_2}{n}\right)|Z(n-2)|
+\frac{B}{n^2}\,.
\]

Choose $K>0$ sufficiently large so that for $N$ chosen above we have
\[
|Z(N)|\leq\frac K N,
\qquad
|Z(N+1)|\leq\frac K{N+1},
\qquad
K\cdot(c_2-2)\geq B\,.
\]
We prove by induction, separately in the two parity classes, that
\[
|Z(n)|\leq\frac K n
\qquad\text{for all }\  n\geq N\,.
\]
The two initial inequalities hold by the choice of $K$. For $n\geq N+2$, the induction hypothesis gives
\[
|Z(n)|
\leq\left(1-\frac{c_2}{n}\right)\frac K{n-2}+\frac B{n^2}
=
\frac K n-\frac{K\cdot(c_2-2)}{n(n-2)}+\frac B{n^2}\\
\leq \frac K n\,.
\]
This completes the induction and proves that
\[
U(n)-1=Z(n)=O\left(\frac1n\right)\,.\qed
\]
\end{proofwithmanualqed}

% ###############################################################
% ###############################################################
% ###############################################################
\section{Volume asymptotics}
\label{s:Volume:asymptotics}

In this section we collect formulae for contributions of individual configurations obtained in Sections~\ref{s:Contributions:of:graphs:a}--\ref{s:Contributions:of:graphs:c} to rewrite our main recurrence relation~\eqref{eq:EKZ:in:W:g} in a much more explicit form in terms of renormalized volumes $W_g(n;d)$.

For the two series of non-varying strata a proof of Theorem~\ref{th:n:1} reduces to the routine verification of the resulting exact closed recurrence relations~\eqref{eq:realation:between:volumes:minimal:g:1} and~\eqref{eq:n:1:init} below.

For other strata the recurrence relation~\eqref{eq:EKZ:in:W:g} is exact with the actual Lyapunov sum. Using Fougeron's bound~\eqref{eq:Fougeron:bound} for that sum gives an asymptotic recurrence with a controllable error term. By the nature of the minimal stratum, certain
configurations are not realizable for it, which makes the sum over configurations in~\eqref{eq:EKZ:in:W:g} very special.
We reduce the recurrence for the minimal stratum to the one described in Lemma~\ref{lm:lemma:2:in:calculus}.
We reduce the generic recurrence for other strata to the one described in Lemma~\ref{lm:lemma:in:calculus}.

%------------------------------------------------------------------------------
\subsection{%
  \texorpdfstring{%
    \for{toc}{Minimal strata $\mathcal Q_1(n,-1^{n})$ in genus $1$}%
    \except{toc}{%
      \boldmath Minimal strata $\mathcal Q_1(n,-1^{n})$ in genus $1$%
    }%
  }{Minimal strata Q1(n, with n simple poles) in genus 1}%
}
\label{ss:n:g:1}

For the sake of completeness we start by presenting the value of the area Siegel--Veech constant and the relation between the volumes
in terms of the classical Masur--Veech volumes. They might
serve for potential alternative calculations in terms of intersection numbers.

\begin{lemma}
\label{lm:SV:minimal:g:1}
The area Siegel--Veech constant of the minimal stratum
in genus $1$ with $n\ge 5$ satisfies the following relation:
\begin{multline}
\label{eq:c:area:minimal}
c_{area}(\cQ_1(n,-1^n))=\frac{1}{\Vol\cQ_1(n,-1^n)}
\Bigg(
\frac{n}{4}\cdot\Vol\cQ_1(n-2,-1^{n-2})\\
+\frac{n-2}{n}\cdot\Vol\cQ_0(n-4,-1^n)+
\frac{n - 4}{2(n-2)(n-3)}\cdot \Vol\cQ_0(n-6,-1^{n-2})\cdot\Vol\cH(0)
\Bigg)\,.
\end{multline}
\end{lemma}
\begin{proof}
The list of graphs realizable for the minimal stratum in genus $1$
is given by Corollary~\ref{cor:configurations:minimal:genus:1}.
It consists of the three graphs presented in Figure~\ref{fig:minimal:genus:1}.
The contribution of the left graph~\textit{I} (representing the ``figure-eight'' configuration) is computed in~\eqref{eq:carea:figure:eight},
where we have to let $g=1, k=0, D=0$. After cancelling
out the common factor $(n-1)$ in the numerator and in the denominator,
we get the second term in the sum in~\eqref{eq:c:area:minimal}.

The contribution of the middle graph~\textit{II}
in Figure~\ref{fig:minimal:genus:1}
(representing the ``pocket'' configuration)
is computed in~\eqref{eq:pocket}, where, as before, we let $g=1, k=0, D=0$. After cancelling
out the common factor $n(n-1)$ in the numerator and in the denominator,
we get the first term in the sum in~\eqref{eq:c:area:minimal}.

Finally,
the contribution of the right graph~\textit{III}
in Figure~\ref{fig:minimal:genus:1}
(representing the ``pocket-through-torus'' configuration)
is computed in~\eqref{eq:pocket:through:torus:n:g:1} and
gives the third term in the sum  in~\eqref{eq:c:area:minimal}.
\end{proof}

\begin{corollary}
For $n\ge 5$, the
Masur--Veech volumes of the minimal strata in genus $1$ satisfy the following relation:
\begin{multline}
\label{eq:realation:between:volumes:minimal:g:1}
\Vol\cQ_1(n,-1^n)=\frac{12(n+2)}{n^2+n+24}\cdot\frac{\pi^2}{3}
\cdot\Bigg(
\frac{n}{4}\cdot\Vol\cQ_1(n-2,-1^{n-2})
\\+
\frac{n-2}{n}\cdot\Vol\cQ_0(n-4,-1^n)+
\frac{n - 4}{2(n-2)(n-3)}\cdot \Vol\cQ_0(n-6,-1^{n-2})\cdot\Vol\cH(0)
\Bigg)\,.
\end{multline}
\end{corollary}
\begin{proof}
By~\eqref{eq:lambda1:plus:n} we have
$\lambda_1^+\big(\cQ_1(n,-1^n)\big)=\frac{2}{n+2}$. For the stratum $\cQ(n,-1^n)$ we have
\[
g=1,\qquad k=D=\Delta=0
\]
(see~\eqref{eq:D:and:Delta}), so by~\eqref{eq:minus:1:24:sum}
for the stratum $\cQ_1(n,-1^n)$ we get
the following value of the weighted sum involved in~\eqref{eq:sum:of:lambda:plus}:
\[
-\frac{1}{24}\sum_{i=1}^p \frac{m_i(m_i+4)}{m_i+2}
=\frac{1}{12}\left(n-1+\frac{2}{n+2}\right)
\]
and hence
\begin{multline*}
\lambda_1^+\big(\cQ_1(n,-1^n)\big)-\frac{1}{24}\sum_{i=1}^p \frac{m_i(m_i+4)}{m_i+2}
\\=
\frac{1}{12}\left(n-1+\frac{26}{n+2}\right)
=\left(\frac{12(n+2)}{n^2+n+24}\right)^{-1}\,,
\end{multline*}
which is the inverse of the first factor  in~\eqref{eq:realation:between:volumes:minimal:g:1}.
The relation~\eqref{eq:realation:between:volumes:minimal:g:1}
now follows from~\eqref{eq:sum:of:lambda:plus}, where
$\carea(\cQ_1(n,-1^n))$ is replaced by the value~\eqref{eq:c:area:minimal}
given in Lemma~\ref{lm:SV:minimal:g:1}.
\end{proof}

We are ready to prove the first part of Theorem~\ref{th:n:1}.

\begin{proof}[Proof of Formula~\eqref{eq:volume:n}]
The stratum $\cQ_1(1,-1)$ is empty, see~\cite{Masur:Smillie}.
Exact numerical values of $\Vol\cQ_1(n,-1^n)$ for $2\le n\le 10$ are given in~\cite[arXiv:\allowbreak 1501.01611v4, Appendix~A]{Goujard:vol}.
A direct check shows that they fit Formula~\eqref{eq:volume:n}.
Let us now prove Formula~\eqref{eq:volume:n}
for arbitrary $n$.

Evaluating quantities $D,\Delta$ and $\Lambda^+_g$
defined in~\eqref{eq:D:and:Delta} and~\eqref{eq:Lambda:plus}
for the minimal stratum in genus $1$ we get
\[
D=0,\qquad k=0,\qquad \Delta=0,\qquad\Lambda^+_1=\frac{2}{n+2}\,,
\]
where the first three equalities are trivial and
where the
rightmost equality corresponds to~\eqref{eq:lambda1:plus:n}, which is a Theorem of D.~Chen and M.~Möller~\cite{Chen:Moeller}. Using the list of
configurations given in Corollary~\ref{cor:configurations:minimal:genus:1}
and illustrated in Figure~\ref{fig:minimal:genus:1}
we get the following form of Equation~\eqref{eq:EKZ:in:W:g} specified
to the stratum $\cQ_1(n,-1^n)$:
\begin{equation*}
%\label{eq:W:1:through:C}
W_1(n)
=
\frac{n+2}{n^2+n+24}
\left(
C_1^{\textit{figure-eight}}(n)
+
C_1^{\textit{pocket}}(n)
+
C_1^{\textit{pocket-through-torus}}(n)
\right)\,.
\end{equation*}
The three normalized Siegel--Veech constants present in this latter equation are evaluated in~\eqref{eq:figure:eight:g:1},
\eqref{eq:pocket:n:g:1:renormalized},
and~\eqref{eq:pocket:through:torus:n:g:1:renormalized}
respectively. Factoring out the common factor $(n+1)/n$, we obtain
\begin{equation}
\label{eq:recurrence:W:1:factored}
W_1(n)
=\frac{(n+2)(n+1)}{n(n^2+n+24)}
\left(
(n-2)W_1(n-2)
+
\frac{8(n-1)}{n^2}
+
\frac{4(n-1)}{3n(n-2)^2}
\right).
\end{equation}

The closed volume formula suggested in~\eqref{eq:volume:n} gives
\begin{equation}
\label{eq:W1:answer}
W_1(n)
=\frac{\Vol\cQ_1(n,-1^n)}{n\cdot v(n)}
=\frac{(n+2)(n+1)(n-1)}{3n^3}\,,
\end{equation}
and
\[
W_1(n-2)
=\frac{\Vol\cQ_1(n-2,-1^{n-2})}{(n-2)\cdot v(n-2)}
=
\frac{n(n-1)(n-3)}{3(n-2)^3}\,.
\]
Substituting $W_1(n)$ and $W_1(n-2)$ in~\eqref{eq:recurrence:W:1:factored}
and simplifying we obtain an identity.
This shows that
the expression~\eqref{eq:volume:n}
satisfies the recursion.
Formula~\eqref{eq:volume:n} from Theorem~\ref{th:n:1}
is proved.
\end{proof}
\begin{remark}
In the way we presented the proof, the solution~\eqref{eq:W1:answer} for the recursion comes out of nothing. What was suppressed from the proof is that after a sequence of convenient changes of variables, the recursion formula progressively becomes considerably simpler. The answer was originally derived for the simplified recursion. We suppressed these obsolete calculations to avoid overloading the paper.
\end{remark}

\begin{remark}
%\label{rm:Q:n:fits:recursion}
   %
The coefficient of $W_1(n-2)$ in~\eqref{eq:recurrence:W:1:factored}
satisfies the identity
\[
\frac{(n+2)(n+1)(n-2)}{n(n^2+n+24)}
=
1-\frac{4(7n+1)}{n(n^2+n+24)}.
\]
Hence the recurrence~\eqref{eq:recurrence:W:1:factored}
can be represented as
\[
W_1(n)
=
\left(
1-\frac{28}{n^2}+O\left(\frac1{n^3}\right)
\right)W_1(n-2)
+
\frac8{n^2}
+
O\left(\frac1{n^3}\right).
\]
as $n\to+\infty$.
In particular, there is no term of order $\frac{1}{n}$ either in the multiplicative coefficient or in the additive term. The recurrence fits the
general scheme of Lemma~\ref{lm:lemma:2:in:calculus}.
\end{remark}

%-------------------------------------------------------------------
\subsection{%
  \texorpdfstring{%
    \for{toc}{Strata $\mathcal Q_1(n,1,-1^{n+1})$}%
    \except{toc}{%
      \boldmath Strata $\mathcal Q_1(n,1,-1^{n+1})$%
    }%
  }{Strata Q1(n,1, with n+1 simple poles)}%
}
\label{ss:n:1:g:1}

We proceed in the same way as in the minimal stratum.

\begin{proposition}
\label{prop:SV:n:1}
The
Masur--Veech volumes of the strata $\cQ_1(n,1,-1^{n+1})$
in genus $1$ satisfy relation~\eqref{eq:n:1:init}
for $n\ge 5$.
The sum of seven terms in the parentheses coincides
with the product
$\Vol(\cQ_1(n,1,-1^{n+1}))\cdot\carea(\cQ_1(n,1,-1^{n+1}))$.
\begin{multline}
\label{eq:n:1:init}
\Vol\cQ(n,1,-1^{n+1})
=\frac{36 (n+2)}{3 n^2+5 n+76}\cdot \frac{\pi^2}{3}
\Bigg(
   \underbrace{\frac{n^2}{4(n+2)}\cdot\Vol\cQ(n-2,1,-1^{n-1})}_{v} % Ca
\\
+ \underbrace{\frac{n }{4 (n+2)}\cdot\Vol\cQ(n,-1^n)}_{vi} % Cb
+ \underbrace{\frac{n}{24(n+2)}\cdot \Vol\cQ(-1^4)\cdot \Vol\cQ(n-2,-1^{n-2})}_{i} % Aa
\\+
   \underbrace{\frac{(n-2)(n-1)}{(n+1)(n+2)}\cdot\Vol\cQ(n-4,1,-1^{n+1})}_{iii} % Ba
\\+
   \underbrace{\frac{2 n}{(n+2)(n+1)}\Vol\cQ(n-2,-1^{n+2})}_{iv} % Bb
\\+
   \underbrace{\frac{(n-4)}{2(n+2)(n-1)}\Vol\cH(0)\Vol\cQ(n-6,1,-1^{n-1})}_{vii} % Cc
\\+
   \underbrace{\frac{(n-4)}{12(n-3)(n-2)(n+2)}\cdot\Vol\cQ(-1^4)\cdot\Vol\cH(0)\cdot\Vol\cQ(n-6,-1^{n-2})}_{ii} % Ab
\Bigg)\,.
\end{multline}
\end{proposition}
\begin{proof}
By~\eqref{eq:lambda1:plus:n:1} we have
$\lambda_1^+\big(\cQ_1(n,1,-1^{n+1})\big)=\frac{2}{n+2}$.
We have
\[
g=k=D=1,\qquad \Delta=2\cdot\frac{1}{(1+2)}=\frac{2}{3}
\]
for the stratum $\cQ(n,1,-1^{n+1})$,
so by~\eqref{eq:minus:1:24:sum} we get
\[
-\frac{1}{24}\sum_{i=1}^p \frac{m_i(m_i+4)}{m_i+2}
=\frac{1}{12}\left(n-1+\frac{2}{3}+\frac{2}{n+2}\right)
\]
and hence
\[
\lambda_1^+\big(\cQ_1(n,1,-1^{n+1})\big)-\frac{1}{24}\sum_{i=1}^p \frac{m_i(m_i+4)}{m_i+2}=
\left(\frac{36 (n+2)}{3 n^2+5 n+76}\right)^{-1}\,,
\]
which is the inverse of the first rational factor $36(n+2)/(3n^2+5n+76)$ in~\eqref{eq:n:1:init}.

The list of graphs contributing to the area Siegel--Veech constant for the stratum $\cQ(n,1,-1^{n+1})$
is given by Proposition~\ref{prop:configurations:with:two:zeroes}.
It consists of the seven graphs presented in Figure~\ref{fig:graphs:n:1:genus:1}. Thus, the desired relation~\eqref{eq:n:1:init} would follow from~\eqref{eq:sum:of:lambda:plus} as soon as we prove that the expression in the brackets in~\eqref{eq:n:1:init} is
the product
$\Vol\cQ(n,1,-1^{n+1})\cdot \carea(\cQ(n,1,-1^{n+1}))$. To verify this we collect contributions of all individual graphs.

The contribution of the graph~\textit{i} (representing the ``dumbbell'' configuration) is computed in~\eqref{eq:n:1:i}.

The contribution of the graph~\textit{ii} (representing the ``dumbbell-through-torus'' configuration) is computed in~\eqref{eq:n:1:ii}.

The contribution of the graph~\textit{iii} (representing the ``figure-eight'' configuration) is computed in~\eqref{eq:carea:figure:eight},
where we let $g=k=D=d_1=1$, and $d=(1)$.

The contribution of the graph~\textit{iv} (representing the ``parallelogram'' configuration) is computed in~\eqref{eq:carea:parallelogram},
where we let $g=k=D=j=d_1=1$, and $d=(1)$.

The contribution of the graph~\textit{v} (representing the ``pocket'' configuration with the zero of degree $n$ at the base of the cylinder) is computed in~\eqref{eq:pocket},
where we let $g=k=D=d_1=1$, and $d=(1)$.

The contribution of the graph~\textit{vi} (representing the ``pocket'' configuration with the zero of degree $1$ at the base of the cylinder) is computed in~\eqref{eq:pocket:d},
where we let $g=k=D=j=d_1=1$ and $d=(1)$. The new simple pole is included in the pole list according to Remark~\ref{rm:entry:minus:1}.

Finally,
the contribution of the graph~\textit{vii}
(representing the ``pocket-through-torus'' configuration)
is computed in~\eqref{eq:pocket:through:torus:n:1:g:1}.

Proposition~\ref{prop:SV:n:1} is proved.
\end{proof}

We are ready to complete the proof of Theorem~\ref{th:n:1}.

\begin{proof}[Proof of Formula~\eqref{eq:volume:n:1}]
Numerical values of volumes $\cQ(n,1,-1^{n+1})$ for $n\le 8$ are given in~\cite[arXiv:1501.01611v4, Appendix~A]{Goujard:vol}. Direct checks show that they fit the general Formula~\eqref{eq:volume:n:1}. Let us prove now Formula~\eqref{eq:volume:n:1} for any $n$ by recursion in $n\ge 5$.

Evaluating the quantities $D,\Delta$ and $\Lambda^+_g$
defined in~\eqref{eq:D:and:Delta} and~\eqref{eq:Lambda:plus}
for the stratum $\cQ_1(n,1,-1^{n+1})$ we get
\[
D=1,\qquad k=1,\qquad \Delta=\frac{2}{1+2},\qquad\Lambda^+_1=\frac{2}{n+2}\,,
\]
so the first line of~\eqref{eq:EKZ:in:W:g} gives
\begin{multline}
\label{eq:Wn1:init}
W_1(n;1)=
\frac{3(n+2)}{3n^2+5n+76}
\\
\times\big(
C_1^{\textit{pocket,n}}(n;1)
+C_1^{\textit{pocket,1}}(n;1)
+C_1^{dumbbell,n,1}(n;1)
\\
+C_1^{\textit{figure-eight},n}(n;1)
+C_1^{\textit{parallelogram},n}(n;1)
\\
+C_1^{\textit{pocket-through-torus}}(n;1)
+C_1^{\textit{dumbbell-through-torus},n,1}(n;1)
\big)\,,
\end{multline}
where we followed the order of terms in~\eqref{eq:n:1:init}.
The values of the corresponding renormalized contributions were
computed in the following formulae:
\[
\begin{array}{lll}
C_1^{dumbbell,n,1}(n;1)&
\textit{graph i}&
\eqref{eq:a:graphs:contributions:Q:n:1}\,,
\\
C_1^{\textit{dumbbell-through-torus},n,1}(n;1)&
\textit{graph ii}&
\eqref{eq:a:graphs:contributions:Q:n:1}\,,
\\
C_1^{\textit{figure-eight},n}(n;1)&
\textit{graph iii}&
\eqref{eq:figure:eight:g:1}\,,
\\
C_1^{\textit{parallelogram},n}(n;1)&
\textit{graph iv}&
\eqref{eq:parallelogram:g:1}\,,
\\
C_1^{\textit{pocket,n}}(n;1)&
\textit{graph v}&
\eqref{eq:pocket:at:n:n:1:g:1:renormalized}\,,
\\
C_1^{\textit{pocket,1}}(n;1)&
\textit{graph vi}&
\eqref{eq:pocket:at:1:n:1:g:1:renormalized}\,,
\\
C_1^{\textit{pocket-through-torus}}(n;1)&
\textit{graph vii}&
\eqref{eq:pocket:through:torus:n:1:g:1:renormalized}\,.
\end{array}
\]
Substituting the seven normalized contributions respecting the order,
coupling the numerator $3(n+2)$ of the fraction in~\eqref{eq:Wn1:init}
with each of the seven additive terms and cancelling out common
factors in numerators and in denominators of the resulting fractions,
we obtain
\begin{multline}
\label{eq:recurrence:W:1:1:expanded}
W_1(n;1)
=
\frac{1}{3n^2+5n+76}
\\
\times\Bigg(
3(n+1)(n-2)\cdot W_1(n-2;1)
+
6n\cdot W_1(n)
+
\frac{2(n+1)(n-2)}{n}\cdot W_1(n-2)
\\
+
\frac{24(n-1)^2}{n^2}
+
\frac{96}{n}
+
\frac{4(n+1)(n-3)}{n^2(n-2)}
+
\frac{8(n+1)(n-1)}{3n^2(n-2)^2}
\Bigg)\,.
\end{multline}
Replacing $W_1(n;1)$ and $W_1(n-2;1)$ in~\eqref{eq:recurrence:W:1:1:expanded} by the explicit expression implied by Formula~\eqref{eq:volume:n:1}, replacing $W_1(n)$ and $W_1(n-2)$ in~\eqref{eq:recurrence:W:1:1:expanded} by the expression implied by Formula~\eqref{eq:volume:n}, which we already proved, and simplifying we obtain an identity. This shows that the expression~\eqref{eq:volume:n:1} satisfies the recursion and Theorem~\ref{th:n:1} is proved.
\end{proof}
\begin{remark}
Having proved Formulae~\eqref{eq:volume:n} and~\eqref{eq:volume:n:1} we see that $W_1(n)=O(1)$ and $W_1(n;1)=O(1)$ as $n\to\infty$.
Analyzing the seven additive terms in the brackets in~\eqref{eq:recurrence:W:1:1:expanded} we
conclude that the corresponding contributions (before division by $3n^2+5n+76$)
are respectively $O(n^2), O(n), O(n), O(1), O(n^{-1}), O(n^{-1}), O(n^{-2})$.
This explains the order in which we placed the additive terms in~\eqref{eq:n:1:init}.
\end{remark}

\begin{remark}
%\label{rm:Q:n:1:fits:recursion}
Let $C_1:=C_1^{\mathit{even}}=C_1^{\mathit{odd}}=1/3$. Replacing $W_1(n)$ in~\eqref{eq:recurrence:W:1:1:expanded} by the explicit expression implied by Formula~\eqref{eq:volume:n}, we get a recurrence relation for $W_1(n;1)$ in terms of $W_1(n-2;1)$ and of an explicit rational function in $n$. After passing to a normalized variable $V_1(n;1):=\frac{W_1(n;1)}{C_1}=3\cdot W_1(n;1)$, the multiplicative coefficient in front of $V_1(n-2;1)$ in this recursion becomes
\[
\frac{3(n+1)(n-2)}{3n^2+5n+76}
=
1-\frac{8n+82}{3n^2+5n+76}
=
1-\frac{8}{3n}+O\left(\frac1{n^2}\right),
\]
while the additive term becomes
\[
\frac{2(n+1)(4n^3+31n^2-43n-76)}
{n^2(n-2)(3n^2+5n+76)}
=\frac{8}{3n}+O\left(\frac1{n^2}\right).
\]
Thus
\[
V_1(n;1)
=
\left(
1-\frac{8}{3n}+O\left(\frac1{n^2}\right)
\right)V_1(n-2;1)
+
\frac{8}{3n}+O\left(\frac1{n^2}\right)
\]
and we see that the recurrence fits the
general scheme of Lemma~\ref{lm:lemma:in:calculus}.
\end{remark}

%--------------------------------------------------------------
\subsection{Minimal stratum in arbitrary genus}
\label{ss:Minimal:stratum:in:arbitrary:genus}

\begin{proposition}
\label{prop:minimal:any:g}
For any genus $g\ge 0$ there exist strictly positive constants $C_g^{\mathit{even}}$
and $C_g^{\mathit{odd}}$ such that
\begin{equation}
\label{eq:W:minimal:genus:g}
W_g(n)=
\begin{cases}
C_g^{\mathit{even}}+O_g\left(\frac{1}{n}\right)
        &\text{for even }n\,,
\\
C_g^{\mathit{odd}}+O_g\left(\frac{1}{n}\right)
        &\text{for odd }n\,.
\end{cases}
\end{equation}
\end{proposition}
\begin{proof}
For $g=0$ the statement follows from~\eqref{eq:W:in:genus:0}.
For $g=1$ the statement follows from~\eqref{eq:volume:n}.
Assume that it is true for
all $g'<g$ and let us prove it for genus $g$.

For the minimal stratum we have $D=k=\Delta=0$, see~\eqref{eq:D:and:Delta},
so our main recurrence relation~\eqref{eq:EKZ:in:W:g} takes the form
\begin{equation}
\label{eq:EKZ:minimal:g}
W_g(n)
=\frac{1}{n+5-6g
+O_g\left(\frac{1}{n}\right)}
\sum_{\cC} C_g^{\cC}(n)\,.
\end{equation}
Decorated ribbon graphs which appear in the minimal stratum $\cQ_g(n,-1^{n-4g+4})$
are described by Proposition~\ref{pr:list:of:configurations:in:minimal:stratum}
and presented in Figure~\ref{fig:minimal:genus:g}. Only graphs of types
``b'' and ``c'' are present in the list.

By the induction assumption $W_{g-1}(n-4)=O(1)$.
The contribution of the figure-eight configuration (represented in the left picture in Figure~\ref{fig:Ba:and:Bb})
is given in~\eqref{eq:figure:eight:genus:g:renormalized}, where
for the minimal stratum
we let $d=\emptyset$
and $k=D=0$.
In renormalized variables it is negligible:
\begin{multline*}
C_g^{\textit{figure-eight},n}(n)
=
\frac{4}{\pi^2}
\left(\frac{n-4}{n}\right)^{g-1}
\frac{(n+1)(n-1)^2}
{n^2(n-2g+1)(n-2g+2)}
W_{g-1}(n-4)\,.
\\
=O(1)\cdot O(n^{-1})\cdot O(1)=O(n^{-1})\,.
\end{multline*}

A parallelogram configuration does not exist in the minimal stratum.
By Proposition~\ref{prop:negligible:b}
combined with the induction assumption, the total contribution of all other graphs of type ``b'' to the sum~\eqref{eq:EKZ:minimal:g} is of order $O_g(n^{-2})$. Thus, the total contribution of graphs of
type ``b'' is of order $O(n^{-1})$.

Among graphs of type ``c'' we have a pocket with a zero of degree $n$
at the base of the cylinder, see Figure~\ref{fig:pocket}. Its
contribution in renormalized variables is given by Formula~\eqref{eq:Ca:total:renormalized:g}, where for the minimal stratum
we let $d=\emptyset$
and $k=D=0$,
\[
C_g^{\mathit{pocket}}(n)
=(n+1)\cdot
\frac{(n-4g+4)(n-4g+3)}
{(n-2g+2)(n-2g+1)}
\left(\frac{n-2}{n}\right)^g
W_g(n-2)\,.
\]
By Proposition~\ref{prop:negligible:c} combined with the induction assumption,
the total contribution of all other graphs of type ``c'' is of order $O_g(n^{-2})$. Thus
Formula~\eqref{eq:EKZ:minimal:g} can be rewritten as
\begin{multline*}
W_g(n)=\frac{(n+1)}{\left(n+5-6g+O(n^{-1})\right)(n-2g+2)(n-2g+1)}
\\
\times\Bigg(
(n-4g+4)(n-4g+3)\cdot\left(\frac{n-2}{n}\right)^g\cdot W_g(n-2)
+O(1)\Bigg)\,.
\end{multline*}
We can write the latter recurrence in the form
\begin{equation}
\label{eq:W:g:n:recursion}
W_g(n)=A_g(n)W_g(n-2)+O\left(\frac{1}{n^2}\right)\,,
\end{equation}
where
\[A_g(n)=
\frac{(n+1)(n-4g+4)(n-4g+3)}
{(n+5-6g+O(n^{-1}))(n-2g+2)(n-2g+1)}
\cdot
\frac{(n-2)^g}{n^g}
\]

For fixed $g$, expanding the rational factor in $A_g(n)$ gives
\begin{multline*}
\frac{(n+1)(n-4g+4)(n-4g+3)}
{(n+5-6g+O_g(n^{-1}))(n-2g+2)(n-2g+1)}
\\
=
\frac{1+\frac{8-8g}{n}+O_g\left(\frac1{n^2}\right)}
{1+\frac{8-10g}{n}+O_g\left(\frac1{n^2}\right)}
=1+\frac{2g}{n}+O_g\left(\frac1{n^2}\right).
\end{multline*}
On the other hand,
\[
\frac{(n-2)^g}{n^g}
=\left(1-\frac2n\right)^g
=1-\frac{2g}{n}+O_g\left(\frac1{n^2}\right).
\]
The terms of order $1/n$ therefore cancel, and we obtain
\[
A_g(n)=1+O_g\left(\frac1{n^2}\right).
\]
Thus recurrence~\eqref{eq:W:g:n:recursion} satisfies the hypotheses of Lemma~\ref{lm:lemma:2:in:calculus}. It follows that there exist finite constants $C_g^{\mathit{even}}$ and $C_g^{\mathit{odd}}$ such that
\begin{align*}
W_g(2n)&=C_g^{\mathit{even}}+O_g\left(\frac1n\right),\\
W_g(2n+1)&=C_g^{\mathit{odd}}+O_g\left(\frac1n\right).
\end{align*}

The additive term in~\eqref{eq:W:g:n:recursion} is the sum of the contributions of all configurations other than the pocket, divided by the denominator in~\eqref{eq:EKZ:minimal:g}. These contributions are nonnegative, and the denominator is positive for all sufficiently large $n$. Hence the additive term is nonnegative. Since $W_g(n)>0$ by the geometric nature of the volume, the positivity assertion of Lemma~\ref{lm:lemma:2:in:calculus} shows that both limiting constants are strictly positive. This proves~\eqref{eq:W:minimal:genus:g}, together with the stated $O_g(n^{-1})$ error estimates, and completes the proof of Proposition~\ref{prop:minimal:any:g}.
\end{proof}

%--------------------------------------------------------------
\subsection{Non-minimal stratum in arbitrary genus}
\label{ss:Generic:stratum:in:arbitrary:genus}

\begin{proposition}
\label{prop:general:any:g}
Consider
the constants $C_g^{\mathit{even}}$ and $C_g^{\mathit{odd}}$ established in Proposition~\ref{prop:minimal:any:g}.
Consider a pair $(g, d)$ satisfying all of the following conditions
\[
g\ge 1,\qquad k\ge 1,\qquad  d_1\ge \dots \ge d_k\ge 1\,.
\]
Assume that relations~\eqref{eq:W:any:genus:g:limits} are valid for all collections $g',k',d'$ satisfying~\eqref{eq:boundary:strata:data} with $g'<g$, or $g'=g$ and $D'<D$. Then relations~\eqref{eq:W:any:genus:g:limits} are also valid for the pair $(g,d)$.
\end{proposition}

By the induction assumption, we have $C_g^{p(n)}\ne 0$,
where $p(n)\in\{\mathit{even, odd}\}$ is the parity of $n$,
so we can define
\begin{equation}
\label{eq:V:and:W}
V_g(n;d_1,\dots,d_k)
:=\frac{W_g(n;d_1,\dots,d_k)}{C_g^{p(n)}}.
\end{equation}
Proposition~\ref{prop:general:any:g} follows immediately from the following one
\begin{proposition}
\label{prop:reccurence:matches:lemma:from:calculus}
Under the assumptions above the variables $V_g(n;d)$
satisfy the following relation as $n\to+\infty$.
\begin{equation}
\label{eq:relation:for:V:to:fit:lemma:from:calculus}
V_g(n;d)
=A_{g,d}(n)\cdot V_g(n-2;d)
+\frac{D+k+\Delta}{n}
+O_{g,d}\left(\frac1{n^2}\right)\,,
\end{equation}
where
\begin{equation}
\label{eq:relation:for:V:to:fit:lemma:from:calculus:lambda}
A_{g,d}(n)=
\left(1-\frac{D+k+\Delta}{n}
+O_{g,d}\left(\frac1{n^2}\right)\right)\,.
\end{equation}
\end{proposition}
\begin{proof}
Under the assumptions of Proposition~\ref{prop:general:any:g},
Proposition~\ref{prop:negligible:a}
identifies dumbbell configurations considered in Proposition~\ref{prop:distinguished:dumbbell}
as the only configurations of type ``a'' carrying a non-negligible contribution to the sum in~\eqref{eq:EKZ:in:W:g}. The contribution
of the dumbbell configurations given by~\eqref{eq:dumbell:genus:g}
constitutes the fourth line in the expression~\eqref{eq:alltogether:renormalized:g} below.

Under the assumptions of Proposition~\ref{prop:general:any:g},
Proposition~\ref{prop:negligible:b}
claims that the total contribution of all configurations
of type ``b'' to the sum in~\eqref{eq:EKZ:in:W:g} is negligible.

Under the assumptions of Proposition~\ref{prop:general:any:g},
Proposition~\ref{prop:negligible:c}
identifies pocket configurations considered in Corollaries~\ref{cor:pocket:renormalized}
and~\ref{cor:pocket:d:renormalized} as the only configurations of type ``c'' carrying a non-negligible contribution to the sum in~\eqref{eq:EKZ:in:W:g}. The contribution
of the pocket configurations given by~\eqref{eq:Ca:total:renormalized:g}
and~\eqref{eq:Ca:total:renormalized:dj:g}
constitutes the second and the third lines respectively in the expression~\eqref{eq:alltogether:renormalized:g} below.
Thus~\eqref{eq:EKZ:in:W:g} takes the form of the first equality
in the expression below.
\begin{multline}
\label{eq:alltogether:renormalized:g}
W_g(n;d_1,\dots,d_k)=
\frac{1}{
n+D-k+5-6g+\Delta
+O_{g,d}\left(\frac1n\right)
}
\\
\times\bigg(
(n+1)\cdot
\left(\frac{n-2}{n}\right)^g\cdot
\frac{(n+D-4g+4)(n+D-4g+3)}{(n+D+k-2g+2)(n+D+k-2g+1)}
W_g(n-2;d)
\\
+\sum_{j=1}^k (d_j+1)\cdot
W_g(n;d_1,\dots,d_{j-1},d_j-2,d_{j+1},\dots,d_k)
\cdot\left(1+O_{g,d}\left(\frac1n\right)\right)
\\
+\sum_{j=1}^k
\frac{2}{d_j+2}\cdot
W_g(n-2;d_1,\dots,\widehat{d_j},\dots,d_k)
\left(1+O_{g,d}\left(\frac1n\right)\right)
+\mathcal{E}_{g,d}(n)
\bigg)
\\
=
\frac{1}{
n+D-k+5-6g+\Delta
+O_{g,d}\left(\frac1n\right)
}
\\
\times\bigg(
(n+1)\cdot
\left(\frac{n-2}{n}\right)^g\cdot
\frac{(n+D-4g+4)(n+D-4g+3)}{(n+D+k-2g+2)(n+D+k-2g+1)}
W_g(n-2;d)
\\
+(D+k+\Delta)\cdot C_g^{p(n)}
+O_{g,d}\left(\frac{1}{n}\right)
\bigg)\,,
\end{multline}

Let us justify the second equality. Before applying the induction hypothesis, we use Remarks~\ref{rm:entry:minus:1} and~\ref{rm:entry:minus:1:in:W}, namely, entries $-1$ are absorbed into the pole list, and entries $0$ are removed from the normalized variable using~\eqref{eq:W:of:0}.

For the pocket configurations with a zero of degree $d_j$ at the base of the cylinder, the sum of the reduced positive fixed degrees is $D-1$ when $d_j=1$ and $D-2$ when $d_j\geq2$. For the distinguished dumbbell terms, this sum is $D-d_j$. In each case it is strictly smaller than $D$. Thus the induction hypothesis gives
\begin{align*}
&W_g(n;d_1,\dots,d_{j-1},d_j-2,d_{j+1},\dots,d_k)
= C_g^{p(n)}+O_{g,d}\left(\frac{1}{n}\right)\,,
\\
&W_g(n-2;d_1,\dots,\widehat{d_j},\dots,d_k)
= C_g^{p(n)}+O_{g,d}\left(\frac{1}{n}\right)\,.
\end{align*}
By Propositions~\ref{prop:negligible:a}, \ref{prop:negligible:b}, \ref{prop:negligible:c} the induction hypothesis implies that the total contribution
$\mathcal{E}_{g,d}(n)$ to the above sum of all configurations
distinct from the pocket configurations and from the distinguished
dumbbell configurations already present in the formula is of order
$\mathcal{E}_{g,d}(n)=O\left(n^{-1}\right)$. Finally, we used
the relation
\[
\sum_{j=1}^k (d_j+1)
+\sum_{j=1}^k
\frac{2}{d_j+2}
=D+k+\Delta\,,
\]
which instantly follows from Definition~\eqref{eq:D:and:Delta}
of $D$ and $\Delta$.

Let us compute the first two terms of the expansion of the coefficient in front of $W_g(n-2;d_1,\dots,d_k)$ in negative powers of $n$. Set
\[
s:=D-k+5-6g+\Delta.
\]
Then
\[
\frac{n+1}{n+s+O_{g,d}\left(\frac1n\right)}
=
1+\frac{1-s}{n}
+O_{g,d}\left(\frac1{n^2}\right).
\]
Also,
\[
\left(\frac{n-2}{n}\right)^g
=
\left(1-\frac2n\right)^g
=
1-\frac{2g}{n}
+O_g\left(\frac1{n^2}\right),
\]
and
\begin{multline*}
\frac{(n+D-4g+4)(n+D-4g+3)}
{(n+D+k-2g+2)(n+D+k-2g+1)}
\\
=
1+\frac{4-4g-2k}{n}
+O_{g,d}\left(\frac1{n^2}\right).
\end{multline*}
Multiplying the three expansions, we find that the coefficient of $\frac1n$ is
\[
(1-s)-2g+(4-4g-2k).
\]
Since
\[
s=D-k+5-6g+\Delta,
\]
we get
\[
(1-s)-2g+(4-4g-2k)
=
-D-k-\Delta.
\]
Therefore
\begin{multline*}
\frac{1}{
n+D-k+5-6g+\Delta
+O_{g,d}\left(\frac1n\right)
}
\cdot
(n+1)
\left(\frac{n-2}{n}\right)^g
\\
\times
\frac{(n+D-4g+4)(n+D-4g+3)}
{(n+D+k-2g+2)(n+D+k-2g+1)}
\\
=
1-\frac{D+k+\Delta}{n}
+O_{g,d}\left(\frac1{n^2}\right).
\end{multline*}

We can rewrite Formula~\eqref{eq:alltogether:renormalized:g} as
\begin{multline*}
W_g(n;d_1,\dots,d_k)
=
\left(1-\frac{D+k+\Delta}{n}
+O_{g,d}\left(\frac1{n^2}\right)\right)
W_g(n-2;d_1,\dots,d_k)
\\
+
\frac{D+k+\Delta}{n}\cdot C_g^{p(n)}
+
O_{g,d}\left(\frac1{n^2}\right).
\end{multline*}
Dividing by $C_g^{p(n)}$
we get~\eqref{eq:relation:for:V:to:fit:lemma:from:calculus}
and~\eqref{eq:relation:for:V:to:fit:lemma:from:calculus:lambda}.
Proposition~\ref{prop:reccurence:matches:lemma:from:calculus}
is proved.
\end{proof}

\begin{proof}[Proof of Proposition~\ref{prop:general:any:g}]
Since the ambient stratum is non-minimal, we have
\[
c_d:=D+k+\Delta\ge 1+1+2\cdot\frac{1}{(1+2)}>2
\]
Let $a_{g,d}(n)=1-A_{g,d}(n)$, where $A_{g,d}(n)$ is given by~\eqref{eq:relation:for:V:to:fit:lemma:from:calculus:lambda}. Let $r_{g,d}(n)$ be the actual additive term in~\eqref{eq:relation:for:V:to:fit:lemma:from:calculus} divided by $a_{g,d}(n)$. Then
\[
a_{g,d}(n)=\frac{c_d}{n}+O_{g,d}(n^{-2}),\qquad
r_{g,d}(n)=1+O_{g,d}(n^{-1}).
\]
Since $p(n-2)=p(n)$ in~\eqref{eq:V:and:W}, we conclude that~\eqref{eq:relation:for:V:to:fit:lemma:from:calculus}
fits the recurrence formalized in Lemma~\ref{lm:lemma:in:calculus}. That lemma gives $V_g(n;d)=1+O_{g,d}(n^{-1})$, hence $W_g(n;d)=C_g^{p(n)}+O_{g,d}(n^{-1})$.
Proposition~\ref{prop:general:any:g} is proved.
\end{proof}

%#################################################################
\appendix
\section{Encoding principal boundary by decorated ribbon graphs}
\label{a:configurations}

To make the paper self-contained we recall briefly an encoding of configurations of saddle connections (or, equivalently, of the principal boundary strata) by decorated ribbon graphs introduced in~\cite{Masur:Zorich}.

% ---------------------------------------------------------------
\subsection{Graph of connected components}
%\label{ss:Graph:of:connected:components}

A  collection   $\gamma$   of   $\hat{h}$\hspace*{-1pt}omologous  saddle  connections
$\gamma=\{\gamma_1,\ldots,\gamma_n\}$  divides  $S$ into  simpler
surfaces   $S_j$   with  boundary.  We  associate  to  any   such
decomposition a  graph  $\Gamma(S,\gamma)$.  The  vertices of the
graph correspond to the connected components $S_j$ of $S\setminus
(\gamma_1\cup\dots\cup\gamma_n)$.   We    denote   the   vertices
corresponding to cylinders (if any) by small circles ``$\circ$''.
The remaining  vertices are labelled with  a ``$+$'' sign  if the
corresponding surface $S_j$ has trivial linear  holonomy and with
a  ``$-$''  sign  if it does not.

The   union   $\gamma=\gamma_1\cup\dots\cup\gamma_n$  of   saddle
connections can be considered  as  a graph $\gamma$ embedded into
the surface  $S$.  By  definition  $\Gamma(S,\gamma)$  is a graph
dual  to  $\gamma$.

\begin{figure}[htb]
\includegraphics[scale=0.8]{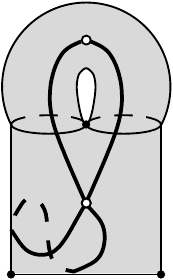}
\hspace*{80pt}
\includegraphics[scale=0.8,angle=90]{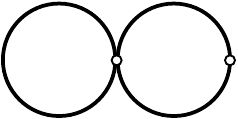}
\begin{picture}(0,0)(304,-111)%(0,0)(197,0)
\put(88,-35){$S_2$}
\put(90,-80){$S_1$}
\put(131,-14){$v_2$}
\put(138,-84){$v_1$}
\put(113,-61){$\gamma_2$}
\put(147,-61){$\gamma_3$}
\put(130,-118){$\gamma_1$}
\end{picture}
\begin{picture}(0,0)(163,-111) %(66,0)
\put(131,-16){$v_2$}
\put(131,-75){$v_1$}
\put(100,-45){$\gamma_2$}
\put(159,-45){$\gamma_3$}
\put(131,-118){$\gamma_1$}
\end{picture}
\vspace{10pt} % {120bp}
   %        %
\caption{
Example of a stable ribbon graph
as in~\cite[Fig.~2]{Masur:Zorich}
}
\label{fig:example:of:a:graph}
\end{figure}

\begin{example}
\label{ex:graph:and:dual:graph}
Consider  the  surface   $S$   and  the  collection  $\gamma$  of $\hat{h}$\hspace*{-1pt}omologous    saddle    connections   $\{\gamma_1,    \gamma_2, \gamma_3\}$  as  in   figure~\ref{fig:example:of:a:graph}.   The complement $S\setminus\gamma$ has two connected components; both are represented by cylinders.  Thus, the  graph  $\Gamma(S,\gamma)$ dual to the graph $\gamma\subset S$ contains  two  vertices, both of the ``$\circ$''-type, and  three edges, see Figure~\ref{fig:example:of:a:graph}.
\end{example}

%------------------------------------------
\subsection{List of stable graphs}
%\label{ss:List:of:stable:graphs}

The Theorem~\cite[Theorem~2]{Masur:Zorich} below classifies all possible graphs corresponding to configurations of $\hat{h}$\hspace*{-1pt}omologous saddle connections.

\begin{figure}[htb]
\centering
\includegraphics[scale=1]{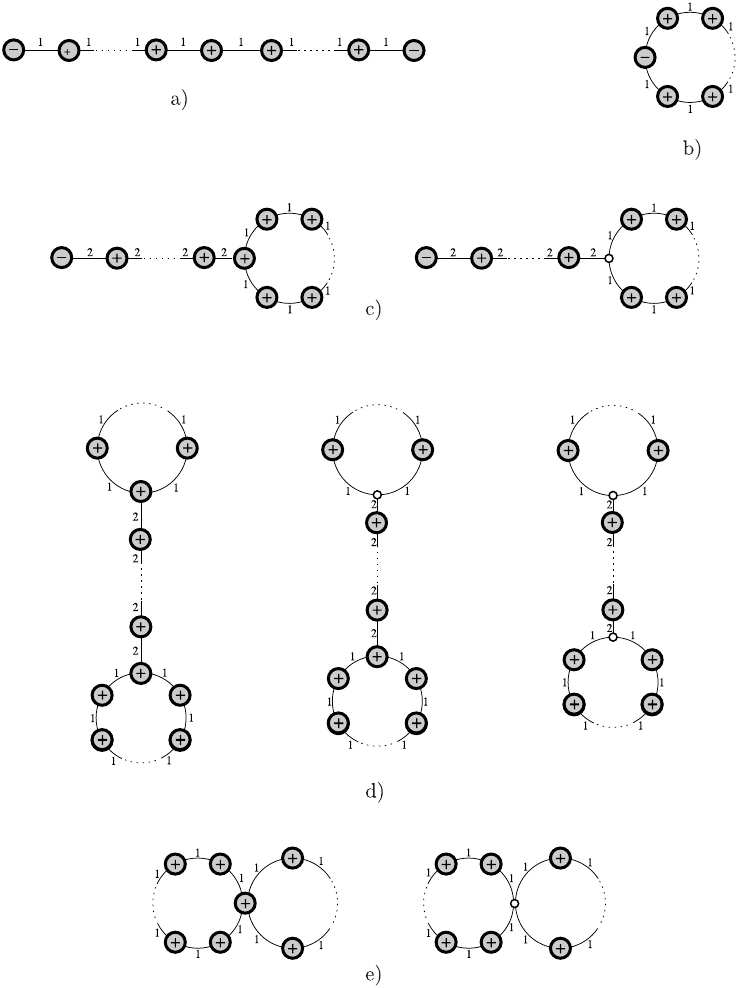}
\vspace*{-25pt}
\caption{
Classification of stable graphs
as in~\cite[Figure~3]{Masur:Zorich}.
}
\label{fig:graphs}
\end{figure}

\begin{NNTheorem}[H.~Masur, A.~Zorich]
Let  $S$  be  a  flat  surface  corresponding  to  a  meromorphic
quadratic  differential  $q$  with  at  most  simple  poles;  let
$\gamma$ be  a  collection  of  $\hat{h}$\hspace*{-1pt}omologous  saddle  connections
$\{\gamma_1, \dots,  \gamma_n\}$,  and  let $\Gamma(S,\gamma)$ be
the graph  of  connected  components  encoding  the decomposition
$S\setminus (\gamma_1\cup\dots\cup\gamma_n)$.

The graph $\Gamma(S,\gamma)$ either has  one  of  the basic types
listed below or  can  be  obtained from one of  these  graphs  by
placing  additional   \cv-vertices   of   valence   two   at  any
subcollection of edges subject to the  following restrictions. At
most one \cv-vertex may be placed at the same edge;  a \cv-vertex
cannot be  placed at an edge adjacent to  a \cv-vertex of valence
$3$ if this edge is separating.
\end{NNTheorem}

% -----------------------------------------------------------------------------
\subsection{Ribbon graph}
%\label{ss:Ribbon:graph}

It  is  convenient  to  consider a closed surface  with  boundary
$S^{comp}_j$  canonically  associated  to  $S_j$  by  taking  the
natural \textit{compactification} of $S_j$. Note that $S^{comp}_j$
need  not  be  the same as the \textit{closure} of $S_j$ in $S$. For
example, if we cut a surface $S$ along a single saddle connection
$\gamma$ joining a pair of distinct  singularities  we  obtain  a
surface $S_1$ whose  compactification  is a surface with boundary
composed of a  pair of parallel  distinct geodesics of  the  same
length, while  the  closure  of  $S_1=S\setminus\gamma_1$  in $S$
coincides with $S$. The closure of $S_j$ in $S$ is  obtained from
the compactification  $S^{comp}_j$  of $S_j$ by identification of
some boundary points (if necessary), or by identification of some
boundary saddle connections (if necessary).

\begin{figure}[hbt]
% An example of a local ribbon graph
%
%\centering
\hspace*{30pt}
%\mbox{
\includegraphics[scale=0.8]{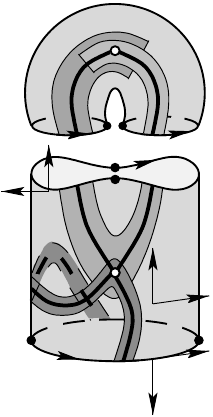}
\includegraphics[scale=0.9,angle=45]{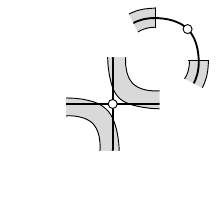}
\hspace*{-20pt}
\includegraphics[scale=0.9,angle=45]{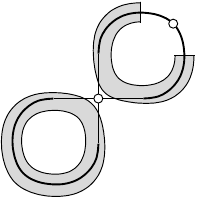}
%}
   %
\begin{picture}(0,0)(247,-167) % (250,0)
\put(80,-35){$S^{comp}_2$}
\put(80,-105){$S^{comp}_1$}
\put(133,-53){$\gamma_2$}
\put(159,-53){$\gamma_3$}
\put(138,-127){$\gamma''_1$}
\put(148,-143){$\gamma'_1$}
\put(118,-54){$1$}
\put(109,-67){$2$}
\put(166,-100){$2$}
\put(174,-107){$1$}
\put(166,-150){$1$}
\put(174,-141){$2$}
\end{picture}
\begin{picture}(0,0)(130,-165) % (130,0)
\put(133,-25){$\GG_{v_2}$}
\put(111,-58){$\gamma_2$}
\put(154,-58){$\gamma_3$}
\put(109,-110){$\gamma''_1$}
\put(154,-110){$\gamma'_1$}
\put(133,-120){$\GG_{v_1}$}
\end{picture}
\begin{picture}(0,0)(200,-155)% (200,0)
\put(303,-15){$\GG(S,\gamma)$}
\end{picture}
\vspace*{-20pt} % {150bp}
\caption{
Compactifications   $S_1^{comp},   S^{comp}_2$    of    connected
components  of  $S\setminus\gamma$,  the associated local  ribbon
graphs  $\GG_{v_1},   \GG_{v_2}$   and   the  global  ribbon  graph
$\GG(S,\gamma)$}
\label{fig:example:of:ribbon:graph}
\end{figure}

Suppose that the boundary of $S^{comp}_j$  has $r=r(j)$ connected
components (called \emph{boundary components}). Every
boundary  component  is  composed  of  a  closed chain of  saddle
connections $\gamma_{j_{i,1}}, \dots, \gamma_{j_{i,p(i)}}$, where
$1\le i \le  r$.  The case $p(i)=1$ is  not  excluded. The
canonical orientation of $S^{comp}_j$ determines the  orientation
of every boundary component $\cB_i$ of  $\partial S^{comp}_j$ and
hence determines a cyclic order
\begin{equation}
\label{eq:chain:i}
\to\gamma_{j_{i,1}}\to  \dots \to\gamma_{j_{i,p(i)}}\to
\end{equation}
on    every     such     chain;     by    convention    we    let
$j_{i,p(i)+1}:=j_{i,1}$. We get a natural decomposition of
the  set of  edges  of $\Gamma_{v_j}$ into  a  disjoint union  of
subsets, each endowed with a cyclic order.
It  is  convenient to encode such a combinatorial structure by  a
\emph{local ribbon  graph}. Local ribbon graphs near each vertex
naturally induce a global structure of a ribbon graph on the
graph $\Gamma(S,\gamma)$.

When  the   boundary   $\partial   S_j^{comp}$  contains  several
connected  components,   the   ribbon   graphs  corresponding  to
different components overlap at $v_j$ (as in the  left picture of
figure~\ref{fig:example:of:ribbon:graph}). However, it is easy to
make  them  disjoint  by  a  small  deformation,  subject  to  an
appropriate     choice      of     the     initial      embedding
$\Gamma_{v_j}\hookrightarrow S^{comp}_j$.

\begin{example}
%\label{ex:ribbon:graph:Gamma:for:fig:1}
   %
Consider once again the  surface  $S$ and the collection $\gamma$
of saddle  connections   $\{\gamma_1,  \gamma_2,
\gamma_3\}$  as  in  Example~\ref{ex:graph:and:dual:graph},
see Figure~\ref{fig:example:of:a:graph}.

The complement $S\setminus\gamma$ has  two  connected components;
their compactifications $S_1^{comp},  S_2^{comp}$ are represented
by a pair of flat cylinders. Each of the two connected components
of  the   boundary   of   $S_2^{comp}$   (the   top  cylinder  in
figure~\ref{fig:example:of:ribbon:graph}) is  formed by a  single
saddle      connection,       so      we      get       $\partial
S_2^{comp}=\{\gamma_2\}\sqcup\{\gamma_3\}$.  Each  of   the   two
connected components of the boundary of  $S_1^{comp}$ (the bottom
cylinder in figure~\ref{fig:example:of:ribbon:graph})  is  formed
by  a  pair  of  saddle  connections,   so   we   get   $\partial
S_1^{comp}=\{\gamma_2\to           \gamma_3\}\sqcup\{\gamma'_1\to
\gamma''_1\}$. The orientation of the boundary components induced
by the canonical  orientation  of $S$  is  indicated in the  left
picture.

The        picture         in        the        center         of
Figure~\ref{fig:example:of:ribbon:graph} shows the  corresponding
local ribbon graphs and the picture on the right shows the global
ribbon graph $\GG(S,\gamma)$ for this example.
\end{example}

\begin{figure}[htb]
\centering
\includegraphics[scale=1]{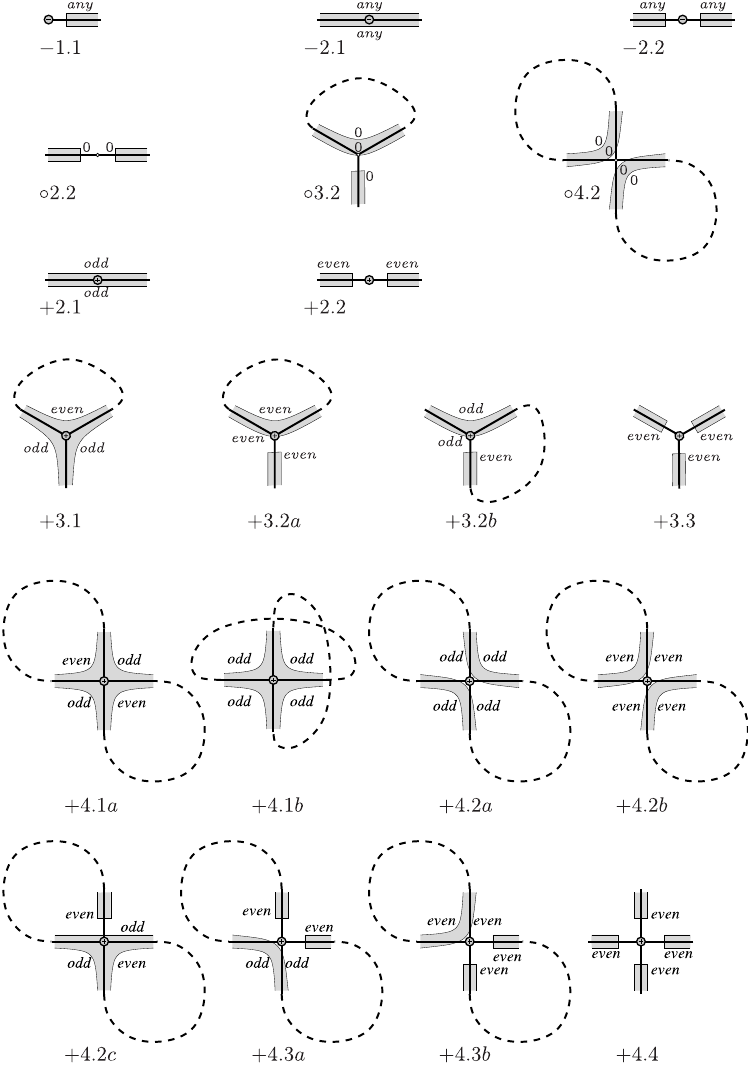}
\caption{
Classification of embedded local ribbon graphs
as in~\cite[Figure~6]{Masur:Zorich}.
}
\label{fig:embedded:local:ribbon:graphs}
\end{figure}

% -----------------------------------------------------------------------------
\subsection{Classification of local ribbon graphs}
Let  $S_j$  be   a   connected  component  of  the  decomposition
$S\setminus (\gamma_1\cup\dots\cup\gamma_n)$; let $S_j^{comp}$ be
its  compactification,  and let a connected component $\cB_i$  of
the  boundary   $\partial   S_j^{comp}$   be   represented  by  a
chain~\eqref{eq:chain:i}  of   saddle  connections.  The   common
endpoint of  $\gamma_{j_i}$  and $\gamma_{j_{i+1}}$ is called the
{\it boundary  singularity}  of  $S_j^{comp}$.  Since  all saddle
connections  $\gamma_1,  \dots,  \gamma_n$   are   parallel,  the
corresponding angle between $\gamma_{j_i}$ and $\gamma_{j_{i+1}}$
is an integer  multiple  of $\pi$. There  might also be  several
conical singularities in  the  interior of $S_j^{comp}$.
They do not belong to any saddle connection $\gamma_i$
of the configuration under consideration and are
called {\it interior singularities}.

\begin{definition}
%\label{def:order:of:boundary:singularity}
   %
If the  total angle at a  boundary singularity is  $(k+1)\pi$ the
{\it order of the boundary singularity} is defined to be $k$, and
the {\it parity of the boundary singularity} is defined to be the
parity of  $k$.

If the total angle
at an interior singularity is $(d + 2)\pi$ the order of the interior singularity is defined to be $d$.
\end{definition}

The following statement~\cite[Theorem~3]{Masur:Zorich} classifies the decorated ribbon graphs representing generic configurations of saddle connections for strata of meromorphic quadratic differentials with at most simple poles.

\begin{NNTheorem}[H.~Masur, A.~Zorich]
Consider  a  decomposition  of  a  flat   surface   $S$   as   in
Figure~\ref{fig:graphs}. Every connected component  $S_j$  of the
decomposition  has  one  of  the  boundary   types  presented  in
figure~\ref{fig:embedded:local:ribbon:graphs}  and  all indicated
boundary types are realizable.
\end{NNTheorem}

%-----------------------------------------------------------------------------
\subsection{Classification of configurations}

We reproduce now a formal definition of a \textit{configuration} and the classification theorem of admissible configurations; see Definition~3 and the Main Theorem
in Section~1.4 of the original paper~\cite{Masur:Zorich}.

\begin{definition}
\label{def:configuration}
The    following    combinatorial    structure   is   called    a
\emph{configuration of \^homologous saddle connections}.
\begin{enumerate}
\item
A finite  graph $\Gamma$ endowed  with a labelling of each
vertex by one  of the symbols ``$+$'', ``$-$'', or ``$\circ$'',
of one of the types  described  in  Figure~\ref{fig:graphs}.
\item
For any vertex $v$ of the graph $\Gamma$ an embedded ribbon graph $\GG_v$ (encoding the decomposition of $\Gamma_v$  into a disjoint union of subsets, called \emph{boundary components}, each endowed with a cyclic order), of one of the types described in Figure~\ref{fig:embedded:local:ribbon:graphs}.
\item
For every  ``$+$''-vertex $v$ of $\Gamma$  and for every  pair of
consecutive    elements    $\gamma_{i,l}\to\gamma_{i,l+1}$     of
$\GG_v$ (called  \emph{boundary singularities}) an  associated
parity         (even        or         odd)         as         in
figure~\ref{fig:embedded:local:ribbon:graphs}.
\item
For  every  vertex   $v$  of  $\Gamma$  and  for  every  boundary
singularity of  $\GG_v$  a nonnegative integer $k_{i,l}$ (referred
to  as  the \emph{order of the boundary singularity})  satisfying
the following conditions.  The  order of the boundary singularity
respects the \emph{parity} associated to the corresponding boundary
singularity as in (3) when $v$  is of the  $\oplus$-type; the order  of  any
boundary singularity of any vertex  of  the  ``$\circ$''-type  is
equal to zero. The sum $\delta_i+2=k_{i,1}+\dots+k_{i,p(i)}$ of orders
of boundary singularities along any boundary component $\cB_i$ of
$v$  satisfies  $ \delta_i\ge 0$ for  a  vertex of $\oplus$-type and  $
\delta_i\ge -1$ for a vertex of $\ominus$-type.
\item
For every vertex $v$ of $\Gamma$  of  $\ominus$-type  an  unordered
(possibly   empty)   collection  of   integers   $\{d_1,   \dots,
d_{s(v)}\}$, where  $d_l\in\{-1,1,2,3,\dots\}$; for every  vertex
$v$ of $\Gamma$ of $\oplus$-type  an  unordered  (possibly  empty)
collection of  positive even integers $\{d_1, \dots, d_{s(v)}\}$,
where $d_l\in\{2,4,\dots\}$.  In  both cases these collections of
integers (called \emph{orders of interior singularities}) satisfy
the following  compatibility  conditions  with  the collection of
boundary singularities of $\GG_v$:
$$
-4\ \le\ \big( \sum d_l + \sum \delta_i \big) \ \equiv\ 0\mod 4\,,
$$
where
the  first sum is  taken over all interior singularities and
the second sum is taken over all boundary components $\cB_i$
of $\GG_v$.
\item
When the vertex $v$ is of the $\ominus$-type  the couple [unordered
collection  of  interior  singularities, unordered collection  of
boundary singularities] is in  addition  not allowed to belong to
the following exceptional list:

$$
\begin{array}{|c|c|}
\hline
%-----------------------------------
-1.1&
[\emptyset,\{2\}];\qquad\quad
[\{-1\},\{3\}];\quad
[\{1\},\{1\}];\quad
[\{-1,1\},\{2\}]\\
\phantom{\quad\text{Type}\quad}
&
[\{1\},\{5\}];\quad
[\{3\},\{3\}];\quad
[\{1,3\},\{2\}];\qquad\quad
[\emptyset,\{6\}];\quad
[\{4\},\{2\}]
\\
\hline % -----------------------------------
-2.1&
[\emptyset,\{2, 0\}];\quad [\emptyset, \{1, 1\}];\qquad\quad
[\{-1\},\{0,3\}];\quad          [\{-1\},\{1,2\}]\\
   &
[\{1\},\{0,1\}];\quad
[\{1,-1\},\{0,2\}];\quad  [\{1,-1\},\{1,1\}]\\
&
[\{3,1\},\{2,0\}];\   [\{3,1\},\{1,1\}];\  [\{3\},\{3,0\}];\    [\{3\},\{2,1\}]\\
&
[\{1\},\{5,0\}];\      [\{1\},\{4,1\}];\    [\{1\},\{3,2\}];\
[\{4\},\{2,0\}];\           [\{4\},\{1,1\}]\\
&
[\emptyset,\{6,0\}];\       [\emptyset,\{5,1\}];\
[\emptyset,\{4,2\}];\       [\emptyset,\{3,3\}]\\
\hline % -----------------------------------
-2.2
&
[\emptyset,\{2,2\}];\
[\emptyset,\{1,3\}];\
[\{-1\},\{2,3\}];\
[\{1\},\{1,2\}];\
[\{-1,1\},\{2,2\}]
\\
&
[\emptyset,\{3,5\}];\quad
[\{1\},\{2,5\}];\quad
[\{3\},\{2,3\}];\quad
[\{1,3\},\{2,2\}]\\
&
[\emptyset,\{2,6\}];\quad
[\{4\},\{2,2\}]
\\
\hline
\end{array}
$$
\end{enumerate}
\end{definition}

The following statement is the Main Theorem in Section~1.4 of~\cite{Masur:Zorich}. It classifies the decorated ribbon graphs representing generic configurations of saddle connections for strata of meromorphic quadratic differentials with at most simple poles.

\begin{NNTheorem}[H.~Masur, A.~Zorich]
Any collection of \^homologous saddle connections in general position on a flat surface in a stratum of meromorphic quadratic differentials with at most simple poles satisfies the formal Definition~\ref{def:configuration} and any data described in the formal Definition~\ref{def:configuration} can be realized by an actual collection of \^homologous saddle connections on a flat surface in the corresponding stratum of quadratic differentials.
\end{NNTheorem}

% ----------------------------------------------------------

\end{document}